\documentclass[10pt]{article}
\usepackage[a4paper, top = 1.25in, bottom = 1.5in]{geometry}
\RequirePackage[hidelinks]{hyperref}
\usepackage{tocloft}
\RequirePackage{amsthm} %pour \theoremstyle

\theoremstyle{definition} 
\newtheorem{definition}{Definition}[section]

\theoremstyle{plain} 
\newtheorem{theorem}{Theorem}[section] 
\newtheorem{proposition}[theorem]{Proposition} %numérotation indexée sur theorem
\newtheorem{corollary}[theorem]{Corollary} 
\newtheorem{lemma}[theorem]{Lemma} 

\theoremstyle{definition} 
\newtheorem{example}[theorem]{Example}

\theoremstyle{remark}

\RequirePackage{graphicx}
\RequirePackage{rotating} %pour vertical isomorphism in tikzcd
\RequirePackage{amsmath}
\RequirePackage{xurl}
\RequirePackage{natbib}
\RequirePackage{graphicx} % pour insérer des figures
\RequirePackage{float} % Allows putting an [H] in \begin{figure} to specify the exact location of the figure
\RequirePackage{amssymb}
\RequirePackage{tikz-cd} % diagrammes commutatifs
\RequirePackage{stmaryrd} %\llbracket
\RequirePackage{faktor} %\faktor{n}{d}
\RequirePackage{array}%
\RequirePackage{bbm} % pour indicatrice \mathbbm{1}
\newcommand{\R}{\mathbb R}
\newcommand{\T}{\mathbb T}
\newcommand{\V}{\mathbb V}
\newcommand{\W}{\mathbb W}
\newcommand{\Z}{\mathbb Z}

\newcommand{\N}{\mathbb N}

\RequirePackage{pbox} %pour retour a la ligne dans tubular
\RequirePackage{mathtools} %pour underbracket
\RequirePackage{xr} %pour cross-referencing
\newcommand{\openball}[2] { \BB\left(#1,#2\right) } %boule ouverte
\newcommand{\closedball}[2] { \overline{\BB}\left(#1,#2\right) } %boule fermée
\newcommand{\closedballM}[3] { \overline{\BB}_{#3}\left(#1,#2\right) } %boule fermée geodesique
\newcommand{\reach}[1] { \mathrm{reach}(#1) } %reach
\newcommand{\Grass}[2] { \mathcal{G}_{#1}(#2) } %grassmanienne
\newcommand{\BB} { \mathcal{B} } %boule euclidienne
\newcommand{\CC} { \mathcal{C} } %classe C^2
\newcommand{\MM} { \mathcal{M} } %manifold
\newcommand{\MMo} { \mathcal{M}_0 } %manifold
\newcommand{\MMcheck} { \check \MM } %manifold
\newcommand{\imm} { u } %immersion
\newcommand{\immcheck} { \check u } %immersion
\newcommand{\matrixspace}[1] { M(#1) } %espace des matrices
\newcommand{\petito}[1] { o(#1) } %petit o
\newcommand{\grando}[1] { O(#1) } %grand o
\newcommand{\eucN}[1] { \left\|#1\right\| } %norme euclidienne
\newcommand{\frobN}[1] { \left\|#1\right\|_\mathrm{F} } %norme Frobenius
\newcommand{\gammaN}[1] { \left\|#1\right\|_\gamma } %norme gamma
\newcommand{\gammaNun}[1] { \left\|#1\right\|_1 } %norme gamma=1
\newcommand{\eucP}[2] { \left\langle #1, #2\right\rangle } %dot product euclidienne
\newcommand{\frobP}[2] { \langle #1, #2\rangle_\mathrm{F} } %dot product Frobenius
\newcommand{\transp}[1] { #1' } %transpose
\newcommand{\CechF}[1] { V[#1] }

\newcommand{\card}[1] { \mathrm{card}(#1) }
\newcommand{\Star}[1] { \mathrm{St}\left(#1\right) }
\newcommand{\closedStar}[1] { \overline{\mathrm{St}}\left(#1\right) }
\newcommand{\subdiv}[2] { \mathrm{sub}^{#2}\left(#1\right) }
\newcommand{\vor}[1] { \mathrm{Vor}(#1) } %Voronoi cell
\newcommand{\Hdist}[2] { \mathrm{d}_\mathrm{H}\left(#1, #2\right) } %Hausdorff distance
\newcommand{\NN} { \mathcal{N} } %nerve
\newcommand{\UU} { \mathcal{U} } %cover
\newcommand{\VV} { \mathcal{V} } %cover
\newcommand{\p} { \mathbbm{p} } %family of projections
\newcommand{\q} { \mathbbm{q} } %family of projections
\newcommand{\vbb} { \mathbbm{v} } %family of maps
\newcommand{\w} { \mathbbm{w} } %family of maps
\newcommand{\Zd} { \mathbb{Z}_2 } %group Z/2Z
\newcommand{\topreal}[1] { \left\vert #1 \right\vert } %Topological realization of a simplicial complex
\newcommand{\skeleton}[2] { {#1}^{#2} } %skeleton of a simplicial complex
\newcommand{\cupp} { \smile } %cup product
\newcommand{\cohomring}[1] { H^*(#1) } %cohomology ring
\newcommand{\bigcohomring}[1] { H^*\big(#1\big) }
\newcommand{\checkX} { \check{X} } %check X
\newcommand{\checkY} { \check{Y} } %check Y
\newcommand{\X} { \mathbb{X} } %filtration
\newcommand{\Y} { \mathbb{Y} } %filtration
\renewcommand{\S}{\ifmmode\operatorname{\mathbb{S}}\else\origS\fi} %S in math mode
\renewcommand{\P}{\ifmmode\operatorname{\mathbb{P}}\else\origS\fi} %P in math mode
\renewcommand{\SS} {\ifmmode\operatorname{\mathcal{S}}\else\origS\fi} 
\newcommand{\facemap} { \mathcal{F} } %face map
\newcommand{\facemapK}[1] { \mathcal{F}_{#1} } %face map
\RequirePackage{bbm}
\newcommand{\pSF}[2] { w_{#1}(#2) } %persistent SF class
\newcommand{\pSFt}[3] { w_{#1}^{#3}(#2) } %persistent SF class at t
\newcommand{\dist}[2] { \mathrm{dist}\left(#1, #2\right) } %distance to a subset
\newcommand{\distgamma}[2] { \mathrm{dist}_{\gamma}\left(#1, #2\right) } %distance to a subset, norm gamma
\newcommand{\proj}[2] { \mathrm{proj}\left(#1, #2\right) } %projection on subset
\newcommand{\projj}[2] { \mathrm{proj}_{#2}\left(#1 \right) } %projection on subset
\newcommand{\projmatrix}[1] { P_{#1} } %projection matrix
\newcommand{\med}[1] { \mathrm{med}\left( #1 \right) } %medial axis
\newcommand{\tmax}[1] { t^{\mathrm{max}}\left( #1 \right) } %maximal index for Cech bundle
\newcommand{\tmaxgamma}[1] { t_{\gamma}^{\mathrm{max}}\left( #1 \right) } %maximal index for Cech bundle
\newcommand{\tdeatho} { t^{\dagger} } %death index for Cech bundle, symbol
\newcommand{\diam}[1] { \mathrm{diam}\left( #1 \right) } %diameter
\newcommand{\coring}[1] { H^*\left( #1 \right) } %cohomology ring
\newcommand{\complementaire}[1] { {#1}^c } %complement, complementaire
\newcommand{\wfs}[1] { \mathrm{wfs}\left( #1 \right) } %cohomology ring
\newcommand{\barcode}[1] { \mathrm{Barcode}\left(#1\right) } %persistence barcode
\newcommand{\bdist}[2] { \mathrm{d}_\mathrm{b}\left(#1, #2\right) } %bottleneck distance
\newcommand{\idist}[2] { \mathrm{d}_\mathrm{i}\left(#1, #2\right) } %interleaving distance, 
\begin{document}
\sloppy %avoid text in margin
\title{Computing persistent Stiefel-Whitney classes \\ of line bundles
}


\author{Raphaël Tinarrage 

\vspace{.1cm}\\

\url{https://raphaeltinarrage.github.io/}
}

\date{}


%\institute{Raphaël Tinarrage \at
%              Inria Saclay Île-De-France, 1 Rue Honoré d'Estienne d'Orves, 91120 Palaiseau, France \\
%%              Tel.: +123-45-678910\\
%%              Fax: +123-45-678910\\
%              \email{raphael.tinarrage@inria.fr}           %  \\
%%             \emph{Present address:} of F. Author  %  if needed
%%           \and
%%           S. Author \at
%%              second address
%\and Raphaël Tinarrage \at
%              Institut de mathématiques d'Orsay, Bâtiment 307, Rue Michel Magat, Faculté des Sciences d’Orsay, Université Paris-Saclay, F-91405 Orsay Cedex, France \\
%%              Tel.: +123-45-678910\\
%%              Fax: +123-45-678910\\
%              \email{raphael.tinarrage@universite-paris-saclay.fr} 

\maketitle

\paragraph{Abstract.}
We propose a definition of persistent Stiefel-Whitney classes of vector bundle filtrations. It relies on seeing vector bundles as subsets of some Euclidean spaces. The usual \v{C}ech filtration of such a subset can be endowed with a vector bundle structure, that we call a \v{C}ech bundle filtration. We show that this construction is stable and consistent. When the dataset is a finite sample of a line bundle, we implement an effective algorithm to compute its first persistent Stiefel-Whitney class. In order to use simplicial approximation techniques in practice, we develop a notion of weak simplicial approximation. As a theoretical example, we give an in-depth study of the normal bundle of the circle, which reduces to understanding the persistent cohomology of the torus knot (1,2).
We illustrate our method on several datasets inspired by image analysis.
%\keywords{Persistent homology \and Vector bundles \and Stiefel-Whitney classes \and Simplicial approximation}

\setlength{\cftbeforesecskip}{4pt}
\tableofcontents

\section{Introduction}
\label{intro}

The inference of relevant topological properties of data represented as point clouds in Euclidean spaces is a central challenge in Topological Data Analysis (TDA).   
Given a (finite) set of points $X$ in $\R^n$, persistent homology provides a now classical and powerful tool to construct persistence diagrams whose points can be interpreted as homological features of $X$ at different scales.
%These persistence diagrams are obtained from {\em filtrations}, i.e. nested families of subspaces or simplicial complexes, built on top of $X$. 
%Among the many filtrations available to the user, unions of growing balls $\cup_{x \in X} B(x,t)$ (sublevel sets of distance functions), $t \in \R^+$, and their nerves, the \v{C}ech complex filtration, or its usually easier to compute version, the Vietoris-Rips filtration, are widely used. The main theoretical advantage of these filtrations is that they have been shown to produce persistence diagrams that are stable with respect to perturbations of $X$ in the Hausdorff metric \cite{Chazal_Geometriccomplexes}.

In this work, we aim at developing a similar theoretical framework for another topological invariant: the Stiefel-Whitney classes.
These classes, and more generally characteristic classes, are a powerful tool from algebraic topology, that contains additional information to the cohomology groups. 
For the Stiefel-Whitney classes to be defined, the input topological space has to be endowed with an additional structure: a real vector bundle.
They have been widely used in differential topology, for instance in the problem of deciding orientability of manifolds, of immersing manifolds in low-dimensional spaces, or in cobordism problems \citep{Milnor_Characteristic}.
Our work is motivated by introducing this tool to the TDA community. 

\paragraph{Previous work.}
To our knowledge, the problem of estimating Stiefel-Whitney classes from a point cloud observation has received little attention.
In the work of \citet{aubrey2011persistent}, one finds an algorithm to compute the Stiefel-Whitney classes in the particular case of the tangent bundle of a Euler mod-$2$ space (that is, a simplicial complex for which the link of each simplex has even Euler characteristic).
Close to the subject, \cite{perea2018multiscale} proposes a dimensionality reduction algorithm, based on the choice of a Stiefel-Whitney class, seen as a persistent cohomology class.

Recently, \cite{scoccola2021approximate} developed several notions of vector bundle adapted to finite simplicial complexes, one of which is used in this paper. They propose algorithms to compute the first two Stiefel-Whitney classes, which are conceptually different than the one presented here.

\paragraph{Our contributions.}
Just as persistent homology allows to extract homological features from filtrations of topological spaces, we propose a framework that allows to extract Stiefel-Whitney classes features from filtrations of vector bundles. It is briefly motivated here.

In general, if $X$ is a topological space endowed with a vector bundle $\xi$ of dimension $d$, there exists a collection of cohomology classes $w_1(\xi), ..., w_d(\xi)$, the Stiefel-Whitney classes, such that $w_i(\xi)$ is an element of the cohomology group $H^i(X)$ over $\Zd$ for $i \in \llbracket1, d\rrbracket$. 
In order to define Stiefel-Whitney classes in a persistent-theoretic framework, we will use a convenient definition of vector bundles:
defining a vector bundle over a compact space $X$ is equivalent (up to isomorphism of vector bundles) to defining a continuous map $p \colon X \rightarrow \Grass{d}{\R^m}$ for $m$ large enough, where $\Grass{d}{\R^m}$ is the Grassmann manifold of $d$-planes in $\R^m$. Such a map is called a \emph{classifying map} for $\xi$. 
%It is closely related to the Gauss map of submanifolds of $\R^3$, as explained in Figure \ref{intro:fig:Gauss_map}.
%
%\begin{figure}[H]
%\centering
%\includegraphics[width=1\linewidth]{Gauss_map.png}
%\caption{If $\MM$ is an orientable 2-submanifold of $\R^3$, the Gauss map $g\colon \MM \rightarrow \mathbb{S}_{2}$ maps every $x \in \MM$ to a normal vector of $\MM$ at $x$.
%By post-composing this map with the usual quotient map $\mathbb{S}_2 \rightarrow \Grass{1}{\R^3}$, we obtain a classifying map $f\colon \MM \rightarrow \Grass{1}{\R^3}$ for the normal bundle of $\MM$. }
%\label{intro:fig:Gauss_map}
%\end{figure}

Given a classifying map $p\colon X \rightarrow \Grass{d}{\R^m}$ of a vector bundle $\xi$, the Stiefel-Whitney classes $w_1(\xi), ..., w_d(\xi)$ can be defined by pushing forward some particular classes of the Grassmannian via the induced map in cohomology $p^*\colon H^*(X) \leftarrow H^*(\Grass{d}{\R^m})$.
If $w_i$ denotes the $i^\text{th}$ Stiefel-Whitney class of the Grassmannian, then the $i^\text{th}$ Stiefel-Whitney class of the vector bundle $\xi$ is
\begin{align}
\label{intro:eq:SFclassdef}
w_i(\xi) = p^*(w_i).
\end{align}

In order to translate these considerations in a persistent-theoretic setting, suppose that we are given a dataset of the form $(X, p)$, where $X$ is a finite subset of $\R^n$, and $p$ is a map $p\colon X \rightarrow \Grass{d}{\R^m}$.
Denote by $(X^t)_{t \geq 0}$ the \v{C}ech filtration of $X$, that is, the collection of the $t$-thickenings $X^t$ of $X$ in the ambient space $\R^n$. 
It is also known as the \emph{offset filtration} of $X$.
In order to define some \textit{persistent} Stiefel-Whitney classes, one would try to extend the map $p\colon X \rightarrow \Grass{d}{\R^m}$ to $p^t\colon X^t \rightarrow \Grass{d}{\R^m}$.
However, we did not find any interesting way to extend this map.
To overcome this issue, we propose to look at the dataset in a different way.
Transform the vector bundle $(X, p)$ into a subset of $\R^n \times \Grass{d}{\R^m}$ via 
\begin{align*}
\checkX = \left\{\left(x, p(x)\right), ~x \in X\right\}.
\end{align*}
The Grassmann manifold $\Grass{d}{\R^m}$ can be naturally embedded in $\matrixspace{\R^m}$, the space of $m\times m$ matrices. 
From this viewpoint, $\checkX$ can be seen as a subset of $\R^n \times \matrixspace{\R^m}$.
Let $(\checkX^t)_{t \geq 0}$ denotes the \v{C}ech filtration of $\checkX$ in the ambient space $\R^n \times \matrixspace{\R^m}$.
A natural map $p^t\colon \checkX^t \rightarrow \Grass{d}{\R^m}$ can be defined: map a point $(x, A) \in \checkX^t$ to the projection of $A$ on $\Grass{d}{\R^m}$, seen as a subset of $\matrixspace{\R^m}$:
\begin{align*}
p^t\colon (x,A) \in \R^n \times \matrixspace{\R^m} \longmapsto \proj{A}{\Grass{d}{\R^m}}.
\end{align*}
The projection is well-defined if $A$ does not belong to the medial axis of $\Grass{d}{\R^m}$. 
We show that this condition can be verified in practice (Lemma \ref{lem:projongrass}).
The \v{C}ech filtration of $\checkX$, endowed with the extended projection maps $(p^t\colon \checkX^t \rightarrow \Grass{d}{\R^m})_{t}$, is called the \emph{\v{C}ech bundle filtration}. 
%Now, using the extended maps $p^t\colon \checkX^t \rightarrow \Grass{d}{\R^m}$, 
Now we can define the $i^{\text{th}}$ persistent Stiefel-Whitney class as the collection of classes $w_i(X)=(w_i^t(X))_{t}$, where $w_i^t(X)$ is the push-forward 
\begin{align*}
w_i^t(X) = (p^t)^*(w_i),
\end{align*}
and where $w_i$ is the $i^{\text{th}}$ Stiefel-Whitney class of the Grassmann manifold (compare with Equation \eqref{intro:eq:SFclassdef}).
We summarize the information given by a persistent Stiefel-Whitney class in a diagram, that we call a \emph{lifebar}.
 
The construction we propose is defined for any subset of $\R^n \times \matrixspace{\R^m}$. 
We prove that this construction is stable, a result reminiscent of the usual stability theorem of persistent homology (Corollary \ref{cor:stability}).
We also show that the persistent Stiefel-Whitney classes are consistent estimators of Stiefel-Whitney classes (Corollary \ref{cor:consistency_stability}).

%As an illustration, Figure \ref{fig:intro_samples} represents the lifebars of the first persistent Stiefel-Whitney classes of samples $X$ and $X'$ the normal bundles of the 2-torus and the Klein bottle.
%Only one of them appears filled. This is interpretated as an indication of the non-orientability of the Klein bottle.
%\begin{figure}[H]
%\begin{minipage}{.49\linewidth}
%\centering
%\includegraphics[width=0.5\linewidth]{intro_torus_sample.png}
%\includegraphics[width=.9\linewidth]{intro_bundle_torus_barcode_lifespan.png}
%\end{minipage}
%\begin{minipage}{.49\linewidth}
%\centering
%\includegraphics[width=0.5\linewidth]{intro_klein_sample.png}
%\includegraphics[width=.9\linewidth]{intro_bundle_klein_barcode_lifespan.png}
%\end{minipage}
%\caption{\textbf{Left:} the sample $X\R^3 \times \matrixspace{\R^3}$, seen in $\R^3$, and the lifebar of its first persistent Stiefel-Whitney class. \textbf{Right:} same for $Y$.}
%\label{fig:intro_samples}
%\end{figure}

Moreover, we propose a concrete algorithm to compute the persistent Stiefel-Whitney classes.
This algorithm is based on several ingredients, including the triangulation of projective spaces, and the simplicial approximation method.
The simplicial approximation, widely used in theory, can be applied only if the simplicial complex is refined enough, a property that is attested by the star condition.
%relies on the star condition, that attests whether the simplicial complex is fine enough to 
%Simplical approximation, although widely used in theory, relies on the star condition, that cannot be verified in practice. 
However, this condition cannot be verified in practice.
We circumvent this problem by introducing the \emph{weak star condition}, a variant that only depends on the combinatorial structure of the simplicial complex.
When the simplicial complex is fine enough, the star condition and the weak star condition turn out to be equivalent notions (Proposition \ref{prop:weak_star_strong_star}).

\paragraph{Numerical experiments.}
A Python notebook, containing a concise demonstration of our method, can be found at \url{https://github.com/raphaeltinarrage/PersistentCharacteristicClasses/blob/master/Demo.ipynb}.
Another notebook, containing experiments on datasets inspired by image analysis, can be found at \url{https://github.com/raphaeltinarrage/PersistentCharacteristicClasses/blob/master/Experiments.ipynb}.

\paragraph{Outline.}
The rest of the paper is as follows. Sect. \ref{background} gathers usual definitions related to vector bundles, Stiefel-Whitney classes, simplicial approximation and persistent cohomology.
The definitions of vector bundle filtrations and persistent Stiefel-Whitney classes are given in Sect. \ref{sec:persistentSWclasses}, where their stability and consistency properties are established.
In Sect. \ref{sec:computation}, we propose a sketch of algorithm to compute these classes, based on simplicial approximation techniques.
In Sect. \ref{sec:algorithm} we give a particular attention to some technical details needed to implement this algorithm.
In Sect. \ref{sec:experiments} we apply our algorithm on concrete datasets.
For the clarity of the paper, the proofs of some results have been postponed to the appendices.

\section{Background}
\label{background}
\subsection{Stiefel-Whitney classes}
\label{background:SFclasses}
In this subsection, we define vector bundles and Stiefel-Whitney classes.
The reader may refer to \cite{Milnor_Characteristic} for an extended presentation.
Let $X$ be a topological space and $d\geq1$ an integer.

\paragraph{Vector bundles.}
A \emph{vector bundle $\xi$} of dimension $d$ over $X$ consists of a topological space $A = A(\xi)$, the \emph{total space}, a continuous map $\pi = \pi(\xi) \colon A \rightarrow X$, the \emph{projection map}, and for every $x \in X$, a structure of $d$-dimensional vector space on the \emph{fiber} $\pi^{-1}(\{ x \})$.
Moreover, $\xi$ must satisfy the local triviality condition: for every $x \in X$, there exists a neighborhood $U \subseteq X$ of $x$ and a homeomorphism $h \colon U \times \R^d \rightarrow \pi^{-1}(U)$ such that for every $y \in U$, the map $z \mapsto h(y,z)$ defines an isomorphism between the vector spaces $\R^d$ and $\pi^{-1}(\{ y \})$.

%An \emph{isomorphism} between two vector bundles $\xi$ and $\eta$ with common base space $X$ is a homeomorphism $f\colon A(\xi) \rightarrow A(\eta)$ which sends each fiber $\pi(\xi)^{-1}(\{ x \})$ isomorphically into $\pi(\eta)^{-1}(\{ f(x) \})$. If such a isomorphism exists, we say that $\xi$ and $\eta$ are \emph{isomorphic}.
%The \emph{trivial bundle} of dimension $d$ over $X$, denoted $\epsilon$, is defined with the total space $A(\epsilon) = X \times \R^d$, with the projection map $\pi(\epsilon)$ being the projection on the first coordinate, and where each fiber is endowed with the usual vector space structure of $\R^d$. 
%A vector bundle $\xi$ over $X$ is said \emph{trivial} if it is isomorphic to $\epsilon$.

A \emph{bundle map} between two vector bundles $\xi$ and $\eta$ with base spaces $X$ and $Y$ is a continuous map $f\colon A(\xi) \rightarrow A(\eta)$ which sends each fiber $\pi(\xi)^{-1}(\{ x \})$ isomorphically into another fiber $\pi(\eta)^{-1}(\{ x' \})$. 
If such a map exists, there exist a unique map $\overline f$ which makes the following diagram commute:
\begin{center}
\begin{tikzcd}
A(\xi) \arrow[r, "f"] \arrow[d, "\pi(\xi)", swap]
&[1em] A(\eta) \arrow[d, "\pi(\eta)" ] \\
X \arrow[r, "\overline f", swap]
& Y
\end{tikzcd}
\end{center}
\noindent
%In this case, $\xi$ is isomorphic to the pullback bundle $\overline{f}^* \eta$ \cite[Lemma 3.1]{Milnor_Characteristic}. We say that the map $\overline f$ \emph{covers} $f$.
If $X=Y$, and if $f$ is a homeomorphism, we say that $f$ is an \emph{isomorphism} of vector bundles, and that $\xi$ and $\eta$ are \emph{isomorphic}.
%An \emph{isomorphism} between two vector bundles $\xi$ and $\eta$ with common base space $X$ is a homeomorphism $f\colon A(\xi) \rightarrow A(\eta)$ which sends each fiber $\pi(\xi)^{-1}(\{ x \})$ isomorphically into $\pi(\eta)^{-1}(\{ x' \})$. If such a isomorphism exists, we say that $\xi$ and $\eta$ are \emph{isomorphic}.
The \emph{trivial bundle} of dimension $d$ over $X$, denoted $\epsilon$, is defined with the total space $A(\epsilon) = X \times \R^d$, with the projection map $\pi(\epsilon)$ being the projection on the first coordinate, and where each fiber is endowed with the usual vector space structure of $\R^d$. 
A vector bundle $\xi$ over $X$ is said \emph{trivial} if it is isomorphic to $\epsilon$.

Let $m \geq 0$.
The Grassmann manifold $\Grass{d}{\R^m}$ is a set which consists of all $d$-dimensional linear subspaces of $\R^m$. It can be given a smooth manifold structure. When $d=1$, $\Grass{1}{\R^m}$ corresponds to the real projective space $\P_{m-1}(\R)$.
In order to avoid mentioning $m$, it is convenient to consider the infinite Grassmannian.
The infinite Grassmann manifold $\Grass{d}{\R^{\infty}}$ is the set of all $d$-dimensional linear subspaces of $\R^{\infty}$, where $\R^{\infty}$ is the vector space of sequences with a finite number of nonzero terms.

Let $X$ be a paracompact space.
There exists a correspondence between the vector bundles over $X$ (up to isomorphism) and the continuous maps $X \rightarrow \Grass{d}{\R^\infty}$ (up to homotopy). Such a map is called a \emph{classifying map}.
When $X$ is compact, there exist an integer $m \geq 1$ such that a classifying map factorizes through
\begin{center}
\begin{tikzcd}
X \arrow[r] 
& \Grass{d}{\R^{m}} \arrow[r, hook]
& \Grass{d}{\R^{\infty}}.
\end{tikzcd}
\end{center}
Consequently, in the rest of this paper, we shall consider that vector bundles are given as a continuous maps $X \rightarrow \Grass{d}{\R^{m}}$ or $X \rightarrow \Grass{d}{\R^{\infty}}$.
%Note that, from this viewpoint, we only know the vector bundle up to isomorphism, which will be enough for our purposes.

\paragraph{Axioms for Stiefel-Whitney classes.}
To each vector bundle $\xi$ over a paracompact base space $X$, one associates a sequence of cohomology classes 
\begin{align*}
w_i(\xi) \in H^i(X, \Z_2), ~~~~~i \in \N,
\end{align*}
called the \emph{Stiefel-Whitney classes of $\xi$}.
%We denote by $w(\xi) = w_0(\xi) + w_2(\xi) + \cdots \in H^*(X, \Z_2)$ the total Stiefel-Whitney class.
These classes satisfy:
\begin{itemize}
\item \textbf{Axiom 1:} $w_0$ is equal to $1 \in H^0(X, \Zd)$, and if $\xi$ is of dimension $d$, then $w_i(\xi) = 0$ for $i>d$. 
\item \textbf{Axiom 2:} if $f\colon \xi \rightarrow \eta$ is a bundle map, then $w_i(\xi) = \overline f^* w_i(\eta)$, where $\overline f^*$ is the map in cohomology induced by the underlying map $\overline f\colon X \rightarrow Y$ between base spaces.
\item \textbf{Axiom 3:} if $\xi, \eta$ are bundles over the same base space $X$, then for all $k \in \N$, $w_k(\xi \oplus \eta) = \sum_{i=0}^k w_i(\xi) \cupp w_{k-i}(\eta)$, where $\oplus$ denotes the Withney sum, and $\cupp$ denotes the cup product.
\item \textbf{Axiom 4:} if $\gamma_1^1$ denotes the tautological bundle of the projective line $\Grass{1}{\R^2}$, then $w_1(\gamma_1^1) \neq 0$.
\end{itemize}
If such classes exists, then one proves that they are unique.
A way to show that they actually exist relies on the cohomology of the Grassmannians.

\paragraph{Construction of the Stiefel-Whitney classes.}
The cohomology rings of the Grassmann manifolds admit a simple description: $H^*(\Grass{d}{\R^\infty}, \Z_2)$ is the free abelian ring generated by $d$ elements $w_1, ..., w_d$. As a graded algebra, the degree of these elements are $|w_1| = 1,..., |w_d| = d$.
Hence we can write 
\begin{align*}
H^*(\Grass{d}{\R^\infty}, \Z_2) \cong \Z_2[w_1, ..., w_d].
\end{align*}
%In particular, the infinite projective $\P_\infty = \Grass{1}{\R^\infty}$ space has cohomology $H^*(\P_\infty, \Z_2) = \Z_2[w_1]$, the polynomial ring.
The generators $w_1, ..., w_d$ can be seen as the Stiefel-Whitney classes of a particular vector bundle on $\Grass{d}{\R^\infty}$, called the \emph{tautological bundle}.
Now, for any vector bundle $\xi$, define
\begin{align*}
w_i(\xi) = f_\xi^* (w_i),
\end{align*}
where $f_\xi \colon X \rightarrow \Grass{d}{\R^\infty}$ is a classifying map for $\xi$, and $f_\xi^* \colon H^*(X )\leftarrow H^*(\Grass{d}{\R^\infty})$ the induced map in cohomology. This construction yields the Stiefel-Whitney classes.

\paragraph{Interpretation of the Stiefel-Whitney classes.}
The Stiefel-Whitney classes are invariants of isomorphism classes of vector bundles, and carry topological information. 
Their main interpretation is the following: the Stiefel-Whitney classes are obstructions to the existence of nowhere vanishing sections of vector bundles. Let us explain this result. A \emph{section} of a vector bundle $\pi(\xi)\colon A \rightarrow X$ is a continuous map $s\colon X \rightarrow A$ such that $s(x) \in \pi^{-1}(\{x\})$ for all $x \in X$. It is nowhere vanishing if $s(x) \neq 0$ for all $x \in X$, where $0$ denotes the origin of the vector space $\pi^{-1}(\{x\})$. Given $k$ sections $s_1, \dots, s_k$, we say that they are \emph{independent} if the family $\left(s_1(x), \dots, s_k(x)\right)$ is free for all $x \in X$.
Then the following result holds: if a vector bundle $\xi$ of dimension $d$ admits $k$ independent and nowhere vanishing sections, then the top $k$ Stiefel-Whitney classes $w_d(\xi), \dots, w_{d-k+1}(\xi)$ are zero.

Another property that we will use in this paper is the following: the first Stiefel-Whitney class detects orientability. 
More precisely, the first Stiefel-Whitney class $w_1(\xi)$ is zero if and only if the vector bundle $\xi$ is orientable.
In the same vein, if $X$ is a compact manifold and $\tau$ its tangent bundle, then the manifold $X$ is orientable if and only if $w_1(\tau) = 0$.

In this paper, we will particularly study \emph{line bundles}, that is, vector bundles of dimension $d=1$. As a consequence of being an obstruction to nowhere vanishing sections, a line bundle $\xi$ on any topological space $X$ is trivial if and only if $w_1(\xi) = 0$.
More generally, the first Stiefel-Whitney class establishes a bijection between the isomorphism classes of line bundles over $X$ and its first cohomology group $H^1(X)$ over $\Zd$.
As an example, the circle $\S_1$ has cohomology group $H^1(\S_1) = \Zd$, hence admits only two isomorphism classes of line bundles. 
As another example, the sphere $\S_2$ has trivial cohomology group $H^1(\S_2) = 0$, hence only admits trivial line bundles.

\subsection{Simplicial approximation}
\label{background:simplicial_approx}
We start by defining the simplicial complexes and their topology.
We then describe the technique of simplicial approximation, based on the book of \citet{Hatcher_Algebraic}.

\paragraph{Simplicial complexes.}
A \emph{simplicial complex} is a set $K$ such that there exists a set $V$, the set of \emph{vertices}, with $K$ a collection of finite and non-empty subsets of $V$, and such that $K$ satisfies the following condition: 
for every $\sigma \in K$ and every non-empty subset $\nu \subseteq \sigma$, $\nu$ is in $K$.
The elements of $K$ are called \emph{faces} or \emph{simplices} of the simplicial complex $K$.

For every simplex $\sigma \in K$, we define its dimension $\dim(\sigma) = \card{\sigma} -1$.
The \emph{dimension} of $K$, denoted $\dim (K)$, is the maximal dimension of its simplices.
For every $i\geq 0$, the \emph{$i$-skeleton} $\skeleton{K}{i}$ is defined as the subset of $K$ consisting of simplices of dimension at most $i$. Note that $\skeleton{K}{0}$ corresponds to the underlying vertex set $V$, and that $\skeleton{K}{1}$ is a graph. 
Given a graph $G$, the corresponding \emph{clique complex} is the simplicial complex whose simplices are the sets of vertices of the cliques of $G$. 
We say that a simplicial complex $K$ is a \emph{flag complex} if it is the clique complex of its 1-skeleton $\skeleton{K}{1}$.

Given a simplex $\sigma \in K$, its \emph{(open) star} $\Star{\sigma}$ is the set of all the simplices $\nu \in K$ that contain $\sigma$. The open star is not a simplicial complex in general.
We also define its \emph{closed star} $\closedStar{\sigma}$ as the smallest simplicial subcomplex of $K$ which contains $\Star{\sigma}$. 

%\begin{figure}[H]
%\begin{minipage}{.32\linewidth}
%\centering
%\includegraphics[width=.6\linewidth]{star_complex.png}
%\newline\noindent
%$K$
%\end{minipage}
%\begin{minipage}{.32\linewidth}
%\centering
%\includegraphics[width=.6\linewidth]{star_open_star.png}
%\newline\noindent
%$\Star{v}$ in red and pink
%\end{minipage}
%\begin{minipage}{.32\linewidth}
%\centering
%\includegraphics[width=.6\linewidth]{star_closed_star.png}
%\newline\noindent
%$\closedStar{v}$ in red and pink
%\end{minipage}
%\caption{Open and closed star of a vertex of $K$.}
%\end{figure}

\paragraph{Geometric realizations.}
For every $p\geq 0$, the \emph{standard $p$-simplex} $\Delta^p$ is the topological space defined as the convex hull of the canonical basis vectors $e_1, ..., e_{p+1}$ of $\R^{p+1}$, endowed with the subspace topology.
To each simplicial complex $K$ is attached a \emph{geometric realization}. It is a topological space, denoted $\topreal{K}$, obtained by gluing the simplices of $K$ together.
According to this construction, each simplex $\sigma \in K$ admits a geometric realization $\topreal{\sigma}$ which is a subset of $\topreal{K}$. 
%Each simplex $\topreal{\sigma}$ is homeomorphic to the interior of the standard simplex of dimension $\dim{\sigma}$ if $\dim{\sigma} \geq 1$, and $\topreal{\sigma}$ is a point of $\dim{\sigma} = 0$. 
%The subset $\topreal{\sigma}$ is to be seen as the open subset corresponding to the face $\sigma$. 
The following set is a partition of $\topreal{K}$:
\begin{equation*}
\left\{ \topreal{\sigma}, \sigma \in K \right\}.
\end{equation*}
This allows to define the \emph{face map} of $K$. It is the unique map $\facemapK{K} \colon \topreal{K} \rightarrow K$ that satisfies $x \in \topreal{\facemapK{K}(x)}$ for every $x \in \topreal{K}$.

If $\sigma$ is a face of $K$ of dimension at least 1, the subset $\topreal{\sigma}$ is canonically homeomorphic to the interior of the standard $p$-simplex $\Delta^p$, where $p = \dim(\sigma)$.
This allows to define on $\topreal{K}$ the barycentric coordinates: for every face $\sigma = [v_0, ..., v_p] \in K$, the points $x \in \topreal{\sigma}$ can be written as 
\begin{equation*}
x = \sum_{i=0}^p \lambda_i v_i
\end{equation*}
with $\lambda_0, ..., \lambda_p > 0$ and $\sum_{i=0}^p \lambda_i = 1$.

If $X$ is any a topological space, a \emph{triangulation} of $X$ consists of a simplicial complex $K$ together with a homeomorphism $h\colon X \rightarrow \topreal{K}$.

\paragraph{Simplicial approximation.}
A \emph{simplicial map} between simplicial complexes $K$ and $L$ is a map between geometric realizations $g \colon \topreal{K} \rightarrow \topreal{L}$ which sends each simplex of $K$ to a simplex of $L$ by a linear maps that sends vertices to vertices.
%which sends vertices on vertices and is linear on every simplices. 
In other words, for every $\sigma = [v_0, ..., v_p] \in K$, the subset $[g(v_0), ..., g(v_p)]$ is a simplex of $L$, and the map $g$ restricted to $\topreal{\sigma} \subset \topreal{K}$ can be written in barycentric coordinates as
\begin{equation}
\label{background:eq:simplicial_map_1}
\sum_{i=0}^p \lambda_i v_i ~ \longmapsto ~ \sum_{i=0}^p \lambda_i g(v_i).
\end{equation}
A simplicial map $g \colon \topreal{K} \rightarrow \topreal{L}$ is uniquely determined by its restriction to the vertex sets $g_{| \skeleton{K}{0}} \colon \skeleton{K}{0} \rightarrow \skeleton{L}{0}$.
Reciprocally, let $f \colon \skeleton{K}{0} \rightarrow \skeleton{L}{0}$ be a map between vertex sets which satisfies the following condition: 
\begin{equation}
\label{background:eq:simplicial_map_2}
\forall \sigma \in K, ~f(\sigma) \in L.
\end{equation}
Then $f$ induces a simplicial map via barycentric coordinates, denoted $|f| \colon \topreal{K} \rightarrow \topreal{L}$.
In the rest of this paper, a simplicial map shall either refer to a map $g \colon \topreal{K} \rightarrow \topreal{L}$ which satisfies Equation \eqref{background:eq:simplicial_map_1}, to a map $f\colon \skeleton{K}{0} \rightarrow \skeleton{L}{0}$ which satisfies Equation \eqref{background:eq:simplicial_map_2}, or to the induced map $f \colon K \rightarrow L$.

Let $g \colon \topreal{K} \rightarrow \topreal{L}$ be any continuous map. The problem of \emph{simplicial approximation} consists in finding a simplicial map $f \colon K \rightarrow L$ with geometric realization $\topreal{f} \colon \topreal{K} \rightarrow \topreal{L}$ homotopy equivalent to $g$.
A way to solve this problem is to consider the following property \cite[Proof of Theorem 2C.1]{Hatcher_Algebraic}: we say that the map $g$ satisfies the \emph{star condition} if for every vertex $v$ of $K$, there exists a vertex $w$ of $L$ such that 
\begin{align*}
g\left(\topreal{ \closedStar{v} }\right) \subseteq \topreal{ \Star{w} }.
\end{align*}
\noindent
If this is the case, let $f \colon \skeleton{K}{0} \rightarrow \skeleton{L}{0}$ be any map between vertex sets such that for every vertex $v$ of $K$, we have $g\left(\topreal{ \closedStar{v}} \right) \subseteq \topreal{ \Star{f(v)} }$.
Equivalently, $f$ satisfies %$g\left(\closedStar{v} \right) \subseteq \Star{f(v)}$.
\begin{equation*}
g\left(\closedStar{v} \right) \subseteq \Star{f(v)}.
\end{equation*}
Such a map is called a \emph{simplicial approximation to $g$}. One shows that it is a simplicial map, and that its geometric realization $\topreal{f}$ is homotopic to $g$.% \cite[Theorem 2C.1]{Hatcher_Algebraic}.
%
%\begin{figure}[H]
%\begin{minipage}{.24\linewidth}
%\centering
%\includegraphics[width=.9\linewidth]{simplicial_approx_K.png}
%\newline\noindent
%$K$
%\end{minipage}
%\begin{minipage}{.24\linewidth}
%\centering
%\includegraphics[width=.9\linewidth]{simplicial_approx_L.png}
%\newline\noindent
%$L$
%\end{minipage}
%\begin{minipage}{.24\linewidth}
%\centering
%\includegraphics[width=.9\linewidth]{simplicial_approx_g.png}
%\newline\noindent
%$g$
%\end{minipage}
%\begin{minipage}{.24\linewidth}
%\centering
%\includegraphics[width=.9\linewidth]{simplicial_approx_f.png}
%\newline\noindent
%$f$
%\end{minipage}
%\caption{The map $f\colon K \rightarrow L$ (in red) is a simplicial approximation to $g$.}
%\label{background:fig:simplicial_approximation}
%\end{figure}

In general, a map $g$ may not satisfy the star condition. However, there is always a way to subdivise the simplicial complex $K$ in order to obtain an induced map which does. % (see Theorem \ref{background:th:starcondition}).
We describe this construction in the following paragraph.

We point out that, for some authors, such as \cite{Munkres84}, the star condition is defined by the property 
$g\left(\topreal{ \Star{v} }\right) \subseteq \topreal{ \Star{w} }$.
%\begin{align*}
%g\left(\topreal{ \Star{v} }\right) \subseteq \topreal{ \Star{w} }.
%\end{align*}
The defintion we used above, although harder to satisfy than this one, will be enough for our purposes.

\paragraph{Barycentric subdivisions.}
Let $\Delta^p$ denote the standard $p$-simplex, with vertices denoted $v_0, ..., v_p$. 
The \emph{barycentric subdivision} of $\Delta^p$ consists in decomposing $\Delta^p$ into $(p+1)!$ simplices of dimension $p$. It is a simplicial complex, whose vertex set corresponds to the points $\sum_{i=0}^p \lambda_i v_i$ for which some $\lambda_i$ are zero and the other ones are equal.
Equivalently, one can see this new set of vertices as a the power set of the set of vertices of $\Delta^p$.

More generally, if $K$ is a simplicial complex, its barycentric subdivision $\text{sub}(K)$ is the simplicial complex obtained by subdivising each of its faces. 
The set of vertices of $\text{sub}(K)$ can be seen as a subset of the power set of the set of vertices of $K$.

If $g\colon \topreal{K} \rightarrow \topreal{L}$ is any map, there exists a canonical extended map $\topreal{\text{sub}(K)} \rightarrow \topreal{L}$, still denoted $g$.
The \emph{simplicial approximation theorem} %(\cite[Theorem 2C.1]{Hatcher_Algebraic}) 
states that for any two simplicial complexes $K, L$ with $K$ finite, and $g \colon \topreal{K} \rightarrow \topreal{L}$ any a continuous map, there exists $n \geq 0$ such that $g \colon \topreal{\subdiv{K}{n}} \rightarrow \topreal{L}$ satisfies the star condition.
As a consequence, such a map $g \colon \topreal{\subdiv{K}{n}} \rightarrow \topreal{L}$ admits a simplicial approximation. 

\subsection{Persistent cohomology}
In this subsection, we write down the definitions of persistence modules, and their associated pseudo-distances, in the context of cohomology.
Compared to the standard definitions of persistent homology, the arrows go backward.  
Let $T \subseteq [0, +\infty)$ be an interval that contains $0$, let $E$ be a Euclidean space, and $k$ a field.

\paragraph{Persistence modules.}
A {\em persistence module} over $T$ is a pair $(\V, \vbb)$ where $\V = (V^t)_{t\in T}$ is a family of $k$-vector spaces, and $\vbb = (v_s^t)_{s\leq t \in T}$ is a family of linear maps $v_s^t\colon V^s \leftarrow V^t$ such that:
\begin{itemize}
\itemsep0em
\item for every $t\in T$, $v_t^t\colon V^t \leftarrow V^t$ is the identity map,
\item for every $r, s,t\in T$ such that $r\leq s\leq t$, $v_r^s \circ v_s^t = v_r^t$.
\end{itemize}
When there is no risk of confusion, we may denote a persistence module by $\V$ instead of $(\V, \vbb)$.
Given $\epsilon \geq 0$, an {\em $\epsilon$-morphism} between two persistence modules $(\V, \vbb)$ and $(\W, \w)$ is a family of linear maps $(\phi_t\colon V^t \rightarrow W^{t-\epsilon})_{t \geq \epsilon}$ such that the following diagram commutes for every $\epsilon \leq s \leq t$:
\begin{center}
\begin{tikzcd}
V^s \arrow["\phi_s", d]  &\arrow[l, "v_s^t"] \arrow["\phi_t", d] V^t \\
W^{s-\epsilon}  &\arrow[l, "w_{s-\epsilon}^{t-\epsilon}"] W^{t-\epsilon}
\end{tikzcd}
\end{center}
If $\epsilon = 0$ and each $\phi_t$ is an isomorphism, the family $(\phi_t)_{t \in T}$ is an {\em isomorphism} of persistence modules.
An {\em $\epsilon$-interleaving} between two persistence modules $(\V, \vbb)$ and $(\W, \w)$ is a pair of $\epsilon$-morphisms $(\phi_t\colon V^t \rightarrow W^{t-\epsilon})_{t \geq \epsilon}$ and $(\psi_t\colon W^t \rightarrow V^{t-\epsilon})_{t \geq \epsilon}$ such that the following diagrams commute for every $t \geq 2 \epsilon$: 
\begin{center}
\begin{minipage}[t]{0.4\textwidth}
\centering
\begin{tikzcd}
V^{t-2\epsilon}   & & \arrow[ll, "v_{t-2\epsilon}^t", swap] V^{t} \arrow[dl, "\phi_{t}"] \\
& W^{t-\epsilon}  \arrow[ul, "\psi_{t-\epsilon}"] & 
\end{tikzcd}
\end{minipage}
\begin{minipage}[t]{0.4\textwidth}
\centering
\begin{tikzcd}
& V^{t-\epsilon} \arrow[dl, "\phi_{t-\epsilon}", swap]  & \\
W^{t-2\epsilon}   & & W^t \arrow[ll, "w_{t-2\epsilon}^t"] \arrow[ul, "\psi_{t}", swap] 
\end{tikzcd}
\end{minipage}
\end{center}
The interleaving pseudo-distance between $(\V, \vbb)$ and $(\W, \w)$ is defined as 
$$\idist{\V}{\W} = \inf \{\epsilon \geq 0, ~\V \text{ and } \W \text{ are } \epsilon \text{-interleaved}\}.$$

\paragraph{Persistence barcodes.}
A persistence module $(\V, \vbb)$ is said to be {\em pointwise finite-dimensional} if for every $t \in T$, $V^t$ is finite-dimensional. This implies that we can define a notion of persistence barcode \citep{botnan2020decomposition}. It comes from the algebraic decomposition of the persistence module into interval modules. Moreover, given two pointwise finite-dimensional persistence  modules $\V, \W$ with persistence barcodes $\barcode{\V}$ and $\barcode{\W}$, the so-called isometry theorem states that 
%$\bdist{\barcode{\V}}{\barcode{\W}} = \idist{\V}{\W}$ 
\begin{align*}
\bdist{\barcode{\V}}{\barcode{\W}} = \idist{\V}{\W},
\end{align*}
where $\idist{\cdot}{\cdot}$ denotes the interleaving distance between persistence modules, and $\bdist{\cdot}{\cdot}$ denotes the bottleneck distance between barcodes.

More generally, the persistence module $(\V, \vbb)$ is said to be {\em $q$-tame} if for every $s,t \in T$ such that $s < t$, the map $v_s^t$ is of finite rank. The $q$-tameness of a persistence module ensures that we can still define a notion of persistence barcode, even though the module may not be decomposable into interval modules. Moreover, the isometry theorem still holds \citep{Chazal_Persistencemodules}.

\paragraph{Filtrations of sets and simplicial complexes.}
A family of subsets $\X=(X^t)_{t \in T}$ of $E$ is a {\em filtration} if it is non-decreasing for the inclusion, i.e. for any $s, t \in T$, if $s \leq t$ then $X^s \subseteq X^t$. 
Given $\epsilon\geq 0$, two filtrations $\X=(X^t)_{t \in T}$ and $\Y=(Y^t)_{t \in T}$ of $E$ are {\em $\epsilon$-interleaved} if, for every $t \in T$, $X^t \subseteq Y^{t+\epsilon}$ and $Y^t \subseteq X^{t+\epsilon}$. 
The interleaving pseudo-distance between $\X$ and $\Y$ is defined as the infimum of such $\epsilon$:
\[\idist{\X}{\Y} = \inf \{ \epsilon, ~\X \text{  and  }  \Y  \text{  are  }\epsilon\text{-interleaved} \}.\]

Filtrations of simplicial complexes and their interleaving distance are similarly defined: 
given a simplicial complex $S$, a {\em filtration of $S$} is a non-decreasing family $\S = (S^t)_{t \in T}$ of subcomplexes of $S$. The interleaving pseudo-distance between two filtrations $(S_1^t)_{t \in T}$ and $(S_2^t)_{t \in T}$ of $S$ is the infimum of the $\epsilon \geq 0$ such that they are $\epsilon$-interleaved, i.e., for any $t \in T$, we have $S_1^{t} \subseteq S_2^{t+\epsilon}$ and  $S_2^{t} \subseteq S_1^{t+\epsilon}$.

\paragraph{Relation between filtrations and persistence  modules.}
Applying the singular cohomology functor to a set filtration gives rise to a persistence module whose linear maps between cohomology groups are induced by the inclusion maps between sets. As a consequence, if two filtrations are $\epsilon$-interleaved, then their associated persistence modules are also $\epsilon$-interleaved, the interleaving homomorphisms being induced by the interleaving inclusion maps. As a consequence of the isometry theorem, if the modules are $q$-tame, then the bottleneck distance between their persistence barcodes is upperbounded by $\epsilon$.
The same remarks hold when applying the simplicial cohomology functor to simplicial filtrations.

\subsection{Notations}
We adopt the following notations:

\begin{itemize}
\itemsep0.12em 
\item $I$ denotes a set, $\card{I}$ its cardinal and $\complementaire{I}$ its complement.
\item if $i$ and $j$ are intergers such that $i\leq j$, $\llbracket i,j \rrbracket$ denotes the set of integers between $i$ and $j$ included.
\item $\R^n$ and $\R^m$ denote the Euclidean spaces of dimension $n$ and $m$, $E$ denotes a Euclidean space.
\item $\matrixspace{\R^m}$ denotes the vector space of $m \times m$ matrices, $\Grass{d}{\R^m}$ the Grassmannian of $d$-subspaces of $\R^m$, and $\S_k \subset \R^{k+1}$ the unit $k$-sphere.
\item $\eucN{\cdot}$ denotes the usual Euclidean norm on $\R^n$, $\frobN{\cdot}$ the Frobenius norm on $\matrixspace{\R^m}$, $\gammaN{\cdot}$ the norm on $\R^n \times \matrixspace{\R^m}$ defined as $\gammaN{(x,A)}^2 = \eucN{x}^2 + \gamma^2 \frobN{A}^2$ where $\gamma>0$ is a parameter.
\item $\X = (X^t)_{t\in T}$ denotes a set filtration. 
$\V[\X]$ denotes the corresponding persistent cohomology module.
If $X$ is a subset of $E$, then $\X = (X^t)_{t\in T}$ denotes the \v{C}ech set filtration of $X$ (also called the offset filtration).
\item $(\V, \vbb)$ denotes a persistence module over $T$, with $\V = (V^t)_{t \in T}$ a family of vector spaces, and $\vbb = (v_s^t \colon X^s \leftarrow X^t)_{s\leq t \in T}$ a family of linear maps.
\item $\UU$ denotes a cover of a topological space, and $\NN(\UU)$ its nerve. $\S = (S^t)_{t \in T}$ denotes a simplicial filtration.
%\item $(\X, \p)$ denotes a vector bundle filtration, with $\X$ a set filtration, and $\p = (p^t)_{t \in T}$ a family of maps $p^t \colon X^t \rightarrow \Grass{d}{\R^m}$.
%If $X$ is a subset of $\R^n \times \matrixspace{\R^m}$, then $(\X, \p)$ denotes the \v{C}ech bundle filtration associated to $X$.
\item If $X$ is a topological space, $H^*(X)$ denotes its cohomology ring over $\Zd$ (the field with two elements), and $H^i(X)$ its $i^\text{th}$ cohomology group over $\Zd$. If $f\colon X\rightarrow Y$ is a continuous map, $f^*\colon H^*(X)\leftarrow H^*(Y)$ is the map induced in cohomology. 
\item If $\xi$ is a vector bundle, $w_i(\xi)$ denotes its $i^\text{th}$ Stiefel-Whitney class. %If $(\X, \p)$ is a vector bundle filtration, $w_i(\p)$ denotes the $i$th persistent Stiefel-Whitney class, with $w_i(\p) = (w_i^t(\p))_{t \in T}$ (see Definition \ref{def:persistent_SF_classes}).
\item If $A$ is a subset of $E$, then $\med{A}$ denotes its medial axis, $\reach{A}$ its reach and $\dist{\cdot}{A}$ the distance to $A$. % (see Subsection \ref{subsec:background_persistentcohomology}).
The projection on $A$ is denoted $\proj{\cdot}{A}$ or $\projj{\cdot}{A}$.
$\Hdist{\cdot}{\cdot}$ denotes the Hausdorff distance between two sets of $E$.
\item If $K$ is a simplicial complex, $\skeleton{K}{i}$ denotes its $i$-skeleton. For every vertex $v \in \skeleton{K}{0}$, $\Star{v}$ and $\closedStar{v}$ denote its open and closed star.
The geometric realization of $K$ is denoted $\topreal{K}$, and the geometric realization of a simplex $\sigma \in K$ is $\topreal{\sigma}$.
The face map is denoted $\facemapK{K} \colon \topreal{K} \rightarrow K$.% (see Subsection \ref{subsec:background_simplicialcomplexes}). 
\item If $f\colon K \rightarrow L$ is a simplicial map, $\topreal{f} \colon \topreal{K} \rightarrow \topreal{L}$ denotes its geometric realization. The $i^\text{th}$ barycentric subdivision of the simplicial complex $K$ is denoted $\subdiv{K}{i}$.% (see Subsection \ref{subsec:simplicial_approximation}).
\end{itemize}

\section{Persistent Stiefel-Whitney classes}
\label{sec:persistentSWclasses}
\subsection{Definition}
Let $E = \R^n$ be a Euclidean space, and $\X = (X^t)_{t\in T}$ a set filtration of $E$.
Let us denote by $i_s^t$ the inclusion map from $X^s$ to $X^t$.
In order to define persistent Stiefel-Whitney classes, we have to give such a filtration a vector bundle structure.
The infinite Grassmann manifold is denoted $\Grass{d}{\R^\infty}$.

\begin{definition}
A \emph{vector bundle filtration} of dimension $d$ on $E$ is a couple $(\X,\p)$ where $\X = (X^t)_{t \in T}$ is a set filtration of $E$ and $\p = (p^t)_{t \in T}$ a family of continuous maps $p^t \colon X^t \rightarrow \Grass{d}{\R^\infty}$ such that, for every $s,t \in T$ with $s \leq t$, we have $p^t \circ i_s^t = p^s$.
In other words, the following diagram commutes:
\begin{center}
\begin{tikzcd}%[row sep=tiny]
X^s \arrow[dr, "p^s", swap] \arrow[rr, "i_s^t"]&  & X^t \arrow[dl, "p^t"] \\
& \Grass{d}{\R^\infty} &
\end{tikzcd}
\end{center}
\end{definition}

Note that for any $m \in \N$, and by using the inclusion $\Grass{d}{\R^m} \hookrightarrow \Grass{d}{\R^\infty}$, one may define a vector bundle filtration by considering maps $p^t \colon X^t \rightarrow \Grass{d}{\R^m}$.

Following Subsect. \ref{background:SFclasses}, the induced map in cohomology, $(p^t)^*$, allows to define the Stiefel-Whitney classes of this vector bundle. 
Let us introduce some notations.
The Stiefel-Whitney classes of $\Grass{d}{\R^\infty}$ are denoted $w_1, \dots, w_d$.
The Stiefel-Whitney classes of the vector bundle $(X^t, p^t)$ are denoted $w_1^t(\p), \dots, w_d^t(\p)$, and can be defined as $w_i^t(\p) = (p^t)^*(w_i)$.

\begin{center}
\begin{tikzcd}[row sep=tiny]
  (p^t)^* \colon\bigcohomring{X^t} 
& \bigcohomring{\Grass{d}{\R^\infty}} \arrow[l] \\
~~~~~~~~ w_1^t(\p) & w_1 \arrow[l, mapsto] \\
%& w_2(p^t) & w_2 \arrow[l, mapsto] \\
 ~ \arrow[r,"\vdots", phantom] & ~ \\
~~~~~~~~ w_d^t(\p) & w_d \arrow[l, mapsto]
\end{tikzcd}
\end{center}
\noindent
Let $(\V, \vbb)$ denote the persistence module associated to the filtration $\X$, with $\V = (V^t)_{t \in T}$ and $\vbb = (v_s^t)_{s\leq t \in T}$. 
Explicitly, $V^t$ is the cohomology ring $\cohomring{X^t}$, and $v_s^t$ is the induced map $\cohomring{X^s} \leftarrow \cohomring{X^t}$.
For every $t \in T$, the classes $w_1^t(\p), \dots, w_d^t(\p)$ belong to the vector space $V^t$.
The persistent Stiefel-Whitney classes are defined to be the collection of such classes over $t$.

\begin{definition}
\label{def:persistent_SF_classes}
Let $(\X, \p)$ be a vector bundle filtration.
The \emph{persistent Stiefel-Whitney classes} of $(\X, \p)$ are the families of classes
\begin{align*}
w_1(\p) &= \big(w_1^t(\p)\big)_{t \in T} \\
&\vdots  \\
w_d(\p)& = \big(w_d^t(\p)\big)_{t \in T}.
\end{align*}
\end{definition}

Let $i \in \llbracket 1,d\rrbracket$, and consider a persistent Stiefel-Whitney class $w_i(\p)$.
Note that it satisfies the following property: for all $s,t \in T$ such that $s\leq t$, $w_i^s(\p) = v_s^t\big( w_i^t(\p) \big)$. 
As a consequence, if a class $w_i^t(\p)$ is given for a $t \in T$, one obtains all the others $w_i^s(\p)$, with $s \leq t$, by applying the maps $v_s^t$.
In particular, if $w_i^t(\p) = 0$, then $w_i^s(\p) = 0$ for all $s \in T$ such that $s \leq t$.

\paragraph{Lifebar.}
In order to visualize the evolution of a persistent Stiefel-Whitney class through the persistence module $(\V, \vbb)$, we propose the following bar representation: the lifebar of $w_i(\p)$ is the set 
\begin{align*}
\left\{ t \in T, w_i^t(\p) \neq 0 \right\}.
\end{align*}
According to the last paragraph, the lifebar of a persistent class is an interval of $T$, of the form $[\tdeatho, \sup(T))$ or $(\tdeatho, \sup(T))$, where 
\begin{align*}
\tdeatho = \inf \left\{t \in T, w_i^t(\p) \neq 0\right\},
\end{align*}
with the convention $\inf (\emptyset) = \inf (T)$.
In order to distinguish the lifebar of a persistent Stiefel-Whitney class from the bars of the persistence barcodes, we draw the rest of the interval hatched (see Figure \ref{fig:1}).
\begin{figure}[H]
\centering
\includegraphics[width=0.6\linewidth]{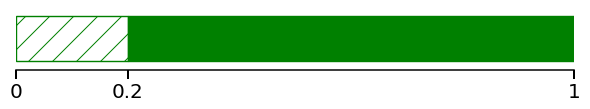}
\caption{Example of a lifebar of a persistent Stiefel-Whitney class with $t^\dagger = 0{.}2$ and $\max(T)=1$.}
\label{fig:1}
\end{figure}

\subsection{\v{C}ech bundle filtrations}
\label{subsec:filtered_cech_bundle}
In this subsection, we propose a particular construction of vector bundle filtration, called the \emph{\v{C}ech bundle filtration}.
We shall work in the ambient space $E = \R^n \times \matrixspace{\R^m}$.
Let $\eucN{\cdot}$ be the usual Euclidean norm on the space $\R^n$, and $\frobN{\cdot}$ the Frobenius norm on $\matrixspace{\R^m}$, the space of $m\times m$ matrices.
Let $\gamma > 0$. We endow the vector space $E$ with the Euclidean norm $\gammaN{\cdot}$ defined for every $(x,A) \in E$ as 
\begin{equation}
\gammaN{(x,A)}^2 = \eucN{x}^2 + \gamma^2 \frobN{A}^2.
\label{eq:gammaN}
\end{equation}
See Subsection \ref{subsec:choice_of_gamma} for a discussion about the parameter $\gamma$.

In order to define the \v{C}ech bundle filtration, we shall first study the usual embedding of the Grassmann manifold $\Grass{d}{\R^m}$ into the matrix space $\matrixspace{\R^m}$.

\paragraph{Embedding of $\Grass{d}{\R^m}$.}
We embed the Grassmannian $\Grass{d}{\R^m}$ into $\matrixspace{\R^m}$ via the application which sends a $d$-dimensional subspace $T \subset \R^m$ to its orthogonal projection matrix $\projmatrix{T}$.
We can now see $\Grass{d}{\R^m}$ as a submanifold of $\matrixspace{\R^m}$. 
Recall that $\matrixspace{\R^m}$ is endowed with the Frobenius norm. According to this metric, $\Grass{d}{\R^m}$ is included in the sphere of center 0 and radius $\sqrt{d}$ of $\matrixspace{\R^m}$.

In the metric space $(\matrixspace{\R^m}, \frobN{\cdot})$, consider the distance function to $\Grass{d}{\R^m}$, denoted $\dist{\cdot}{\Grass{d}{\R^m}}$.
Let $\med{\Grass{d}{\R^m}}$ denote the medial axis of $\Grass{d}{\R^m}$.
It consists in the points $A \in \matrixspace{\R^m}$ which admit at least two projections on $\Grass{d}{\R^m}$: 
\begin{align*}
\med{\Grass{d}{\R^m}} = \{A \in \matrixspace{\R^m}, \exists P, P' &\in \Grass{d}{\R^m}, P \neq P', \\
&\frobN{A - P} = \frobN{A - P} = \dist{A}{\Grass{d}{\R^m}} \}.
\end{align*}

%\begin{figure}[H]
%\centering
%\includegraphics[width=0.6\linewidth]{grassmannian.png}
%\caption[rep]{Representation of the Grassmannian $\Grass{1}{\R^2} \subset \matrixspace{\R^2} \simeq \R^4$. %, projected in a 2-dimensional subspace of $\R^4$. 
%It is equal to the circle of radius $\frac{\sqrt{2}}{2}$, in the 2-affine space generated by $ \begin{psmallmatrix} 1 & 0  \\ 0 & -1 \end{psmallmatrix}$ and $ \begin{psmallmatrix} 0 & 1  \\ 1 & 0 \end{psmallmatrix}$, and with origin $ \frac{1}{2} \begin{psmallmatrix} 1 & 0  \\ 0 & 1 \end{psmallmatrix}$.
%%It is diffeomorphic to the circle $\S_1$. 
%The matrix $ \frac{1}{2} \begin{psmallmatrix} 1 & 0  \\ 0 & 1 \end{psmallmatrix}$ is an element of $\med{\Grass{1}{\R^2}}$.}
%\end{figure}

\noindent
On the set $\matrixspace{\R^m} \setminus \med{\Grass{d}{\R^m}}$, the projection on $\Grass{d}{\R^m}$ is well-defined:
\begin{align*}
\proj{\cdot}{\Grass{d}{\R^m}} \colon \matrixspace{\R^m} \setminus \med{\Grass{d}{\R^m}} &\longrightarrow \Grass{d}{\R^m} \subset \matrixspace{\R^m} \\
A &\longmapsto P~ \text{ s.t. } \frobN{P-A} = \dist{A}{\Grass{d}{\R^m}}.
\end{align*}
The following lemma describes this projection explicitly.
%We defer its proof to Appendix \ref{sec:appendix_persistentSFclasses}.

\begin{lemma}
\label{lem:projongrass}
For any $A \in \matrixspace{\R^m}$, let $A^s$ denote the matrix $A^s = \frac{1}{2}(A + \transp{A})$, where $\transp{A}$ is the transpose of $A$, and let $\lambda_1(A^s), ..., \lambda_m(A^s)$ be the eigenvalues of $A^s$ in decreasing order.
The distance from $A$ to $\med{\Grass{d}{\R^m}}$ is
\begin{align*}
\dist{A}{\med{\Grass{d}{\R^m}}} = \frac{\sqrt{2}}{2} \big|\lambda_d(A^s) - \lambda_{d+1}(A^s)\big|.
\end{align*}
If this distance is positive, the projection of $A$ on $\Grass{d}{\R^m}$ can be described as follows: consider the symmetric matrix $A^s$, and let $A^s = O D \transp{O}$, with $O$ an orthogonal matrix, and $D$ the diagonal matrix containing the eigenvalues of $A^s$ in decreasing order. Let $J_d$ be the diagonal matrix whose first $d$ terms are 1, and the other ones are zero. We have 
\begin{align*}
\proj{A}{\Grass{d}{\R^m}} = O J_d \transp{O}.
\end{align*}
\end{lemma}

\begin{proof}%[Lemma \ref{lem:projongrass}]%[of Lemma \ref{lem:projongrass} page \pageref{lem:projongrass}]
\label{appendix:proof:projongrass}
Note that $\Grass{d}{\R^m}$ is contained in the linear subspace $\SS$ of symmetric matrices.
Therefore, to project a matrix $A \in \matrixspace{\R^m}$ onto $\Grass{d}{\R^m}$, we may project on $\SS$ first.
It is well known that the projection of $A$ onto $\SS$ is the matrix $A^s = \frac{1}{2}(A + \transp{A})$.

Suppose now that we are given a symmetric matrix $B$. Let it be diagonalized as $B = O D \transp{O}$ with $O$ an orthogonal matrix.
A projection of $B$ onto $\Grass{d}{\R^m}$ is a matrix $P$ which minimizes the following quantity:
\begin{equation}
\label{eq:min_grass}
\min_{P \in \Grass{d}{E}} \frobN{ B - P }.
\end{equation}
%This problem is equivalent to 
%\begin{align*}
%\min_{P \in \Grass{d}{E}} \frobN{ D - P }
%\end{align*}
%via $P \mapsto \transp{O} P O$.
By applying the transformation $P \mapsto \transp{O} P O$, we see that this problem is equivalent to $\min_{P \in \Grass{d}{E}} \frobN{ D - P }$.
Now, let $e_1, \cdots, e_m$ denote the canonical basis of $\R^m$.
We have
\begin{align*}
\frobN{ D - P}^2 
&=  \frobN{D}^2 + \frobN{P}^2 - 2\frobP{D}{P} \\
&=  \frobN{D}^2 + \frobN{P}^2 - 2\sum  \eucP{\lambda_i e_i}{P(e_i)},
\end{align*}
where $\frobP{\cdot}{\cdot}$ is the Frobenius inner product, and $\eucP{\cdot}{\cdot}$ the usual inner product on $\R^m$.
Therefore, Equation \eqref{eq:min_grass} is a problem equivalent to
\begin{align*}
\max_{P\in \Grass{d}{E}} \sum \lambda_i \eucP{e_i}{P(e_i)}.
\end{align*}
Since $P$ is an orthogonal projection, we have $\eucP{e_i}{P(e_i)} = \eucP{P(e_i)}{P(e_i)} = \eucN{P(e_i)}^2$ for all $i \in \llbracket1,m\rrbracket$.
Moreover, $d = \frobN{P}^2 = \sum \eucN{P(e_i)}^2$.
Denoting $p_i = \eucN{P(e_i)}^2 \in [0,1]$, we finally obtain the following alternative formulation of Equation \eqref{eq:min_grass}:
\begin{align*}
\max_{\substack{p_1,...p_m \in [0,1] \\ p_1+...+p_m = d}} \sum \lambda_i p_i.
\end{align*}
Using that $\lambda_1 \geq ... \geq \lambda_m$, we see that this maximum is attained when $p_1 = ... = p_d=1$ and $p_{d+1}=...=p_m = 0$. 
Consequently, a minimizer of Equation \eqref{eq:min_grass} is $P = J_d$, where $J_d$ is the diagonal matrix whose first $d$ terms are 1, and the other ones are zero.
Moreover, it is unique if $\lambda_d \neq \lambda_{d+1}$.
As a consequence of these considerations, we obtain the following characterization: for every $B\in \matrixspace{\R^m}$,
\begin{equation}
\label{eq:med_Gd}
B \in \med{\Grass{d}{\R^m}} \iff \lambda_d(B^s) = \lambda_{d+1}(B^s).
\end{equation}

Let us now show that for every matrix $A\in \matrixspace{\R^m}$, we have
\begin{align*}
\dist{A}{\med{\Grass{d}{\R^m}}} = \frac{\sqrt{2}}{2} \big|\lambda_d(A^s) - \lambda_{d+1}(A^s)\big|.
\end{align*}
First, remark that
\begin{equation}
\label{eq:proj_on_med}
\dist{A}{\med{\Grass{d}{\R^m}}} = \dist{A^s}{\med{\Grass{d}{\R^m}}}.
\end{equation}
Indeed, if $B$ is a projection of $A$ on $\med{\Grass{d}{\R^m}}$, then $B^s$ is still in $\med{\Grass{d}{\R^m}}$ according to Equation \eqref{eq:med_Gd}, and 
\begin{align*}
\dist{A}{\med{\Grass{d}{\R^m}}} = \frobN{A-B} \geq \frobN{A^s - B^s} \geq \dist{A^s}{\med{\Grass{d}{\R^m}}}.
\end{align*}
Conversely, if $B$ is a projection of $A^s$ on $\med{\Grass{d}{\R^m}}$, then $\hat B = B + A - A^s$ is still in $\med{\Grass{d}{\R^m}}$, and \begin{align*}
\dist{A}{\med{\Grass{d}{\R^m}}} \leq \frobN{A - \hat B} = \frobN{A^s - B} = \dist{A^s}{\med{\Grass{d}{\R^m}}}.
\end{align*}
We deduce Equation \eqref{eq:proj_on_med}.
Now, let $A \in \SS$ and $B \in \med{\Grass{d}{\R^m}}$. 
Let $e_1,...,e_m$ be a basis of $\R^m$ that diagonalizes $A$.
Writing $\frobN{A-B} = \sum \eucN{A(e_i) - B(e_i)}^2 = \sum \eucN{\lambda_i(A) e_i - B(e_i)}^2$, it is clear that the closest matrix $B$ must satisfy $B(e_i) = \lambda_i(B) e_i$, with
\begin{itemize}
\item $\lambda_i(B) = \lambda_i(A)$ for $i \notin  \{d,d+1\}$,
\item $\lambda_d(B) = \lambda_{d+1}(B) = \frac{1}{2}(\lambda_d(A)+\lambda_{d+1}(A))$.
\end{itemize}
We finally compute
\begin{align*}
\frobN{A-B}^2 
&= \sum \eucN{\lambda_i(A) e_i - \lambda_i(B) e_i}^2 \\
&= \left\lvert\lambda_d(A) - \lambda_d(B) \right\rvert^2 + \left \lvert\lambda_{d+1}(A) - \lambda_{d+1}(B) \right\rvert^2 \\
&= \frac{1}{2} \left\lvert \lambda_d(A) - \lambda_{d+1}(A)\right\rvert^2
\end{align*}
which yields the result.
\end{proof}

Observe that, as a consequence of this lemma, every point of $\Grass{d}{\R^m}$ is at equal distance from $\med{\Grass{d}{\R^m}}$, and this distance is equal to $\frac{\sqrt{2}}{2}$. Therefore the reach of the subset $\Grass{d}{\R^m} \subset \matrixspace{\R^m}$ is 
\begin{align*}
\reach{\Grass{d}{\R^m}} = \frac{\sqrt{2}}{2}. 
\end{align*}

\paragraph{\v{C}ech bundle filtration.}
Let $X$ be a subset of $E= \R^n \times \matrixspace{\R^m}$.
Consider the usual \v{C}ech filtration $\X = (X^t)_{t \geq 0}$, where $X^t$ denotes the $t$-thickening of $\check{X}$ in the metric space $(E, \gammaN{\cdot})$.
It is also known as the offset filtration.
In order to give this filtration a vector bundle structure, consider the map $p^t$ defined as the composition
\begin{equation}
\begin{tikzcd}[baseline=(current  bounding  box.center), column sep = 5.3em]
X^t \subset \R^n \times \matrixspace{\R^m} \arrow[r, "\mathrm{proj}_2"] %\arrow[rr, bend right, "p^t"]
& \matrixspace{\R^m} \setminus \med{\Grass{d}{\R^m}} \arrow[r, "\proj{\cdot}{\Grass{d}{\R^m}}"]
& \Grass{d}{\R^m},
\end{tikzcd}
\label{eq:def_cech_bundle_proj}
\end{equation}
where $\mathrm{proj}_2$ represents the projection on the second coordinate of $\R^n \times \matrixspace{\R^m}$, and $\proj{\cdot}{\Grass{d}{\R^m}}$ the projection on $\Grass{d}{\R^m} \subset \matrixspace{\R^m}$.
Note that $p^t$ is well-defined only when $X^t$ does not intersect $\R^n \times \med{\Grass{d}{\R^m}}$. The supremum of such $t$'s is denoted $\tmaxgamma{X}$. We have
\begin{equation}
\label{eq:tmaxgamma}
\tmaxgamma{X} = \inf \left\{ \distgamma{x}{\R^n \times \med{ \Grass{d}{\R^m} }}, ~x \in X \right\},
\end{equation}
where $\distgamma{x}{\R^n \times \med{ \Grass{d}{\R^m} }}$ is the distance between the point $x \in \R^n \times \matrixspace{\R^m}$ and the subspace $\R^n \times \med{ \Grass{d}{\R^m}}$, with respect to the norm $\gammaN{\cdot}$.
By definition of $\gammaN{\cdot}$, Equation \eqref{eq:tmaxgamma} rewrites as
\begin{equation*}
\tmaxgamma{X} = \gamma \cdot \inf \{ \dist{A}{\med{ \Grass{d}{\R^m} }}, ~(y,A) \in X \}, 
\end{equation*}
where $\dist{A}{\med{ \Grass{d}{\R^m} }}$ represents the distance between the matrix $A$ and the subset $\med{ \Grass{d}{\R^m}}$ with respect to the Frobenius norm $\frobN{\cdot}$.
Denoting $\tmax{X}$ the value $\tmaxgamma{X}$ for $\gamma = 1$, we obtain 
\begin{align}
\label{eq:tmax}
\begin{split}
\tmaxgamma{X} &= \gamma \cdot \tmax{X} \\
\mathrm{and}~~~~~~~\tmax{X} &=\inf \{ \dist{A}{\med{ \Grass{d}{\R^m} }}, ~(y,A) \in X \}.
\end{split}
\end{align}
Note that the values $\tmax{X}$ can be computed explicitly thanks to Lemma \ref{lem:projongrass}.
In particular, if $X$ is a subset of $\R^n \times \Grass{d}{\R^m}$, then $\tmax{X} = \frac{\sqrt{2}}{2}$. Accordingly,
\begin{equation}
\label{eq:tmax_subset_grass}
\tmaxgamma{X} = \frac{\sqrt{2}}{2} \gamma.
\end{equation} 

\begin{definition}
\label{def:filtered_cech_bundle}
Consider a subset $X$ of $E=\R^n \times \matrixspace{\R^m}$, and suppose that $\tmax{X} > 0$.
The \emph{\v{C}ech bundle filtration} associated to $X$ in the ambient space $(E, \gammaN{\cdot})$ is the vector bundle filtration $(\X, \p)$ consisting of the \v{C}ech filtration $\X = (X^t)_{t \in T}$, and the maps $\p = (p^t)_{t \in T}$ as defined in Equation \eqref{eq:def_cech_bundle_proj}. This vector bundle filtration is defined on the index set $T = \left[0, \tmaxgamma{X} \right)$, where $\tmaxgamma{X}$ is defined in Equation \eqref{eq:tmax}.
\end{definition}

The $i^\text{th}$ persistent Stiefel-Whitney class of the \v{C}ech bundle filtration $(\X, \p)$, as in Definition \ref{def:persistent_SF_classes}, shall be denoted $w_i(X)$ instead of $w_i(\p)$.

\begin{example}
\label{ex:normal_mobius}
Let $E = \R^2 \times \matrixspace{\R^2}$.
Let $X$ and $Y$ be the subsets of $E$ defined as:
\begin{align*}
&X = \bigg\{
\bigg( 
\begin{pmatrix}
\cos(\theta)  \\
\sin(\theta) 
\end{pmatrix}
,
\begin{pmatrix}
\cos(\theta)^2 & \cos(\theta) \sin(\theta) \\
\cos(\theta) \sin(\theta) & \sin(\theta)^2 
\end{pmatrix}
\bigg),
\theta \in [0, 2\pi)  \bigg\} \\
&Y = \bigg\{
\bigg( 
\begin{pmatrix}
\cos(\theta)  \\
\sin(\theta) 
\end{pmatrix}
,
\begin{pmatrix}
\cos(\frac{\theta}{2})^2 & \cos(\frac{\theta}{2}) \sin(\frac{\theta}{2}) \\
\cos(\frac{\theta}{2}) \sin(\frac{\theta}{2}) & \sin(\frac{\theta}{2})^2 
\end{pmatrix}
\bigg),
\theta \in [0, 2\pi)  \bigg\}
\end{align*}
The set $X$ is to be seen as the normal bundle of the circle, and $Y$ as the tautological bundle of the circle, also known as the Mobius strip.
They are pictured in Figures \ref{fig:2} and \ref{fig:3}.
We have $\tmax{X} = \tmax{Y} = \frac{\sqrt{2}}{2}$ as in Lemma \ref{lem:projongrass}.
Let $\gamma = 1$.

\begin{figure}[H]
\begin{minipage}{.49\linewidth}
\centering
\includegraphics[width=.6\linewidth]{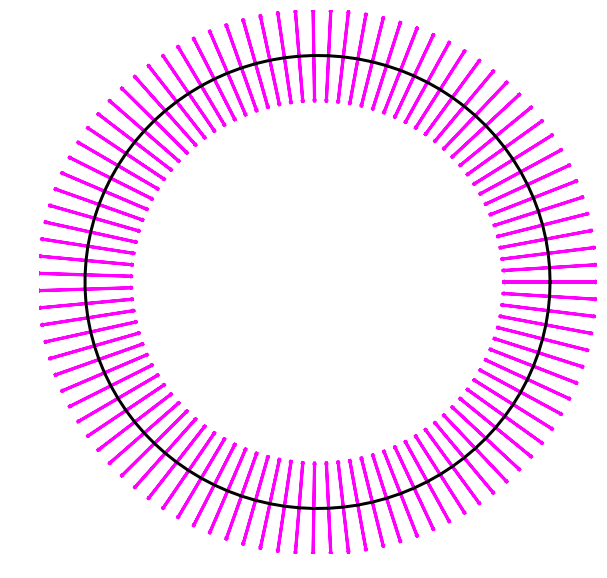}
\end{minipage}
\begin{minipage}{.49\linewidth}
\centering
\includegraphics[width=.6\linewidth]{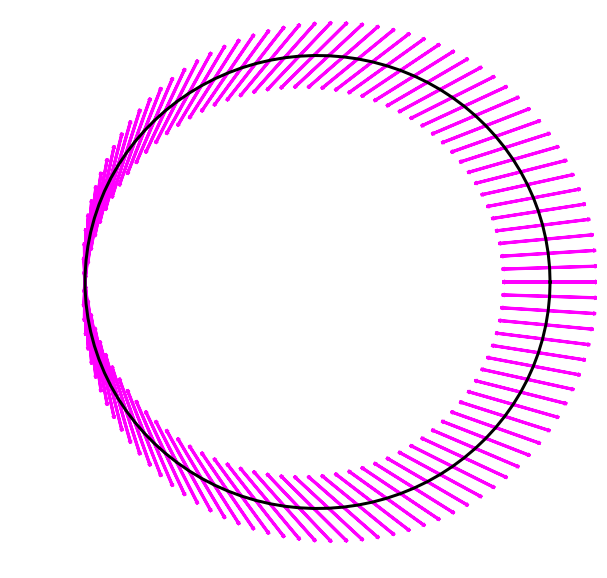}
\end{minipage}
\caption{Representation of the sets $X$ and $Y \subset \R^2 \times \matrixspace{\R^2}$: the black points correspond to the $\R^2$-coordinate, and the pink segments over them correspond to the orientation of the $\matrixspace{\R^2}$-coordinate.}
\label{fig:2}
\end{figure}

\begin{figure}[H]
\begin{minipage}{.49\linewidth}
\centering
\includegraphics[width=.6\linewidth]{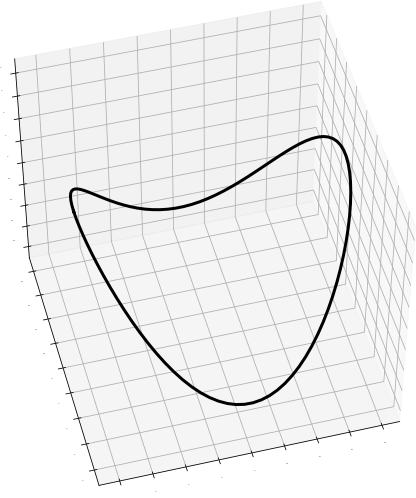}
\end{minipage}
\begin{minipage}{.49\linewidth}
\centering
\includegraphics[width=.6\linewidth]{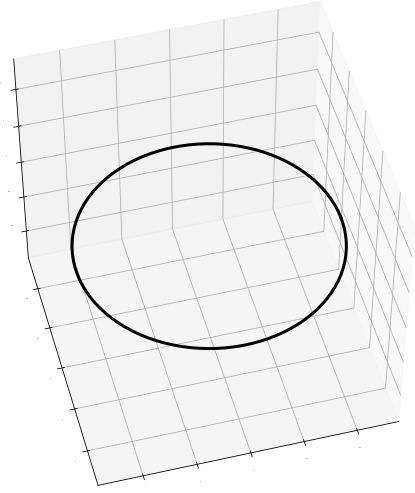}
\end{minipage}
\caption{The sets $X$ and $Y \subset \R^2 \times \matrixspace{\R^2}$, projected in a 3-dimensional subspace of $\R^3$ via Principal Component Analysis.}
\label{fig:3}
\end{figure}

We now compute the persistence barcodes of the \v{C}ech filtrations of $X$ and $Y$ in the ambient space $E$, as represented in Figure \ref{fig:4}.

\begin{figure}[H]
\begin{minipage}{.49\linewidth}
\centering
\includegraphics[width=.9\linewidth]{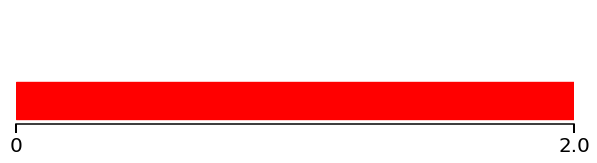}
\includegraphics[width=.9\linewidth]{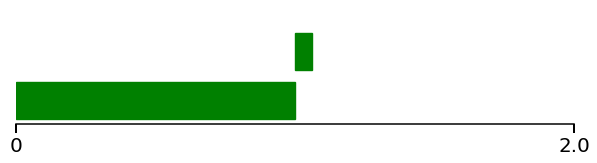}
\end{minipage}
\begin{minipage}{.49\linewidth}
\centering
\includegraphics[width=.9\linewidth]{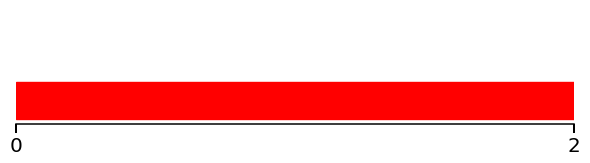}
\includegraphics[width=.9\linewidth]{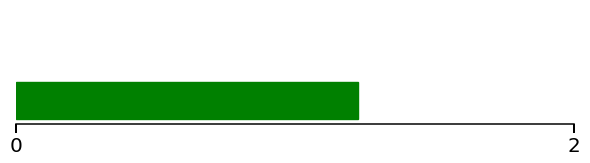}
\end{minipage}
\caption{$H^0$ and $H^1$ persistence barcodes of the \v{C}ech filtration of $X$ (left) and $Y$ (right).}
\label{fig:4}
\end{figure}

\noindent
Consider the first persistent Stiefel-Whitney classes $w_1(X)$ and $w_1(Y)$ of the corresponding \v{C}ech bundle filtrations. We compute that their lifebars are $\emptyset$ for $w_1(X)$, and $\left[0, \tmax{Y}\right)$ for $w_1(Y)$. This is illustrated in Figure \ref{fig:5}. 
One reads these bars as follows: $w_1^t(X)$ is zero for every $t \in \left[0, \frac{\sqrt{2}}{2}\right)$, while $w_1^t(Y)$ is nonzero. 

\begin{figure}[H]
\begin{minipage}{.49\linewidth}
\centering
\includegraphics[width=.9\linewidth]{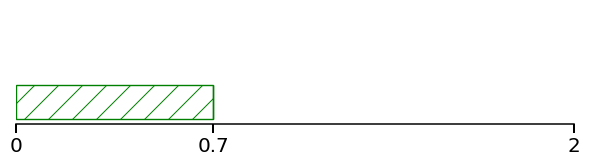}
\end{minipage}
\begin{minipage}{.49\linewidth}
\centering
\includegraphics[width=.9\linewidth]{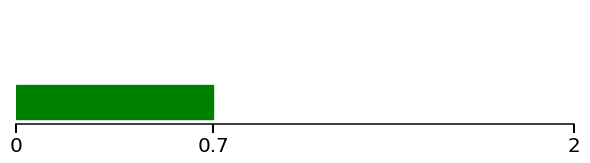}
\end{minipage}
\caption{Lifebars of the first persistent Stiefel-Whitney classes $w_1(X)$ and $w_1(Y)$.}
\label{fig:5}
\end{figure}
\end{example}

\subsection{Stability}
In this subsection we derive a straightforward stability result for persistent Stiefel-Whitney classes. 
We start by defining a notion of interleavings for vector bundle filtrations, in the same vein as the usual interleavings of set filtrations.

\begin{definition}
\label{def:interleavings_bundles}
Let $\epsilon \geq 0$, and consider two vector bundle filtrations $(\X, \p)$, $(\Y,\q)$ of dimension $d$ on $E$ with respective index sets $T$ and $U$.
They are \emph{$\epsilon$-interleaved} if the underlying filtrations $\X = (X^t)_{t \in T}$ and $\Y = (Y^t)_{t \in U}$ are $\epsilon$-interleaved, and if the following diagrams commute for every $t \in T \cap (U - \epsilon)$ and $s \in U \cap (T - \epsilon)$:
\begin{center}
\begin{minipage}{.49\linewidth}
\centering
\begin{tikzcd}%[row sep=tiny]
X^t \arrow[dr, "p^t", swap] \arrow[rr, hook]&  & Y^{t+\epsilon} \arrow[dl, "q^{t+\epsilon}"] \\
& \Grass{d}{\R^\infty} &
\end{tikzcd}
\end{minipage}
\begin{minipage}{.49\linewidth}
\centering
\begin{tikzcd}%[row sep=tiny]
Y^s \arrow[dr, "q^s", swap] \arrow[rr, hook]&  & X^{s+\epsilon} \arrow[dl, "q^{s+\epsilon}"] \\
& \Grass{d}{\R^\infty} &
\end{tikzcd}
\end{minipage}
\end{center}
\end{definition}

The following theorem shows that interleavings of vector bundle filtrations give rise to interleavings of persistence modules which respect the persistent Stiefel-Whitney classes.

\begin{theorem}
\label{thm:stability}
Consider two vector bundle filtrations $(\X, \p)$, $(\Y,\q)$ of dimension $d$ with respective index sets $T$ and $U$. Suppose that they are $\epsilon$-interleaved. Then there exists an $\epsilon$-interleaving $(\phi, \psi)$ between their corresponding persistent cohomology modules which sends persistent Stiefel-Whitney classes on persistent Stiefel-Whitney classes.
In other words, for every $i \in \llbracket 1,d \rrbracket$, and for every $t \in (T + \epsilon) \cap U$ and $s \in U \cap (T + \epsilon)$, we have
\begin{align*}
&\phi^t( w_i^t(\p) ) = w_i^{t-\epsilon}(\q) \\
\text{ and ~~ } &\psi^s( w_i^s(\p) ) = w_i^{s-\epsilon}(\q).
\end{align*}
\end{theorem}

\begin{proof}
Define $(\phi, \psi)$ to be the $\epsilon$-interleaving between the cohomology persistence modules $\V(\X)$ and $\V(\Y)$ given by the $\epsilon$-interleaving between the filtrations $\X$ and $\Y$.
Explicitly, if $i_t^{t+\epsilon}$ denotes the inclusion $X^t \hookrightarrow Y^{t+\epsilon}$ and $j_s^{s+\epsilon}$ denotes the inclusion $Y^s \hookrightarrow X^{s+\epsilon}$, then $\phi = (\phi^t)_{t \in (T+\epsilon) \cap U}$ is given by the induced maps in cohomology $\phi^t = (i_{t-\epsilon}^{t})^*$, and $\psi = (\psi^s)_{s \in (U+\epsilon) \cap T}$ is given by $\psi^s = (j_{s-\epsilon}^{s})^*$. 
%Therefore the persistence modules are $\epsilon$-interleaved.

Now, by fonctoriality, the diagrams of Definition \ref{def:interleavings_bundles} give rise to commutative diagrams in cohomology:
\begin{center}
\begin{minipage}{.49\linewidth}
\centering
\begin{tikzcd}[column sep=tiny]
H^*(X^{t-\epsilon}) &  & H^*(Y^{t}) \arrow[ll, "\phi^t", swap]  \\
& H^*(\Grass{d}{\R^\infty}) \arrow[ul, "(p^{t-\epsilon})^*"] \arrow[ur, "(q^{t})^*", swap] &
\end{tikzcd}
\end{minipage}
\begin{minipage}{.49\linewidth}
\centering
\begin{tikzcd}[column sep=tiny]
H^*(Y^{s-\epsilon}) &  & H^*(X^{s}) \arrow[ll, "\psi^s", swap]  \\
& H^*(\Grass{d}{\R^\infty}) \arrow[ul, "(q^{s-\epsilon})^*"] \arrow[ur, "(p^{s})^*", swap] &
\end{tikzcd}
\end{minipage}
\end{center}
Let $i \in \llbracket1,d\rrbracket$. By definition, the persistent Stiefel-Whitney classes $w_i(\p) = (w_i^t(\p))_{t \in T}$ and $w_i(\q) = (w_i^s(\q))_{s \in U}$ are equal to $w_i^t(\p) = (p^t)^*(w_i)$ and $w_i^s(\q) = (q^s)^*(w_i)$, where $w_i$ is the $i^\text{th}$ Stiefel-Whitney class of $\Grass{d}{\R^\infty}$.
The previous commutative diagrams then translates as $\phi^t( w_i^t(\p) ) = w_i^{t-\epsilon}(\q)$ and $\psi^s( w_i^s(\p) ) = w_i^{s-\epsilon}(\q)$, as wanted.
\end{proof}

Consider two vector bundle filtrations $(\X, \p)$, $(\Y,\q)$ such that there exists an $\epsilon$-interleaving $(\phi, \psi)$ between their persistent cohomology modules $\V(\X)$, $\V(\Y)$ which sends persistent Stiefel-Whitney classes on persistent Stiefel-Whitney classes.
Let $i \in \llbracket1,d\rrbracket$. Then the lifebars of their $i^\text{th}$ persistent Stiefel-Whitney classes $w_i(\p)$ and $w_i(\q)$ are $\epsilon$-close in the following sense: if we denote $\tdeatho(\p) = \inf \{t \in T, w_i^t(\p) \neq 0\}$ and $\tdeatho(\q) = \inf \{t \in T, w_i^t(\q) \neq 0\}$, then $|\tdeatho(\p) - \tdeatho(\q)| \leq \epsilon$.
This can be seen from their lifebar representations, as shown in Figure \ref{fig:6}.

\begin{figure}[H]
\centering
\includegraphics[width=.5\linewidth]{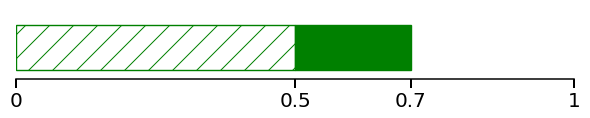}
\includegraphics[width=.5\linewidth]{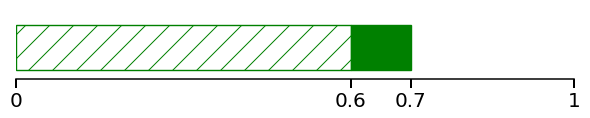}
\caption{Two $\epsilon$-close lifebars, with $\epsilon = 0{.}1$.}
\label{fig:6}
\end{figure}

Let us apply this result to the \v{C}ech bundle filtrations. Let $X$ and $Y$ be two subsets of $E = \R^n \times \matrixspace{\R^m}$. Suppose that the Hausdorff distance $\Hdist{X}{Y}$, with respect to the norm $\gammaN{\cdot}$, is lower than $\epsilon$, meaning that the $\epsilon$-thickenings $X^\epsilon$ and $Y^\epsilon$ satisfy $Y \subseteq X^\epsilon$ and $X \subseteq Y^\epsilon$. It is then clear that the vector bundle filtrations are $\epsilon$-interleaved, and we can apply Theorem \ref{thm:stability} to obtain the following result.

\begin{corollary}
\label{cor:stability}
If two subsets $X, Y \subset E$ satisfy $\Hdist{X}{Y} \leq \epsilon$, then there exists an $\epsilon$-interleaving between the persistent cohomology modules of their corresponding \v{C}ech bundle filtrations which sends persistent Stiefel-Whitney classes on persistent Stiefel-Whitney classes. 
\end{corollary}

\begin{example}
In order to illustrate Corollary \ref{cor:stability}, consider the sets $X'$ and $Y'$ represented in Figure \ref{fig:7}. They are noisy samples of the sets $X$ and $Y$ defined in Example \ref{ex:normal_mobius}.
They contain 50 points each.

Figure \ref{fig:8} represents the barcodes of the \v{C}ech filtrations of the sets $X'$ and $Y'$, together with the lifebar of the first persistent Stiefel-Whitney class of their corresponding \v{C}ech bundle filtrations.
We see that they are close to the original descriptors of $X$ and $Y$ (Figure \ref{fig:5}).
Experimentally, we computed that the Hausdorff distances between $X,X'$ and $Y,Y'$ are approximately $\Hdist{X}{X'} \approx 0{.}5$ and $\Hdist{Y}{Y'} \approx 0{.}4$.
We observe that this is coherent with the lifebar of $w_1(Y')$, which is $\epsilon$-close to the lifebar of $w_1(Y)$ with $\epsilon \approx 0{.}3 \leq 0{.}4$.

\begin{figure}[H]
\centering
\begin{minipage}{.49\linewidth}
\centering
\includegraphics[width=.6\linewidth]{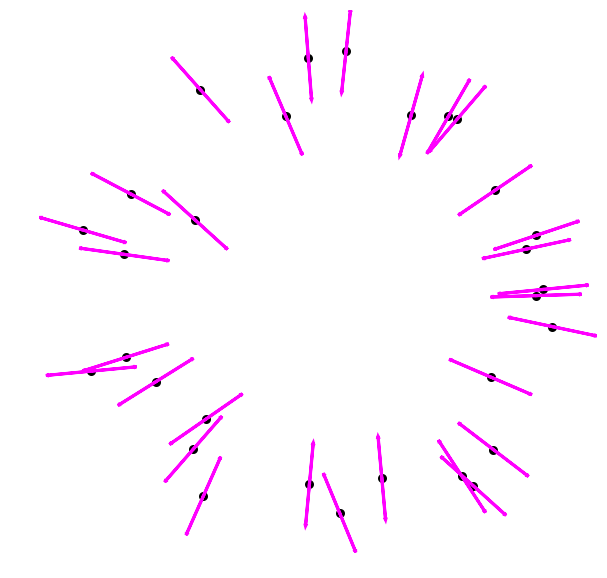}
\end{minipage}
\begin{minipage}{.49\linewidth}
\centering
\includegraphics[width=.6\linewidth]{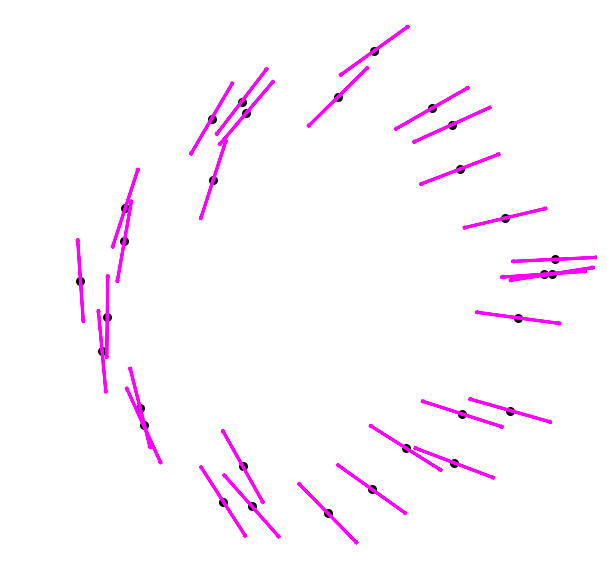}
\end{minipage}
\caption{Representation of the sets $X', Y' \subset \R^2 \times \matrixspace{\R^2}$. The black points correspond to the $\R^2$-coordinate, and the pink segments over them correspond to the orientation of the $\matrixspace{\R^2}$-coordinate.}
\label{fig:7}
\end{figure}

\begin{figure}[H]
\centering
\begin{minipage}{.49\linewidth}
\centering
\includegraphics[width=.9\linewidth]{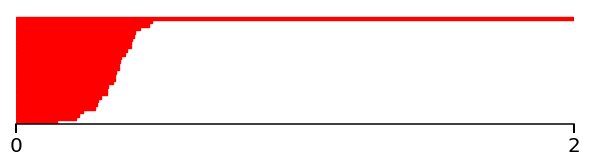}
\includegraphics[width=.9\linewidth]{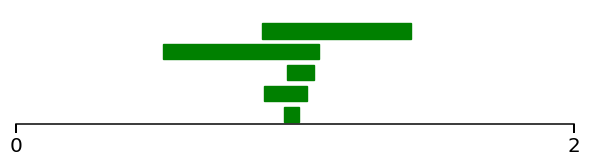}
\includegraphics[width=.9\linewidth]{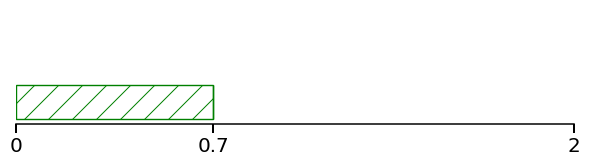}
\end{minipage}
\begin{minipage}{.49\linewidth}
\centering
\includegraphics[width=.9\linewidth]{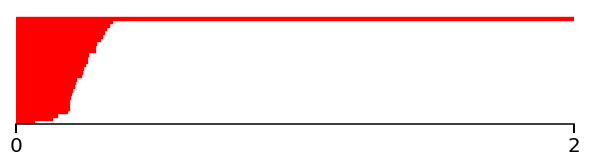}
\includegraphics[width=.9\linewidth]{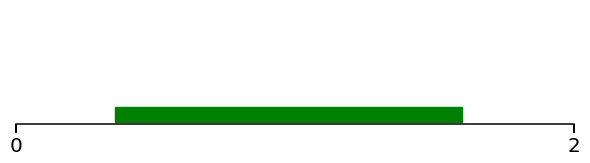}
\includegraphics[width=.9\linewidth]{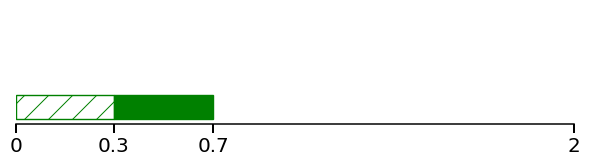}
\end{minipage}
\caption{Left: $H^0$ and $H^1$ barcodes of $X'$ and lifebar of $w_1(X')$. Right: same for $Y'$. Only bars of length larger than $0.05$ are represented.}
\label{fig:8}
\end{figure}
\end{example}

\subsection{Consistency}
\label{subsec:consistency}

In this subsection we describe a setting where the persistent Stiefel-Whitney classes $w_i(X)$ of the \v{C}ech bundle filtration of a set $X$ can be seen as consistent estimators of the Stiefel-Whitney classes of some underlying vector bundle.

Let $\MMo$ be a compact $\CC^3$-manifold, and $u \colon \MMo \rightarrow \MM\subset \R^n$ an immersion. Suppose that $\MMo$ is given a $d$-dimensional vector bundle structure $p \colon \MMo \rightarrow \Grass{d}{\R^m}$. 
Let $E = \R^n \times \matrixspace{\R^m}$, and consider the set 
\begin{equation}
\label{eq:lifted_set}
\MMcheck = \left\{ \left(\imm(x_0), \projmatrix{p(x_0)}\right), x_0 \in \MMo \right\} \subset E,
\end{equation}
where $\projmatrix{p(x_0)}$ denotes the orthogonal projection matrix onto the subspace $p(x_0) \subset \R^m$. The set $\MMcheck$ is called the \emph{lift} of $\MMo$.
Consider the \emph{lifting map} defined as
\begin{align}
\label{eq:lifting_map}
\begin{split}
\immcheck\colon \MMo & \longrightarrow \MMcheck \subset E \\
x_0 & \longmapsto \left(\imm(x_0), \projmatrix{p(x_0)}\right).
\end{split}
\end{align}
We make the following assumption: $\immcheck$ is an embedding. As a consequence, $\MMcheck$ is a submanifold of $E$, and $\MMo$ and $\MMcheck$ are diffeomorphic. 
%It is worth noting that this point of view is strongly connected to the work of \cite{Tinarrage}, where we estimated the tangent bundle of an immersed manifold in order to get back to the abstract manifold $\MMo$.  

The persistent cohomology of $\MMcheck$ can be used to recover the cohomology of $\MMo$.
To see this, select $\gamma>0$, and denote by $\reach{\MMcheck}$ the reach of $\MM$, where $E$ is endowed with the norm $\gammaN{\cdot}$.
Since $\MMcheck$ is a $\CC^2$-submanifold, $\reach{\MMcheck}$ is positive. Note that it depends on $\gamma$.
Let $\CechF{\MMcheck}= (\MMcheck^t)_{t\geq 0}$ be the \v{C}ech set filtration of $\MMcheck$ in the ambient space $(E, \gammaN{\cdot})$, and let $\V(\MMcheck)$ be the corresponding persistent cohomology module. 
For every $s, t \in [0, \reach{\MMcheck})$ such that $s \leq t$, we know that the inclusion maps $i_s^t\colon \MMcheck^s \hookrightarrow \MMcheck^t$ are homotopy equivalences.
Hence the persistence module $\V(\MMcheck)$ is constant on the interval $[0, \reach{\MMcheck})$, and is equal to the cohomology $H^*(\MMcheck) = H^*(\MMo)$.

Consider the \v{C}ech bundle filtration $(\CechF{\MMcheck}, \p)$ of $\MMcheck$. The following theorem shows that the persistent Stiefel-Whitney classes $w_i^t(\MMcheck)$ are also equal to the usual Stiefel-Whitney classes of the vector bundle $(\MMo, p)$.

\begin{theorem}
\label{thm:consistency}
Let $\MMo$ be a compact $\CC^3$-manifold, $u \colon \MMo \rightarrow \R^n$ an immersion and $p \colon \MMo \rightarrow \Grass{d}{\R^m}$ a continuous map.
Let $\MMcheck$ be the lift of $\MMo$ (Equation (\ref{eq:lifted_set})) and $\immcheck$ the lifting map (Equation (\ref{eq:lifting_map})).
Suppose that $\imm$ is an embedding.

Let $\gamma > 0$ and consider the \v{C}ech bundle filtration $(\CechF{\MMcheck}, \p)$ of $\MMcheck$. 
Its maximal filtration value is $\tmaxgamma{\MMcheck} = \frac{\sqrt{2}}{2} \gamma$.
Denote by $\pSF{i}{\mathbbm{p}} = (\pSFt{i}{\mathbbm{p}}{t})_{t \in T}$ its persistent Stiefel-Whitney classes, $i \in \llbracket1,d\rrbracket$.
Denote also by $i_0^t$ the inclusion $\MMcheck \rightarrow \MMcheck^t$, for $t \in [0, \reach{\MMcheck})$.

Let $t\geq0$ be such that $t < \min\left(\reach{\MMcheck}, \tmaxgamma{\MMcheck}\right)$.
Then the map $i_0^t \circ \immcheck\colon \MMo \rightarrow \MMcheck^t$ induces an isomorphism $H^*(\MMo) \leftarrow H^*(\MMcheck^t)$ which maps the $i^\text{th}$ persistent Stiefel-Whitney class $\pSFt{i}{\p}{t}$ of $(\CechF{\MMcheck}, \p)$ to the $i^\text{th}$ Stiefel-Whitney class of $(\MMo, p)$.
\end{theorem}

\begin{proof}
Consider the following commutative diagram, defined for every $t < \tmaxgamma{\MMcheck}$:
\begin{equation*}
\begin{tikzcd}
\MMo \arrow[r, "\immcheck"]  \arrow[dr, "p", swap] & \MMcheck \arrow[r, hook, "i_0^t"] & \MMcheck^t \arrow[dl, "p^t" ] \\
& \Grass{d}{\R^m} & &
\end{tikzcd}
\end{equation*}
We obtain a commutative diagram in cohomology:
\begin{equation*}
\begin{tikzcd}
H^*(\MMo)  & H^*(\MMcheck) \arrow[l, "\immcheck^*", swap]  & H^*(\MMcheck^t) \arrow[l, "(i_0^t)^*", swap] \\ %\arrow[ll, bend right = 20] 
& H^*(\Grass{d}{\R^m}) \arrow[ur, "(p^t)^*", swap] \arrow[ul, "p^*"] & &
\end{tikzcd}
\end{equation*}
Since $t < \reach{\MMcheck}$, the map $(i_0^t)^*$ is an isomorphism. So is $\immcheck^*$ since $\immcheck$ is an embedding.
As a consequence, the map $i_0^t\circ \immcheck$ induces an isomorphism $H^*(\MMo) \simeq H^*(\MMcheck^t)$.

Let $w_i$ denotes the $i^\text{th}$ Stiefel-Whitney class of $\Grass{d}{\R^m}$.
By definition, the $i^\text{th}$ Stiefel-Whitney class of $(\MMo, p)$ is $p^*(w_i)$,
and the $i^\text{th}$ persistent Stiefel-Whitney class of $(\CechF{\MMcheck}, \p)$ is $\pSFt{i}{\p}{t} = (p^t)^*(w_i)$.
By commutativity of the diagram, we obtain $p^*(w_i) = (p^t)^*(w_i)$ under the identification $H^*(\MMo) \simeq H^*(\MMcheck^t)$.
\end{proof}

We deduce an estimation result.

\begin{corollary}
\label{cor:consistency_stability}
Let $X \subset E$ be any subset such that $\Hdist{X}{\MMcheck} \leq \frac{1}{17}\reach{\MMcheck}$. Define $\epsilon = \Hdist{X}{\MMcheck}$.
Then for every 
$t\in \left[4\epsilon, \reach{\MMcheck}-3\epsilon\right)$, 
the composition of inclusions $\MMo \hookrightarrow \MMcheck \hookrightarrow X^t$ induces an isomorphism $H^*(\MMo) \leftarrow H^*(X^t)$ which sends the $i^\text{th}$ persistent Stiefel-Whitney class $\pSFt{i}{X}{t}$ of the \v{C}ech bundle filtration of $X$ to the $i^\text{th}$ Stiefel-Whitney class of $(\MMo, p)$.
\end{corollary}

\begin{proof}
This is a consequence of Theorems \ref{thm:stability} and \ref{thm:consistency} previously stated, and Proposition 3.1 of \cite{chazal2009sampling}.
\end{proof}

As a consequence of this corollary, on the set $[4\epsilon, \reach{\MMcheck}-3\epsilon)$, the $i^\text{th}$ persistent Stiefel-Whitney class of the \v{C}ech bundle filtration of $X$ is zero if and only if the $i^\text{th}$ Stiefel-Whitney class of $(\MMo, p)$ is.

\begin{example}
In order to illustrate Corollary \ref{cor:consistency_stability}, consider the torus and the Klein bottle, immersed in $\R^3$ as in Figure \ref{fig:9}.

\begin{figure}[H]
\centering
\begin{minipage}{.49\linewidth}
\centering
\includegraphics[width=.5\linewidth]{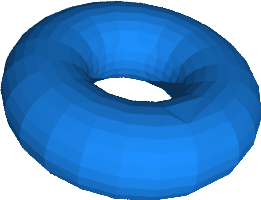}
\end{minipage}
\begin{minipage}{.49\linewidth}
\centering
\includegraphics[width=.5\linewidth]{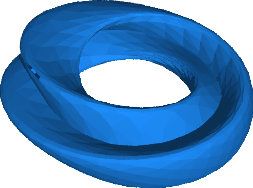}
\end{minipage}
\caption{Immersion of the torus and the Klein bottle in $\R^3$.}
\label{fig:9}
\end{figure}
\noindent
Let them be endowed with their normal bundles. They can be seen as submanifolds $\check\MM, \check\MM'$ of $\R^3 \times \matrixspace{\R^3}$.
We consider two samples $X, X'$ of $\check\MM, \check\MM'$, represented in Figure \ref{fig:10}. 
They contain respectively 346 and 1489 points.
We computed experimentally the Hausdorff distances $\Hdist{X}{\check\MM} \approx 0{.}6$ and $\Hdist{X'}{\check\MM'} \approx 0{.}45$, with respect to the norm $\gammaN{\cdot}$ where $\gamma=1$.

\begin{figure}[H]
\centering
\begin{minipage}{.49\linewidth}
\centering
\includegraphics[width=.6\linewidth]{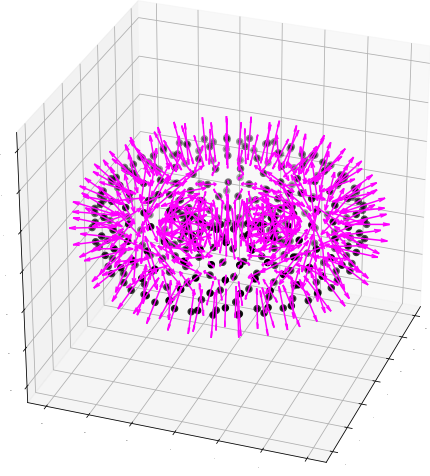}
\end{minipage}
\begin{minipage}{.49\linewidth}
\centering
\includegraphics[width=.6\linewidth]{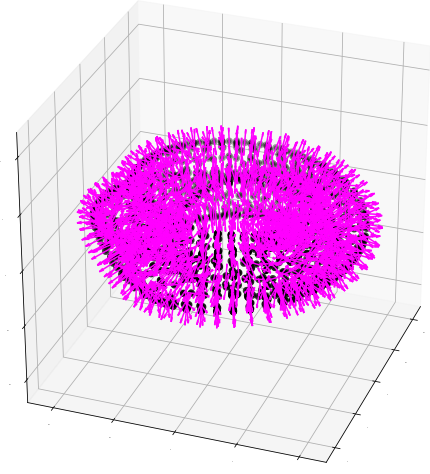}
\end{minipage}
\caption{Samples $X$ and $X'$ of $\check\MM$ and $\check\MM'$. The black points corresponds to the $\R^3$-coordinate, and the pink arrows over them correspond to the orientation of the $\matrixspace{\R^3}$-coordinate.}
\label{fig:10}
\end{figure}

\noindent
Figure \ref{fig:11} represents the barcodes of the persistent cohomology of $X$ and $X'$, and the lifebars of their first persistent Stiefel-Whitney classes $w_1(X)$ and $w_1(X')$.
Observe that $w_1(X)$ is always zero, while $w_1(X')$ is nonzero for $t\geq 0{.}3$.
This is an indication that $\check\MM$, the underlying manifold of $X$, is orientable, while $\check\MM'$ is not.
Lemma \ref{lem:orientability}, stated below, justifies this assertion.
Therefore, one interprets these lifebars as follows: $X$ is sampled on an orientable manifold, while $X'$ is sampled on a non-orientable one.

\begin{figure}[H]
\centering
\begin{minipage}{.49\linewidth}
\centering
\includegraphics[width=.9\linewidth]{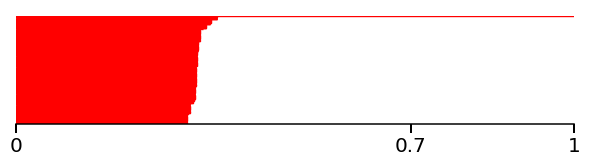}
\includegraphics[width=.9\linewidth]{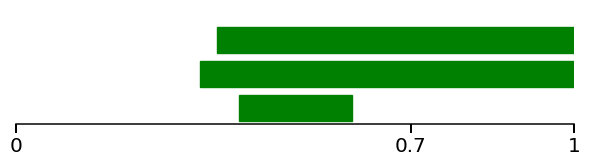}
\includegraphics[width=.9\linewidth]{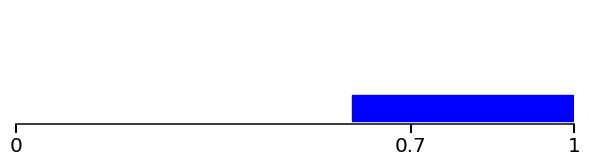}
\includegraphics[width=.9\linewidth]{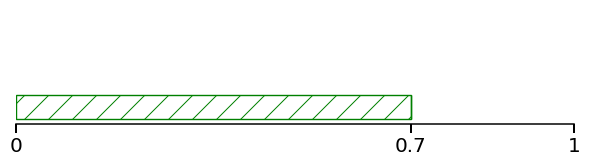}
\end{minipage}
\begin{minipage}{.49\linewidth}
\centering
\includegraphics[width=.9\linewidth]{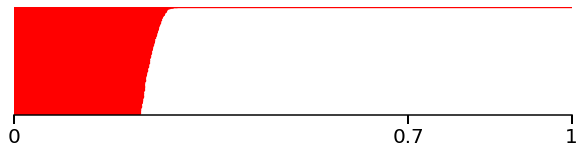}
\includegraphics[width=.9\linewidth]{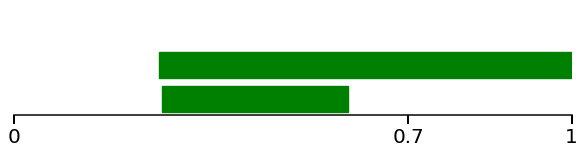}
\includegraphics[width=.9\linewidth]{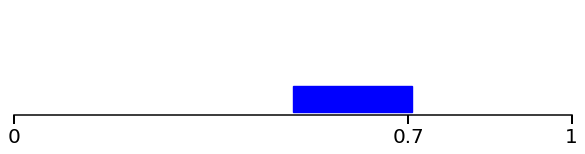}
\includegraphics[width=.9\linewidth]{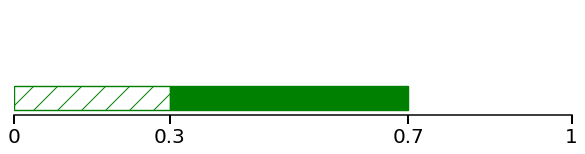}
\end{minipage}
\caption{\textbf{Left:} $H^0$, $H^1$ and $H^2$ barcodes of $X$ and lifebar of $w_1(X)$. \textbf{Right:} same for $X'$. 
Only bars of length larger than $0.2$ are represented.
}
\label{fig:11}
\end{figure}
\end{example}
\begin{lemma}
Let $\MMo \rightarrow \MM$ be an immersion of a manifold $\MMo$ in a Euclidean space.
Then $\MMo$ is orientable if and only if the first Stiefel-Whitney class of its normal bundle is zero.
\label{lem:orientability}
\end{lemma}

\begin{proof}
Let $\tau$ and $\nu$ denote the tangent and normal bundles of $\MMo$.
The Whitney sum $\tau \oplus \nu$ is a trivial bundle, hence its first Stiefel-Whitney class is $w_1(\tau \oplus \nu) = 0$. 
Using Axioms 1 and 3 of the Stiefel-Whitney classes, we obtain
\begin{align*}
w_1(\tau \oplus \nu) 
&= w_1(\tau)\cupp w_0(\nu) + w_0(\tau) \cupp w_1(\nu) \\
&= w_1(\tau) \cupp 1 + 1 \cupp w_1(\nu) \\
&= w_1(\tau) + w_1(\nu).
\end{align*}
Therefore, $w_1(\tau) = w_1(\nu)$, hence $w_1(\tau)$ is zero if and only if $w_1(\nu)$ is zero. 
Besides, it is known that the first Stiefel-Whitney class of the tangent bundle of a manifold is zero if and only if the manifold is orientable. We deduce the result.
\end{proof}

\section{Computation of persistent Stiefel-Whitney classes}
\label{sec:computation}
In order to build an effective algorithm to compute the persistent Stiefel-Whitney classes, we have to find an equivalent formulation in terms of simplicial cohomology.
We will make use of the well-known technique of simplicial approximation, as described in Subsect. \ref{background:simplicial_approx}.

\subsection{Simplicial approximation to \v{C}ech bundle filtrations}
\label{subsec:simplicial_approx_cech}

Let $X$ be a subset of $E = \R^n \times \matrixspace{\R^m}$.
Let us recall Definition \ref{def:filtered_cech_bundle}: the \v{C}ech bundle filtration associated to $X$ is the vector bundle filtration $(\X, \p)$ whose underlying filtration is the \v{C}ech filtration $\X = (X^t)_{t \in T}$, with $T = [0, \tmaxgamma{X})$, and whose maps $\p = (p^t)_{t \in T}$ are given by the following composition, as in Equation \eqref{eq:def_cech_bundle_proj}:
\begin{equation*}
\begin{tikzcd}[baseline=(current  bounding  box.center), column sep = 6em]
X^t \arrow[r, "\mathrm{proj}_2"] %\arrow[rr, bend right, "p^t"]
\arrow[rr, bend right = 20, "p^t", swap]
& \matrixspace{\R^m} \setminus \med{\Grass{d}{\R^m}} \arrow[r, "\proj{\cdot}{\Grass{d}{\R^m}}"]
& \Grass{d}{\R^m}.
\end{tikzcd}
\end{equation*}
\noindent
Let $t \in T$. The aim of this subsection is to describe a simplicial approximation to $p^t \colon X^t \rightarrow \Grass{d}{\R^m}$.
To do so, let us fix a triangulation $L$ of $\Grass{d}{\R^m}$. It comes with a homeomorphism $h \colon \Grass{d}{\R^m} \rightarrow \topreal{L}$.
We shall now triangulate the thickenings $X^t$ of the \v{C}ech set filtration.
% as described in Subsection \ref{subsec:background_persistentcohomology}.
The thickening $X^t$ is a subset of the metric space $(E, \gammaN{\cdot})$ which consists in a union of closed balls centered around points of $X$:
\begin{align*}
X^t = \bigcup_{x \in X} \closedballM{x}{t}{\gamma},
\end{align*}
where $\closedballM{x}{t}{\gamma}$ denotes the closed ball of center $x$ and radius $t$ for the norm $\gammaN{\cdot}$.
Let $\UU^t$ denote the cover $\left\{ \closedballM{x}{t}{\gamma}, ~x \in X\right\}$ of $X^t$, and let $\NN(\UU^t)$ be its nerve.
By the nerve theorem for convex closed covers \cite[Theorem 2.9]{boissonnat2018geometric}, the simplicial complex $\NN(\UU^t)$ is homotopy equivalent to its underlying set $X^t$. That is to say, there exists a continuous map $g^t \colon \topreal{\NN(\UU^t)} \rightarrow X^t$ which is a homotopy equivalence. 

As a consequence, in cohomological terms, the map $p^t \colon X^t \rightarrow \Grass{d}{E}$ is equivalent to the map $q^t$ defined as $q^t = h \circ p^t \circ g^t$.
\begin{equation}
\begin{tikzcd}
X^t \arrow[r, "p^t"] 
&[1em] \Grass{d}{\R^m} \arrow[d, "h"] \\
\topreal{\NN(\UU^t)} \arrow[u, "g^t"] \arrow[r, "q^t", dashed]
& \topreal{L}
\end{tikzcd}
\label{eq:map_approx_cech}
\end{equation}
This gives a way to compute the induced map $(p^t)^* \colon \coring{X^t} \leftarrow \coring{\Grass{d}{\R^m}}$ algorithmically: 
\begin{itemize}
\item Subdivise $\NN(\UU^t)$ until $q^t$ satisfies the star condition.
\item Choose a simplicial approximation $f^t$ to $q^t$.
\item Compute the induced map between simplicial cohomology groups $(f^t)^* \colon H^*(\NN(\UU^t)) \leftarrow H^*(L)$.
\end{itemize}
By correspondence between simplicial and singular cohomology, the map $(f^t)^*$ corresponds to $(p^t)^*$.
Hence the problem of computing $(p^t)^*$ is solved, if it were not for the following issue: in practice, the map $g^t\colon \topreal{\NN(\UU^t)} \rightarrow X^t$ given by the nerve theorem is not explicit. The rest of this subsection is devoted to showing that $g^t$ can be chosen canonically as the \emph{shadow map}.

\paragraph{Shadow map.}
We still consider $X^t$, the corresponding cover $\UU^t$ and its nerve $\NN(\UU^t)$. The underlying vertex set of the simplicial complex $\NN(\UU^t)$ is the set $X$ itself.
The shadow map $g^t \colon \topreal{ \NN(\UU^t) } \rightarrow X^t$ is defined as follows: 
for every simplex $\sigma = [x_0, ..., x_p] \in \NN(\UU^t)$ and every point $\sum_{i=0}^p \lambda_i x_i$ of $\topreal{\sigma}$ written in barycentric coordinates, associate the point $\sum_{i=0}^p \lambda_i x_i$ of $E$:
\begin{align*}
g^t \colon \sum_{i=0}^p \lambda_i x_i \in \topreal{\sigma} \longmapsto \sum_{i=0}^p \lambda_i x_i \in E.
\end{align*}
%We are not aware whether the shadow map is indeed a homotopy equivalence from $|\NN(\UU^t)|$ to $X^t$.
%Nevertheless, the following result will be enough for our purposes: the shadow map induces an isomorphism at cohomology level.
The following lemma states that this map is a homotopy equivalence.
We are not aware whether the general position assumption can be removed.

\begin{lemma}
\label{lem:shadowmap}
Suppose that $X$ is finite and in general position.
Then the shadow map $g^t \colon |\NN(\UU^t)| \rightarrow X^t$ is a homotopy equivalence. Consequently, it induces an isomorphism $(g^t)^* \colon H^*(|\NN(\UU^t)|) \leftarrow H^*(X^t)$.
\end{lemma}

\begin{proof}
Recall that $\UU^t = \left\{ \closedballM{x}{t}{\gamma}, x \in X \right \}$.
Let us consider a smaller cover. For every $x \in X$, let $\vor{x}$ denote the Voronoi cell of $x$ in the ambient metric space $(E, \gammaN{\cdot})$, and define
\begin{align*}
\VV^t = \left\{\closedballM{x}{t}{\gamma} \cap \vor{x}  , x \in X\right \}.
\end{align*}
The set $\VV^t$ is a cover of $X^t$, and its nerve $\NN(\VV^t)$ is known as the Delaunay complex. Let $h^t \colon \topreal{\NN(\VV^t)} \rightarrow X^t$ denote the shadow map of $\NN(\VV^t)$.
The Delaunay complex is a subcomplex of the \v{C}ech complex, hence we can consider the following diagram:
\begin{center}
\begin{tikzcd}
\arrow[rr, bend left, "h^t"] |\NN(\VV^t)| \arrow[r, hook]
&\arrow[r, "g^t", swap] |\NN(\UU^t)| 
&X^t .
\end{tikzcd}
\end{center}
Now, \citet[Theorem 3.2]{edelsbrunner1993union} has proven that the shadow map $h^t \colon |\NN(\VV^t)| \rightarrow X^t$ is a homotopy equivalence (it is required here that $X$ is in general position). 
Moreover, we know from \citet[Theorem 5.10]{bauer2017morse} that $\NN(\UU^t)$ collapses to $\NN(\VV^t)$. Therefore the inclusion $\topreal{\NN(\VV^t)} \hookrightarrow \topreal{\NN(\UU^t)}$ also is a homotopy equivalence.
By the 2-out-of-3 property of homotopy equivalences, we conclude that $g^t$ is a homotopy equivalence.
%This yields the following commutative diagram between cohomology rings:
%\begin{center}
%\begin{tikzcd}
%H^*(|\NN(\VV^t)|) 
%& H^*(|\NN(\UU^t)|) \arrow[l] 
%& H^*(X^t) \arrow[l, "(g^t)^*"] \arrow[ll, bend right = 23, "(h^t)^*", swap].
%\end{tikzcd}
%\end{center}
%Now, \citet[Theorem 3.2]{edelsbrunner1993union} has proven that the shadow map $h^t \colon |\NN(\VV^t)| \rightarrow X^t$ is a homotopy equivalence (it is required here that $X$ is in general position). Therefore the map $(h^t)^* \colon H^*(|\NN(\VV^t)|) \leftarrow H^*(X^t)$ is an isomorphism.
%Moreover, we know from \citet[Theorem 5.10]{bauer2017morse} that $\NN(\UU^t)$ collapses to $\NN(\VV^t)$. Therefore the inclusion $\topreal{\NN(\VV^t)} \hookrightarrow \topreal{\NN(\UU^t)}$ also is a homotopy equivalence, hence the induced map $H^*(|\NN(\VV^t)|) \leftarrow H^*(|\NN(\UU^t)|)$ is an isomorphism.
%We conclude from the last diagram that $(g^t)^*$ is an isomorphism.
%\qed
\end{proof}

\subsection{A sketch of algorithm}
Suppose that we are given a finite set $X \subset E = \R^n \times \matrixspace{\R^m}$. Choose $d \in \llbracket1, n-1\rrbracket$ and $\gamma > 0$. 
Consider the \v{C}ech bundle filtration of dimension $d$ of $X$.
Let $T = \left[0, \tmaxgamma{X}\right)$, $t \in T$ and $i \in \llbracket1,d\rrbracket$.
From the previous discussion we can infer an algorithm to solve the following problem:
\begin{center}
\fbox{\begin{minipage}{.95\linewidth}
Compute the persistent Stiefel-Whitney class $\pSFt{i}{X}{t}$ of the \v{C}ech bundle filtration of $X$, using a cohomology computation software.
\end{minipage}}
\end{center}
Denote:
\begin{itemize}
\itemsep0.2em 
\item $\X = (X^t)_{t\geq0}$ the \v{C}ech set filtration of $X$,
\item $\S$ the \v{C}ech simplicial filtration of $X$, and $g^t \colon \topreal{S^t} \rightarrow X^t$ the shadow map,
\item $L$ a triangulation of $\Grass{d}{\R^m}$ and $h \colon \Grass{d}{\R^m} \rightarrow \topreal{L}$ a homeomorphism,
\item $(\X, \p)$ the \v{C}ech bundle filtration of $X$,
\item $(\V, \vbb)$ the persistent cohomology module of $\X$,
\item $w_i \in H^i(\Grass{d}{\R^m})$ the $i^\text{th}$ Stiefel-Whitney class of the Grassmannian.
\end{itemize}
Let $t\in T$ and consider the map $q^t$, as defined in Equation \eqref{eq:map_approx_cech}:
\begin{equation*}
\begin{tikzcd}[baseline=(current  bounding  box.center), column sep = 4em]
\topreal{S^t} \arrow[r, "g^t", swap] 
\arrow[rrr, bend left = 18, "q^t"]
& X^t \arrow[r, "p^t", swap]
&\Grass{d}{\R^m} \arrow[r, "h", swap]
& \topreal{L}.
\end{tikzcd}
\end{equation*}
We propose the following algorithm:
\begin{itemize}
\itemsep0.2em 
\item Subdivise barycentrically $S^t$ until $q^t$ satisfies the star condition. Denote $k$ the number of subdivisions needed.
\item Consider a simplicial approximation $f^t \colon \subdiv{S^t}{k} \rightarrow L$ to $q^t$.
\item Compute the class $(f^t)^*(w_i)$. %\in H^i(\subdiv{S^t}{k})$.
\end{itemize}
\noindent
The output $(f^t)^*(w_i)$ is equal to the persistent Stiefel-Whitney class $\pSFt{i}{X}{t}$ at time $t$, seen in the simplicial cohomology group $H^i(S^t) = H^i(\subdiv{S^t}{k})$.
In the following section, we gather some technical details needed to implement this algorithm in practice.

Note that this also gives a way to compute the lifebar of $\pSF{i}{X}$. This bar is determined by the value $\tdeatho = \inf \{t \in T, \pSF{i}{X} \neq 0\}$. This quantity can be approximated by dichotomic search, by computing the classes $\pSFt{i}{X}{t}$ for several values of $t$.
We point out that, in order to compute the value $\tdeatho$, there may exist a better algorithm than evaluating the class $\pSFt{i}{X}{t}$ several times.

Let us describe briefly such an algorithm when $i=1$, that is, when the first persistent Stiefel-Whitney class $\pSF{1}{X}$ is to be computed.
First, we remind the reader that the first cohomology group $H^1(\Grass{d}{\R^m})$ of the Grassmannian is generated by one element, the first Stiefel-Whitney class $w_1$. For any $t \in T$, consider the map $p^t$ as above, and the map induced in cohomology, $(p^t)^*\colon H^1(X^t) \leftarrow H^1(\Grass{d}{\R^m})$.
Since $\pSFt{1}{X}{t} = (p^t)^*(w_1)$, and since $H^1(\Grass{d}{\R^m})$ has dimension 1, we have that $\pSFt{1}{X}{t}$ is nonzero if and only if $\mathrm{rank}(p^t)^*$ is nonzero.

Next, let $C(p^t)$ denotes the \emph{mapping cone} of $p^t\colon X^t \rightarrow \Grass{d}{\R^m}$. 
This is a usual construction in algebraic topology.
In a few words, the mapping cone of a map is a topological space that contains information about the map.
The mapping cone $C(p^t)$ comes with a long exact sequence
$$ ... \longrightarrow
H^k(X^t)\longrightarrow
H^{k+1}(C(p^t)) \longrightarrow
H^{k+1}(\Grass{d}{\R^m}) \longrightarrow
H^{k+1}(X^t) \longrightarrow 
...$$
from which we deduce the formula
$$ 
\mathrm{rank}(p^t)^*=
\sum_{k = 1}^{+ \infty}
(-1)^k \bigg( 
\dim H^k(X^t) - \dim H^{k+1}(C(p^t)) + \dim H^{k+1} (\Grass{d}{\R^m})
\bigg)
$$

On the simplicial side, is not complicated to build a triangulation $C^t$ of the mapping cone $C(p^t)$, nor to build a simplicial filtration $(C^t)_{t \in T}$ of $\left(C(p^t)\right)_{t \in T}$.
It relies on finding simplicial approximations to the maps $q^t\colon S^t \rightarrow L$, as in the proof of \citet[Theorem 2C.5]{Hatcher_Algebraic}. We point out that, just as in the previous algorithm, we may have to apply barycentric subdivisions to $S^t$ here.
Now, the previous formula translates in simplicial cohomology as
$$ 
\mathrm{rank}(q^t)^*=
\sum_{k = 1}^{+ \infty}
(-1)^k \bigg( 
\dim H^k(S^t) - \dim H^{k+1}(C^t) + \dim H^{k+1} (L)
\bigg).
$$
All these terms can be computed efficiently by applying the persistent homology algorithm to the filtrations $(S^t)_{t \in T}$, $\left(C(p^t)\right)_{t \in T}$ and $(L)_{t \in T}$.
Finally, we identify the value $\tdeatho$ as the first value of $t \in T$ such that $\mathrm{rank}(q^t)^*$ is nonzero.

\section{An algorithm when $d=1$}
\label{sec:algorithm}

Even though the last sections described a theoretical way to compute the persistent Stiefel-Whitney classes, some concrete issues are still to be discussed:
\begin{itemize}
\item verifying that the star condition is satisfied,
\item the Grassmann manifold has to be triangulated,
\item in practice, the Vietoris-Rips filtration is preferred to the \v{C}ech filtration,
\item the parameter $\gamma$ has to be tuned.
\end{itemize}
The following subsections will elucidate these points. 
Concerning the first one, we are not aware of a computational-explicit process to triangulate the Grassmann manifolds $\Grass{d}{\R^m}$, except when $d=1$, which corresponds to the projective spaces $\Grass{1}{\R^m}$.
We shall then restrict to the case $d=1$.
Note that, in this case, the only nonzero Stiefel-Whitney classes are the first two (by Axiom 1 of Stiefel-Whitney classes). Since $w_0$ is always equal to 1, the only class to estimate is $w_1$.

\subsection{The star condition in practice}
\label{subsec:combinatorial_star_condition}

Let us get back to the context of Subsect. \ref{background:simplicial_approx}: $K, L$ are two simplicial complexes, $K$ is finite, and $g \colon \topreal{K} \rightarrow \topreal{L}$ is a continuous map. 
We have seen that finding a simplicial approximation to $g$ reduces to finding a small enough barycentric subdivision $\subdiv{K}{n}$ of $K$ such that $g \colon \topreal{\subdiv{K}{n}} \rightarrow \topreal{L}$ satisfies the star condition, that is, 
for every vertex $v$ of $\subdiv{K}{n}$, there exists a vertex $w$ of $L$ such that 
\begin{equation*}
g\left(\topreal{\closedStar{v}}\right) \subseteq \topreal{\Star{w}}.
\end{equation*}
\noindent
In practice, one can compute the closed star $\closedStar{v}$ from the finite simplicial complex $\subdiv{K}{n}$. However, computing $g\left(\topreal{\closedStar{v}}\right)$ requires to evaluate $g$ on the infinite set $\topreal{\closedStar{v}}$. In order to reduce the problem to a finite number of evaluations of $g$, we shall consider a related property that we call the \emph{weak star condition}.

\begin{definition}
A map $g\colon \topreal{K} \rightarrow \topreal{L}$ between geometric realizations of simplicial complexes $K$ and $L$ satisfies the \emph{weak star condition} if for every vertex $v$ of $\subdiv{K}{n}$, there exists a vertex $w$ of $L$ such that 
\begin{equation*}
\topreal{ g \left( \skeleton{\closedStar{v}}{0} \right) }  \subseteq \topreal{\Star{w}},
\end{equation*}
where $\skeleton{\closedStar{v}}{0}$ denotes the 0-skeleton of $\closedStar{v}$, i.e. its vertices.
\end{definition}

Observe that the practical verification of the condition $\topreal{ g \left( \skeleton{\closedStar{v}}{0} \right) }  \subseteq \topreal{\Star{w}}$ requires only a finite number of computations. 
Indeed, one just has to check whether every neighbor $v'$ of $v$ in the graph $\skeleton{K}{1}$, $v$ included, satisfies $g(v') \in \topreal{\Star{w}}$. 
The following lemma rephrases this condition by using the face map $\facemapK{L}\colon \topreal{L} \rightarrow L$.
We remind the reader that the face map is defined by the relation $x \in \facemapK{L}(x)$ for all $x \in \topreal{L}$.
\begin{lemma}
\label{lem:weak_star}
The map $g$ satisfies the weak star condition if and only if for every vertex $v$ of $K$, there exists a vertex $w$ of $L$ such that for every neighbor $v'$ of $v$ in $\skeleton{K}{1}$, we have
\begin{equation*}
w \in \facemapK{L}(g(v')).
\end{equation*}
\end{lemma}

\begin{proof}
Let us show that the assertion ``$w \in \facemapK{L}(g(v'))$" is equivalent to ``$g(v') \in \topreal{\Star{w}}$".
Recall that the open star $\Star{w}$ consists of the simplices of $L$ that contain $w$. Moreover, the geometric realization $\topreal{\Star{w}}$ is the union of the $\topreal{\sigma}$ for $\sigma \in \Star{w}$.
As a consequence, $g(v')$ belongs to $\topreal{\Star{w}}$ if and only if it belongs to $\topreal{\sigma}$ for some simplex $\sigma \in L$ that contains $w$. Equivalently, the face map $\facemapK{L}(g(v'))$ contains $w$.
\end{proof}

Suppose that $g$ satisfies the weak star condition. 
Let $f\colon \skeleton{K}{0} \rightarrow \skeleton{L}{0}$ be a map, between vertex sets, such that for every $v \in \skeleton{K}{0}$,
\begin{equation*}
\topreal{ g \left( \skeleton{\closedStar{v}}{0} \right) }  \subseteq \topreal{\Star{f(v)}}.
\end{equation*}
According to the proof of Lemma \ref{lem:weak_star}, an equivalent formulation of this condition is: for all neighbor $v'$ of $v$ in $\skeleton{K}{1}$,
\begin{equation}
\label{eq:face_map_weakstarcondition}
f(v) \in \facemapK{L}(g(v')).
\end{equation}
Such a map is called a \emph{weak simplicial approximation to $g$}.
It plays a similar role as the simplicial approximations to $g$.

\begin{lemma}
\label{lem:combinatorial_is_simplicial}
If $f\colon \skeleton{K}{0} \rightarrow \skeleton{L}{0}$ is a weak simplicial approximation to $g\colon \topreal{K} \rightarrow \topreal{L}$, then $f$ is a simplicial map.
\end{lemma}

\begin{proof}
Let $\sigma = [v_0, ..., v_n]$ be a simplex of $K$. We have to show that $f(\sigma ) = [f(v_0), ..., f(v_n)]$ is a simplex of $L$.
Note that each closed star $\closedStar{v_i}$ contains $\sigma$.
Therefore each $\topreal{ g \left( \skeleton{\closedStar{v_i}}{0} \right) }$ contains $\topreal{ g \left( \skeleton{\sigma}{0} \right) } = \{g(v_0), ..., g(v_n)\}$.
Using the weak simplicial approximation property of $f$, we deduce that each $\topreal{\Star{f(v_i)}}$ contains $\{g(v_0), ..., g(v_n)\}$.
Using Lemma \ref{lem:intersection_stars} stated below, we obtain that $[f(v_0), ..., f(v_n)]$ is a simplex of $L$.
\end{proof}

\begin{lemma}[\citeauthor{Hatcher_Algebraic}, \citeyear{Hatcher_Algebraic}, Lemma 2C.2]
\label{lem:intersection_stars}
Let $w_0, ..., w_n$ be vertices of a simplicial complex $L$. Then $\bigcap_{i=0}^n \Star{w_i} \neq \emptyset$ if and only if $[w_0, ..., w_n]$ is a simplex of $L$.
\end{lemma}

As one can see from the definitions, the weak star condition is weaker than the star condition. Consequently, the simplicial approximation theorem admits the following corollary. 
\begin{corollary}
Consider two simplicial complexes $K, L$ with $K$ finite, and let $g\colon \topreal{K} \rightarrow \topreal{L}$ be a continuous map.
Then there exists $n \geq 0$ such that $g\colon \topreal{\subdiv{K}{n}} \rightarrow \topreal{L}$ satisfies the weak star condition.
\end{corollary}

\noindent
However, some weak simplicial approximations to $g$ may not be simplicial approximations, and may not even be homotopic to $g$.
Figure \ref{fig:12} gives such an example.

\begin{figure}[H]
\begin{minipage}{.32\linewidth}
\centering
\includegraphics[width=.6\linewidth]{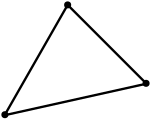}
\newline\noindent
$K$
\end{minipage}
\begin{minipage}{.32\linewidth}
\centering
\includegraphics[width=.7\linewidth]{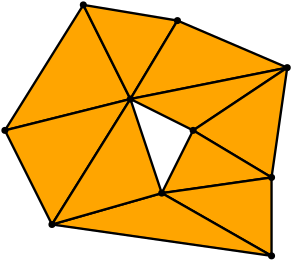}
\newline\noindent
$L$
\end{minipage}
\begin{minipage}{.32\linewidth}
\centering
\includegraphics[width=.7\linewidth]{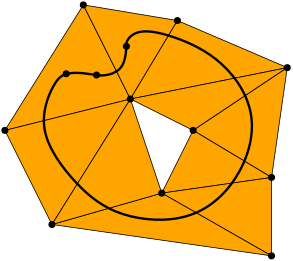}
\newline\noindent
$g\colon \topreal{K} \rightarrow \topreal{L}$
\end{minipage}
\caption{The map $g$ is non-trivial but admits a weak simplicial approximation which is constant.}
\label{fig:12}
\end{figure}

\noindent
Fortunately, these two notions coincides under the star condition assumption:

\begin{proposition}
\label{prop:weak_star_strong_star}
Suppose that $g$ satisfies the star condition.
Then every weak simplicial approximation to $g$ is a simplicial approximation.
\end{proposition}

\begin{proof}
Let $f$ be a weak simplicial approximation to $g$, and $f'$ any simplicial approximation. 
Let us show that $f$ and $f'$ are contiguous simplicial maps.
Let $\sigma = [v_0, ..., v_n]$ be a simplex of $K$. We have to show that $[f(v_0), ..., f(v_n), f'(v_0), ..., f'(v_n)]$ is a simplex of $L$.
As we have seen in the proof of Lemma \ref{lem:combinatorial_is_simplicial}, each $\topreal{ g \left( \skeleton{\closedStar{v_i}}{0} \right) }$ contains $\{g(v_0), ..., g(v_n)\}$.
Therefore, by definition of weak simplicial approximations and simplicial approximations, each $\topreal{\Star{f(v_i)}}$ and $\topreal{\Star{f'(v_i)}}$ contains $\{g(v_0), ..., g(v_n)\}$. We conclude by applying Lemma \ref{lem:intersection_stars}. 
\end{proof}

Remark that the proof of this proposition can be adapted to obtain the following fact: any two weak simplicial approximations are equivalent---as well as any two simplicial approximations.

Let us comment Proposition \ref{prop:weak_star_strong_star}.
If $K$ is subdivised enough, then every weak simplicial approximation to $g$ is homotopic to $g$.
We face the following problem in practice: the number of subdivisions needed by the star condition is not known.
In order to work around this problem, we propose to subdivise the complex $K$ until it satisfies the weak star condition, and then use a weak simplicial approximation to $g$.
However, such a weak simplicial approximation may not be homotopic to $g$, and our algorithm would output a wrong result.

To close this subsection, we state a lemma that gives a quantitative idea of the number of subdivisions needed by the star condition.
We say that a Lebesgue number for an open cover $\UU$ of a compact metric space $X$ is a positive number $\epsilon$ such that every subset of $X$ with diameter less than $\epsilon$ is included in some element of the cover $\UU$.

\begin{lemma}
Let $\topreal{K}, \topreal{L}$ be endowed with metrics.
Suppose that $g\colon \topreal{K} \rightarrow \topreal{L}$ is $l$-Lipschitz with respect to these metrics. Let $\epsilon$ be a Lebesgue number for the open cover $\left \{ \topreal{\Star{w}}, w \in L \right\}$ of $\topreal{L}$. 
Let $p$ be the dimension of $K$ and $D$ an upper bound on the diameter of its faces.
Then for $n > \log(\frac{Dl}{\epsilon})\big/\log(\frac{p+1}{p})$, the map $g\colon \topreal{\subdiv{K}{n}}\rightarrow \topreal{L}$ satisfies the star condition.
\end{lemma}

\begin{proof}
The map $g$ satisfies the star condition if for every vertex $v$ of $K$, there exists a vertex $w$ of $L$ such that $g(\topreal{\closedStar{v}}) \subseteq \topreal{\Star{w}}$. 
Since the cover $\left \{ \topreal{\Star{w}}, w \in L \right\}$ admits $\epsilon$ as a Lebesgue number, it is enough for $v$ to satisfy the following inequality:
\begin{equation}
\label{eq:proof_number_subdivision}
\diam{ g(\topreal{\closedStar{v}}) } < \epsilon.
\end{equation}
Since $g$ is $l$-Lipschitz, we have $\diam{ g\left(\topreal{\closedStar{v}} \right) } \leq l\cdot \diam{ \topreal{\closedStar{v}} }$. Using the hypothesis $\diam{ \topreal{\closedStar{v}} } \leq D$, Equation \eqref{eq:proof_number_subdivision} leads to the condition $D l < \epsilon$.
Now, we use the fact that a barycentric subdivision reduces the diameter of each face by a factor $\frac{p}{p+1}$. After $n$ barycentric subdivision, the last inequality rewrites $\left( \frac{p}{p+1} \right)^n D l < \epsilon$. It admits $n > \log(\frac{Dl}{\epsilon})\big/\log(\frac{p+1}{p})$ as a solution.
\end{proof}

\subsection{Triangulating the projective spaces}
As we described in Subsect. \ref{subsec:combinatorial_star_condition}, the algorithm we propose rests on a triangulation $L$ of the Grassmannian $\Grass{1}{\R^m}$, together with the map $\facemapK{L}  \circ h \colon \Grass{1}{\R^m} \rightarrow L$, where $h \colon \Grass{1}{\R^m} \rightarrow \topreal{L}$ is a homeomorphism and $\facemapK{L} \colon \Grass{1}{\R^m} \rightarrow L$ is the face map.
In the following, we also call $\facemap:= \facemapK{L}  \circ h$ the face map.
\medbreak
We shall use the triangulation of the projective space $\Grass{1}{\R^m}$ by \cite{von1987minimal}. It uses the fact that the quotient of the sphere $\S_{m-1}$ by the antipodal relation gives $\Grass{1}{\R^m}$.
Let $\Delta^{m}$ denote the standard $m$-simplex, $v_0, ..., v_m$ its vertices, and $\partial \Delta^m$ its boundary. 
The simplicial complex $\partial \Delta^m$ is a triangulation of the sphere $\mathbb{S}_{m-1}$.
Denote its first barycentric subdivision as $\subdiv{\partial \Delta^m}{1}$. The vertices of $\subdiv{\partial \Delta^m}{1}$ are in bijection with the non-empty proper subsets of $\{v_0, ..., v_m\}$.
Consider the equivalence relation on these vertices which associates a vertex to its complement. 
The quotient simplicial complex under this relation, $L$, is a triangulation of $\Grass{1}{\R^m}$.
Figures \ref{fig:13} and \ref{fig:14} represent this construction.

\begin{figure}[H]
\begin{minipage}{.23\linewidth}
\centering
\includegraphics[width=1\linewidth]{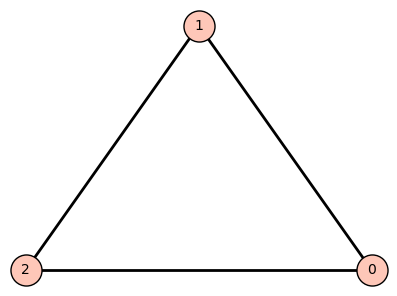}
$\partial \Delta^2$
\end{minipage}
\begin{minipage}{.23\linewidth}
\centering
\includegraphics[width=1\linewidth]{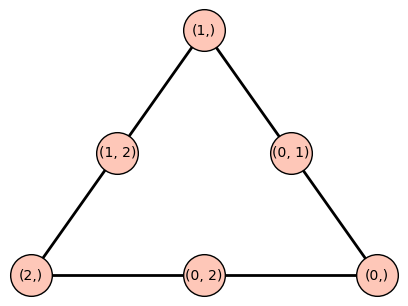}
$\subdiv{\partial \Delta^2}{1}$
\end{minipage}
\begin{minipage}{.23\linewidth}
\centering
\includegraphics[width=1\linewidth]{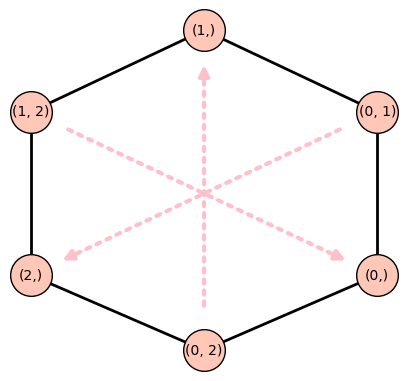}
Equivalence relation
\end{minipage}
\begin{minipage}{.23\linewidth}
\centering
\includegraphics[width=1\linewidth]{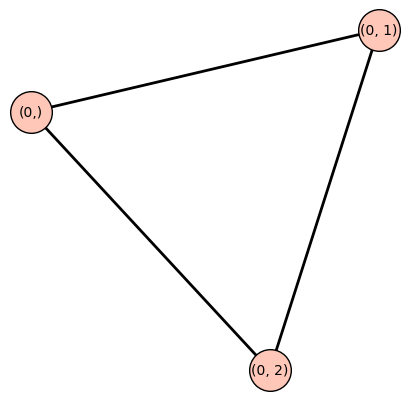}
Quotient complex $L$
\end{minipage}
\caption{Triangulating $\Grass{1}{\R^2}$.}
\label{fig:13}
\end{figure}

Let us now describe how to define the homeomorphism $h\colon \Grass{1}{\R^m} \rightarrow \topreal{L}$.
First, embed $\Delta^m$ in $\R^{m+1}$ via $v_i \mapsto (0, ..., 0, 1, 0, ...)$, where $1$ sits at the $i^\text{th}$ coordinate. 
Its image lies on a $m$-dimensional affine subspace $P$, with origin being the barycenter of $v_0, ..., v_m$. Seen in $P$, the vertices of $\Delta^m$ now belong to the sphere centered at the origin and of radius $\sqrt{\frac{m}{m+1}}$ (see Figure \ref{fig:14}). Let us denote this sphere as $\mathbb{S}_{m-1}$.
Next, subdivise barycentrically $\partial \Delta^m$ once, and project each vertex of $\subdiv{\partial \Delta^m}{1}$ on $\mathbb{S}_{m-1}$. 
By taking the convex hulls of its faces, we now see $\topreal{\subdiv{\partial \Delta^m}{1}}$ as a subset of $P$.
Define an application $p\colon \mathbb{S}_{m-1} \rightarrow \topreal{\subdiv{\partial \Delta^m}{1}}$ as follows: for every $x \in \mathbb{S}_{m-1}$, the image $p(x)$ is the unique intersection point between the segment $[0, x]$ and the set $\topreal{\subdiv{\partial \Delta^m}{1}}$.
The application $p$ can also be seen as the inverse function of the projection on $\S_{m-1}$, written $\text{proj}_{\S_{m-1}}\colon \topreal{\subdiv{\partial \Delta^m}{1}} \rightarrow \mathbb{S}_{m-1}$.
As a consequence, we can factorize  $p\colon \mathbb{S}_{m-1} \rightarrow \topreal{\subdiv{\partial \Delta^m}{1}}$ as
\begin{align*}
h\colon \left(\mathbb{S}_{m-1} / {\sim} \right) \rightarrow \left(\topreal{\subdiv{\partial \Delta^m}{1}} / {\sim} \right).
\end{align*}
Using Lemma \ref{lem:identification} stated below, we can identify these spaces with 
\begin{align*}
h\colon \Grass{1}{\R^m}  \rightarrow \topreal{L},
\end{align*}
giving the desired triangulation.

\begin{figure}[H]
\begin{minipage}{.32\linewidth}
\centering
\includegraphics[width=.7\linewidth]{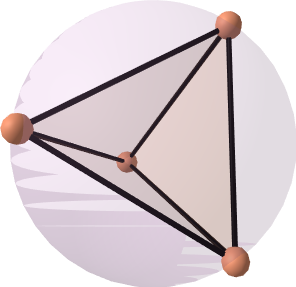}
\newline\noindent
$\partial \Delta^3$ is included in $\mathbb{S}_{m-1}$
\end{minipage}
\begin{minipage}{.32\linewidth}
\centering
\includegraphics[width=.7\linewidth]{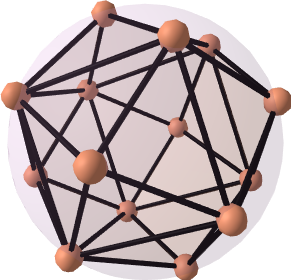}
\newline\noindent
$\subdiv{\partial \Delta^3}{1}$ and $\mathbb{S}_{m-1}$
\end{minipage}
\begin{minipage}{.32\linewidth}
\centering
\includegraphics[width=.7\linewidth]{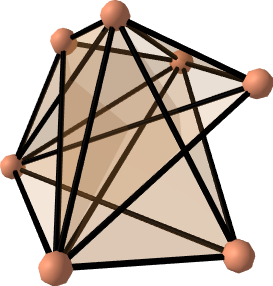}
\newline\noindent
$L$
\end{minipage}
\caption{Triangulating $\Grass{1}{\R^3}$.}
\label{fig:14}
\end{figure}

\begin{lemma}
For any vertex $x \in \subdiv{\partial \Delta^m}{1}$, denote by $\topreal{x}$ its embedding in $P$. 
Let $-\topreal{x}$ denote the image of $\topreal{x}$ by the antipodal relation on $\S_{m-1}$. Denote by $y$ the image of $x$ by the relation on $\subdiv{\partial \Delta^m}{1}$. Then $y = -\topreal{x}$.

More generally, pulling back the antipodal relation onto $\topreal{\subdiv{\partial \Delta^m}{1}}$ via $p$ gives the relation we defined on $\subdiv{\partial \Delta^m}{1}$.
\label{lem:identification}
\end{lemma}

\begin{proof}
Pick a vertex $x$ of $\subdiv{\partial \Delta^m}{1}$. 
It can be described as a proper subset $\{v_i, i \in I\}$ of the vertex set $\skeleton{(\partial \Delta^m)}{0} = \{v_0, ..., v_m\}$, where $I \subset \llbracket0,m\rrbracket$.
According to the relation on $(\partial \Delta^m)$, the vertex $x$ is in relation with the vertex $y$ described by the proper subset $\{v_i, i \in \complementaire{I}\}$. 
The point $x$ can be written in barycentric coordinates as $\frac{1}{\card{I}} \sum_{i \in I} \topreal{v_i}$.
Seen in $P$, $\topreal{x}$ can be written $\topreal{x} = \projj{\sum_{i \in I} v_i}{\S_{m-1}}$.
Similarly, $\topreal{y}$ can be written $\topreal{y} = \projj{\sum_{i \in \complementaire{I}} v_i}{\S_{m-1}}$.

Now, denote by 0 the origin of the hyperplane $P$, and embed the vertices $v_0, ..., v_m$ in $P$. Observe that
\begin{align*}
0 
= \sum_{i \leq 0} v_i
= \sum_{i \in I} v_i + \sum_{i \in \complementaire{I}} v_i.
\end{align*}
Hence $- \sum_{i \in I} v_i = \sum_{i \in \complementaire{I}} v_i$, and we deduce that 
\begin{align*}
-\topreal{x} 
= \projj{- \sum_{i \in I} v_i}{\S_{m-1}}
=  \projj{\sum_{i \in \complementaire{I}} v_i}{\S_{m-1}}
= \topreal{y}.
\end{align*}
\noindent
Applying the same reasoning, one obtains the following result: for every simplex $\sigma$ of $\subdiv{\partial \Delta^m}{1}$, if $\nu$ denotes the image of $\sigma$ by the relation on $\subdiv{\partial \Delta^m}{1}$, then the image of $\topreal{\sigma}$ by the antipodal relation is also $\topreal{\nu}$.
As a consequence, these two relations coincide.
\end{proof}

At a computational level, let us describe how to compute the face map $\facemap \colon \Grass{1}{\R^m} \rightarrow L$.
Since $\facemap$ can be obtained as a quotient, it is enough to compute the face map of the sphere, $\facemap' \colon \mathbb{S}_{m-1} \rightarrow \subdiv{\partial \Delta^m}{1}$, which corresponds to the homeomorphism $p\colon \S_{m-1} \rightarrow \topreal{\subdiv{\partial \Delta^m}{1}}$. It is given by the following lemma, which can be used in practice.

\begin{lemma}
For every $x \in \S_{m-1}$, the image of $x$ by $\facemap'$ is equal to the intersection of all maximal faces $\sigma = [w_0, ..., w_m]$ of $\subdiv{\partial \Delta^m}{1}$ that satisfies the following conditions: denoting by $x_0$ any point of the affine hyperplane spanned by $\{w_0, ..., w_m\}$, and by $h$ a vector orthogonal to the corresponding linear hyperplane, 
\begin{itemize}
\item the inner product $\eucP{x}{h}$ has the same sign as $\eucP{x_0}{h}$,
\item the point $\frac{\eucP{x_0}{h} }{\eucP{x}{h}} x$, which is included in the affine hyperplane spanned by $\{w_0, ..., w_m\}$, has nonnegative barycentric coordinates.
\end{itemize}
\end{lemma}

\begin{proof}
Recall that for every $x \in \mathbb{S}_{m-1}$, the image $p(x)$ is defined as the unique intersection point between the segment $[0, x]$ and the set $\topreal{\subdiv{\partial \Delta^m}{1}}$. 
Besides, the face map $\facemap'(x)$ is the unique simplex $\sigma \in \subdiv{\partial \Delta^m}{1}$ such that $p(x) \in \topreal{\sigma}$.
Equivalently, $\facemap'(x)$ is equal to the intersection of all maximal faces $\sigma \in \subdiv{\partial \Delta^m}{1}$ such that $p(x)$ belongs to the closure $\overline{ \topreal{\sigma} }$.

Consider any maximal face $\sigma = [w_0, ..., w_m]$ of $\subdiv{\partial \Delta^m}{1}$.
The first condition of the lemma ensures that the segment $[0,x]$ intersects the affine hyperplane spanned by $\{w_0, ..., w_m\}$.
In this case, a computation shows that this intersection consists of the point $\frac{\eucP{x_0}{h} }{\eucP{x}{h}} x$.
Then, the second condition of the lemma tests whether this point belongs to the convex hull of $\{w_0, ..., w_k\}$.
In conclusion, if $\sigma$ satisfies these two conditions, then $p(x) \in \overline{ \topreal{\sigma} }$.
\end{proof}

As a remark, let us point out that the verification of the conditions of this lemma is subject to numerical errors. In particular, the point $\frac{\eucP{x_0}{h} }{\eucP{x}{h}} x$ may have nonnegative coordinates, yet mathematical softwares may return (small) negative values.
Consequently, the algorithm may recognize less maximal faces that satisfy these conditions, hence return a simplex that strictly contains the wanted simplex $\facemap'(x)$. 
Nonetheless, such an error will not affect the output of the algorithm.
Indeed, if we denote by $\widetilde{\facemap'}$ the face map computed by the algorithm, we have that $\facemap'(x) \subseteq \widetilde{\facemap'}(x)$ for all $x \in \S_{m-1}$. As a consequence of Lemma \ref{lem:weak_star}, $\widetilde{\facemap'}$ satisfies the weak star condition if $\facemap'$ does, and Equation \eqref{eq:face_map_weakstarcondition} shows that every weak simplicial approximations for $\facemap'$ are weak simplicial approximations for $\widetilde{\facemap'}$.
Since every weak simplicial approximations are homotopic, we obtain that the induced maps in cohomology are equal, therefore the output of the algorithm is unchanged.

\subsection{Vietoris-Rips version of the \v{C}ech bundle filtration}
\label{subsec:Rips}
We still consider a subset $X \subset \R^n \times \matrixspace{\R^m}$. Denote by $\X$ the corresponding \v{C}ech set filtration, and by $\S = (S^t)_{t\geq0}$ the simplicial \v{C}ech filtration. 
For every $t \geq 0$, let $R^t$ be the flag complex of $S^t$, i.e. the clique complex of the 1-skeleton $\skeleton{(S^t)}{1}$ of $S^t$. It is known as the \emph{Vietoris-Rips complex} of $X$ at time $t$.
The collection $\R = (R^t)_{t \geq 0}$ is called the \emph{Vietoris-Rips filtration} of $X$.
The simplicial filtrations $\S$ and $\R$ are multiplicatively $\sqrt{2}$-interleaved \cite[Theorem 3.1]{bell2017weighted}. In other words, for every $t \geq 0$, we have 
\begin{align*}
S^t \subseteq R^t \subseteq S^{\sqrt{2} t}.
\end{align*}

Let $\gamma>0$ and consider the \v{C}ech bundle filtration $(\X, \p)$ of $X$. Suppose that its maximal filtration value $\tmaxgamma{X}$ is positive.
Let $\topreal{\R} = (\topreal{R^t})_{t\geq 0}$ denote the geometric realization of the Vietoris-Rips filtration. 
We can give $\topreal{\R}$ a vector bundle filtration structure with $(p')^t \colon \topreal{R^t} \rightarrow \Grass{d}{\R^m}$ defined as
\begin{equation*}
(p')^t = p^{\sqrt{2}t} \circ i^t,
\end{equation*}
where $p^{\sqrt{2}t}$ denotes the maps of the \v{C}ech bundle filtration $(\X, \p)$, and $i^t$ denotes the inclusion $\topreal{R^t} \hookrightarrow \topreal{S^{\sqrt{2}t}}$. These maps fit in the following diagram:
\begin{equation*}
\begin{tikzcd}
\topreal{R^t}  \arrow[r, hook, "i^t"] \arrow[drr, "(p')^t", dashed, swap]
&[0em] \topreal{S^{\sqrt{2} t}} \arrow[r, equal, shorten <= 1em, shorten >= 1em]
& X^{\sqrt{2} t} \arrow[d, "p^{\sqrt{2} t}"]  \\
&& \Grass{d}{\R^m}
\end{tikzcd}
\end{equation*}
\noindent
This new vector bundle filtration is defined on the index set $T' = \left[0, \frac{1}{\sqrt{2}}\tmaxgamma{X}\right)$.

It is clear from the construction that the vector bundle filtrations $(\X, \p)$ and $(\topreal{\R}, \p')$ are multiplicatively $\sqrt{2}$-interleaved, with an interleaving that preserves the persistent Stiefel-Whitney classes. This property is a multiplicative equivalent of Theorem \ref{thm:stability}.
Remember that if $X$ is a subset of $\R^n \times \Grass{d}{\R^m}$, then the maximal filtration value of the \v{C}ech bundle filtration on $X$ is $\tmaxgamma{X} = \frac{\sqrt{2}}{2} \gamma$ (see Equation \eqref{eq:tmax_subset_grass}). Consequently, the maximal filtration value of its Vietoris-Rips version is $\frac{1}{2}\gamma$.

From an application perspective, we choose to work with the Vietoris-Rips filtration since it is easier to compute. Indeed, its construction only relies on computing pairwise distances, and finding cliques in graphs.

\subsection{Choice of the parameter $\gamma$}
\label{subsec:choice_of_gamma}
This subsection is devoted to discussing the influence of the parameter $\gamma>0$.
Recall that $\gamma$ affects the norm $\gammaN{\cdot}$ we chose on $\R^n \times \matrixspace{\R^m}$:
\begin{align*}
\gammaN{(x,A)}^2 = \eucN{x}^2 + \gamma^2 \frobN{A}^2.
\end{align*}
Let $X \subset \R^n \times \matrixspace{\R^m}$.
If $\gamma_1 \leq \gamma_2$ are two positive real numbers, the corresponding \v{C}ech filtrations $\X_1$ and $\X_2$, as well as the \v{C}ech bundle filtrations $(\X_1, \p_1)$ and $(\X_2, \p_2)$, are $\frac{\gamma_2}{\gamma_1}$-interleaved multiplicatively.
This comes from the straightforward inequality
\begin{align*}
\|\cdot\|_{\gamma_1} 
~\leq~ \|\cdot\|_{\gamma_2} 
~\leq~ \frac{\gamma_2}{\gamma_1} \|\cdot\|_{\gamma_1}.
\end{align*}
Note that we also have the additive inequality 
\begin{align*}
\| (x,A) \|_{\gamma_1} 
~\leq~ \| (x,A) \|_{\gamma_2} 
~\leq~  \| (x,A) \|_{\gamma_1} + \sqrt{\gamma_2^2-\gamma_1^2} \frobN{A}.
\end{align*}
One deduces that the \v{C}ech bundle filtrations $(\X_1, \p_1)$ and $(\X_2, \p_2)$ are $\sqrt{\gamma_2^2 - \gamma_1^2} \cdot \tmax{X}$-interleaved additively, where $\tmax{X}$ is the maximal filtration value when $\gamma=1$.
As a consequence of these interleavings, when the values $\gamma_1$ and $\gamma_2$ are close, the persistence barcodes and the lifebars of the persistent Stiefel-Whitney classes are close (see Theorem \ref{thm:stability}).

As a general principle, one would choose the parameter $\gamma$ to be large, since it would lead to large filtrations. More precisely, if $t_{\gamma_1}^{\mathrm{max}}\left( X \right)$ and $t_{\gamma_2}^{\mathrm{max}}\left( X \right)$ denote respectively the maximal filtration values of $(\X_1,\p_1)$ and $(\X_2,\p_2)$, then $t_{\gamma_1}^{\mathrm{max}}\left( X \right) = \gamma_1 \cdot \tmax{X}$ and $t_{\gamma_2}^{\mathrm{max}}\left( X \right) = \gamma_2 \cdot \tmax{X}$, as in Equation \eqref{eq:tmax}. Moreover, we have the following inclusion:
\begin{align*}
X_1^{t_{\gamma_1}^{\mathrm{max}}\left( X \right)}
\subseteq X_2^{t_{\gamma_2}^{\mathrm{max}}\left( X \right)},
\end{align*}
where $X_1^{t_{\gamma_1}^{\mathrm{max}}\left( X \right)}$ denotes the thickening of $X$ with respect to the norm $\|\cdot\|_{\gamma_1}$, and $X_2^{t_{\gamma_2}^{\mathrm{max}}\left( X \right)}$ with respect to $\|\cdot\|_{\gamma_2}$.
This inclusion can be proven from the following fact, valid for every $x \in \R^n$ and $A \in \matrixspace{\R^m}$ such that $\frobN{A} \leq \tmax{X}$:
\begin{align*}
\| (x,A) \|_{\gamma_1} \leq t_{\gamma_1}^{\mathrm{max}}\left( X \right) 
\implies
\| (x,A) \|_{\gamma_2} \leq t_{\gamma_2}^{\mathrm{max}}\left( X \right).
\end{align*}

\noindent
Hence larger parameters $\gamma$ lead to larger maximal filtration values and larger filtrations.

However, as we show in the following examples, different values of $\gamma$ may result in different behaviours of the persistent Stiefel-Whitney classes.
In Example \ref{ex:gamma_normal}, large values of $\gamma$ highlight properties of the dataset that are not consistent with the underlying vector bundle, which is orientable. 
Notice that, so far, we always picked the value $\gamma = 1$, for it seemed experimentally relevant with the datasets we chose.

\begin{example}
\label{ex:gamma_mobius}
Consider the set $Y \subset \R^2 \times \matrixspace{\R^2}$ representing the Mobius strip, as in Example \ref{ex:normal_mobius} of Subsect. \ref{subsec:filtered_cech_bundle}:
\begin{align*}
&Y = \bigg\{
\bigg( 
\begin{pmatrix}
\cos(\theta)  \\
\sin(\theta) 
\end{pmatrix}
,
\begin{pmatrix}
\cos(\frac{\theta}{2})^2 & \cos(\frac{\theta}{2}) \sin(\frac{\theta}{2}) \\
\cos(\frac{\theta}{2}) \sin(\frac{\theta}{2}) & \sin(\frac{\theta}{2})^2 
\end{pmatrix}
\bigg),
\theta \in [0, 2\pi)  \bigg\}.
\end{align*}
As we show in Appendix \ref{subsec:gamma_mobius}, $Y$ is a circle, included in a 2-dimensional affine subspace of $\R^2 \times \matrixspace{\R^2}$.
Its radius is $\sqrt{1 + \frac{\gamma^2}{2}}$.
As a consequence, the persistence of the \v{C}ech filtration of $Y$ consists of two bars: one $H^0$-feature, the bar $[0, +\infty)$, and one $H^1$-feature, the bar $\left[0, \sqrt{1 + \frac{\gamma^2}{2}}\right)$.

For any $\gamma>0$, the maximal filtration value of the \v{C}ech bundle filtration of $Y$ is $\tmaxgamma{Y} = \frac{\sqrt{2}}{2} \gamma$.
Moreover, the persistent Stiefel-Whitney class $w_1^t(Y)$ is nonzero all along the filtration.
In this example, we see that the parameter $\gamma$ does not influence the qualitative interpretation of the persistent Stiefel-Whitney class. It is always nonzero where it is defined.
The following example shows a case where $\gamma$ does influence the persistent Stiefel-Whitney class.
\end{example}

\begin{example}
\label{ex:gamma_normal}
Consider the set $X \subset \R^2 \times \matrixspace{\R^2}$ representing the normal bundle of the circle $\S_1$, as in Example \ref{ex:normal_mobius}:
\begin{align*}
&X = \bigg\{
\bigg( 
\begin{pmatrix}
\cos(\theta)  \\
\sin(\theta) 
\end{pmatrix}
,
\begin{pmatrix}
\cos(\theta)^2 & \cos(\theta) \sin(\theta) \\
\cos(\theta) \sin(\theta) & \sin(\theta)^2 
\end{pmatrix}
\bigg),
\theta \in [0, 2\pi)  \bigg\}.
\end{align*}
As we show in Appendix \ref{subsec:gamma_normal}, $X$ is a subset of a 2-dimensional flat torus embedded in $\R^2 \times \matrixspace{\R^2}$, hence can be seen as a torus knot.

Before studying the \v{C}ech bundle filtration of $X$, we discuss its \v{C}ech filtration $\X$. Its behaviour depends on $\gamma$:
\begin{itemize}
\item if $\gamma \leq \frac{\sqrt{2}}{2}$, then $X^t$ retracts on a circle for $t \in [0,1)$, $X^t$ retracts on a 3-sphere for $t \in \left [ 1, \sqrt{1+\frac{1}{2}\gamma^2} \right)$, and $X^t$ retracts on a point for $t \geq \sqrt{1+\frac{1}{2}\gamma^2}$.
\item if $\gamma \geq \frac{\sqrt{2}}{2}$, then $X^t$ retracts on a circle for $t \in [0,1)$, $X^t$ retracts on another circle for $t \in \left[1,\frac{\sqrt{2} }{2}\sqrt{1 + \gamma^2 + \frac{1}{4 \gamma^2}}\right)$, $X^t$ retracts on a 3-sphere for $t \in \left[ \frac{\sqrt{2} }{2}\sqrt{1 + \gamma^2 + \frac{1}{4 \gamma^2}}, \sqrt{1+\frac{1}{2}\gamma^2}\right)$, and $X^t$ has the homotopy type of a point for $t \geq \sqrt{1+\frac{1}{2}\gamma^2}$.
\end{itemize}
\noindent
Let us interpret these facts. If  $\gamma \leq \frac{\sqrt{2}}{2}$, then the persistent cohomology of $X$ looks similar to the persistent cohomology of the underlying set $\left\{\begin{psmallmatrix} \cos(\theta) \\ \sin(\theta)  \end{psmallmatrix}, \theta \in [0, 2\pi)\right\} \subset \R^2$, but with a $H^3$ cohomology feature added. 
Besides, if  $\gamma \geq \frac{\sqrt{2}}{2}$, a new topological feature appears in the $H^1$-barcode: the bar $\left[1,\sqrt{2} \sqrt{1 + \gamma^2 + \frac{1}{4 \gamma^2}}\right)$.
These barcodes are depicted in Figures \ref{fig:15} and \ref{fig:16}.

Let us now discuss the corresponding \v{C}ech bundle filtrations.
For any $\gamma>0$, the maximal filtration value of the \v{C}ech bundle filtration of $X$ is $\tmaxgamma{X} = \frac{\sqrt{2}}{2} \gamma$.
We observe two behaviours: if $\gamma \leq \frac{\sqrt{2}}{2}$, then $w_1^t(X)$ is zero all along the filtration, and if $\gamma > \frac{\sqrt{2}}{2}$, then $w_1^t(X)$ is nonzero from $t^\dagger = 1$.
%\begin{itemize}
%\itemsep0.2em 
%\item if $\gamma \leq \frac{\sqrt{2}}{2}$, then $w_1^t(X)$ is zero all along the filtration,
%\item if $\gamma > \frac{\sqrt{2}}{2}$, then $w_1^t(X)$ is nonzero from $t^\dagger = 1$.
%\end{itemize}
This in proven in Appendix \ref{subsec:gamma_normal}.
To conclude, this persistent Stiefel-Whitney class is consistent with the underlying bundle---the normal bundle of the circle, which is trivial---only for $t\leq 1$.

\begin{figure}[H]
\centering
\begin{minipage}{.49\linewidth}
\centering
\includegraphics[width=.9\linewidth]{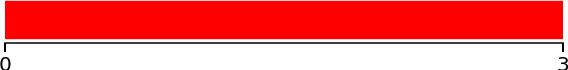}\vspace{.4cm}
\includegraphics[width=.9\linewidth]{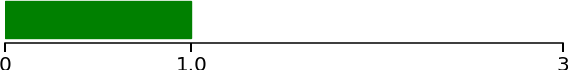}\vspace{.4cm}
\includegraphics[width=.9\linewidth]{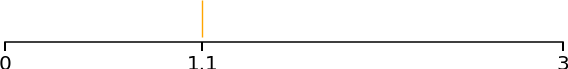}\vspace{.4cm}
\includegraphics[width=.9\linewidth]{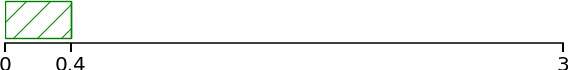}
\end{minipage}
\begin{minipage}{.49\linewidth}
\centering
\includegraphics[width=.9\linewidth]{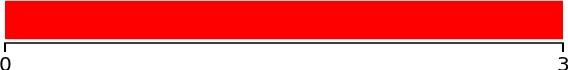}\vspace{.4cm}
\includegraphics[width=.9\linewidth]{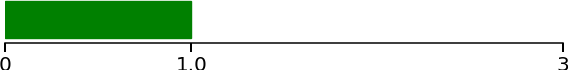}\vspace{.4cm}
\includegraphics[width=.9\linewidth]{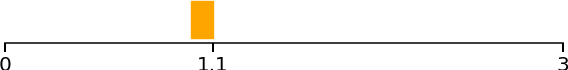}\vspace{.4cm}
\includegraphics[width=.9\linewidth]{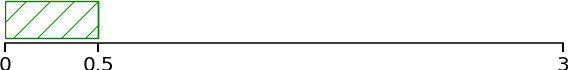}
\end{minipage}
\caption{$H^0$-, $H^1$-, $H^3$-barcodes and lifebar of the first persistent Stiefel-Whitney class of $X$ with $\gamma=\frac{1}{2}$ (left) and $\gamma=\frac{\sqrt{2}}{2}$ (right).}
\label{fig:15}
\end{figure}

\begin{figure}[H]
\centering
\begin{minipage}{.49\linewidth}
\centering
\includegraphics[width=.9\linewidth]{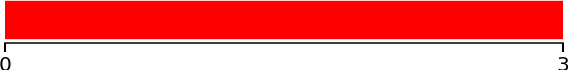}\vspace{.4cm}
\includegraphics[width=.9\linewidth]{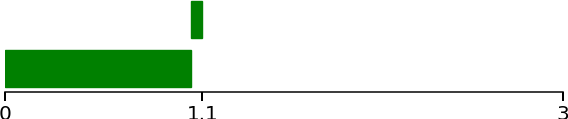}\vspace{.4cm}
\includegraphics[width=.9\linewidth]{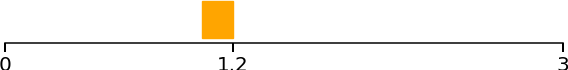}\vspace{.4cm}
\includegraphics[width=.9\linewidth]{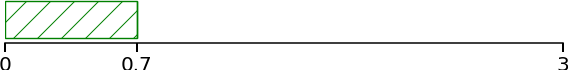}
\end{minipage}
\begin{minipage}{.49\linewidth}
\centering
\includegraphics[width=.9\linewidth]{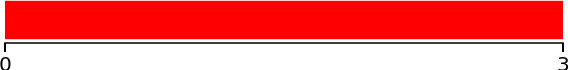}\vspace{.4cm}
\includegraphics[width=.9\linewidth]{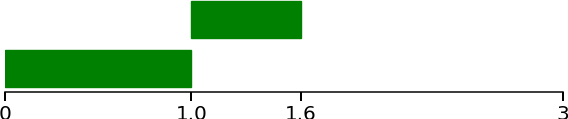}\vspace{.4cm}
\includegraphics[width=.9\linewidth]{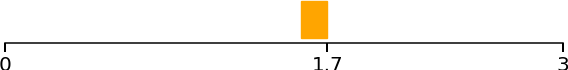}\vspace{.4cm}
\includegraphics[width=.9\linewidth]{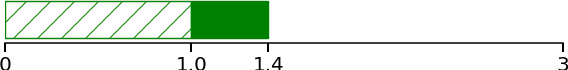}
\end{minipage}
\caption{$H^0$-, $H^1$-, $H^3$-barcodes and lifebar of the first persistent Stiefel-Whitney class of $X$ with $\gamma=1$ (left) and $\gamma=2$ (right).}
\label{fig:16}
\end{figure}
\end{example}

\section{Numerical experiments}
\label{sec:experiments}
In this last section, we propose to apply our algorithm on synthetic datasets, inspired by image analysis.
We aim at illustrating how the first persistent Stiefel-Whintey class may reveal two properties: when the datasets contain certain symmetries, and when the datasets are close to non-orientables vector bundles.
A Python notebook gathering these experiments can be found at \url{https://github.com/raphaeltinarrage/PersistentCharacteristicClasses/blob/master/Experiments.ipynb}.

\subsection{Datasets with symmetries}
\paragraph{Giraffe moving forward.}
Let us start with a simple dataset: it consists in a picture of a giraffe, that we translate to the right via cyclic permutations. 
The dataset contains of 150 images, each of size $150\times300$ pixels, in RGB format. 
Since $150\times300\times3 = 135\,000$, the dataset can be enbedded in $\R^{135\,000}$.
Some of these images can be seen in Figure \ref{fig:22}.

%\begin{figure}[H]
%\centering
%\begin{minipage}{.24\linewidth}
%\centering
%\includegraphics[width=1\linewidth]{Fig22-1.png}
%\end{minipage}
%\begin{minipage}{.24\linewidth}
%\centering
%\includegraphics[width=1\linewidth]{Fig22-2.png}
%\end{minipage}
%\begin{minipage}{.24\linewidth}
%\centering
%\includegraphics[width=1\linewidth]{Fig22-3.png}
%\end{minipage}
%\begin{minipage}{.24\linewidth}
%\centering
%\includegraphics[width=1\linewidth]{Fig22-4.png}
%\end{minipage}
%\caption{Some of the images...}
%\label{fig:22}
%\end{figure}

\begin{figure}[H]
\centering
\begin{minipage}{.3\linewidth}
\centering
\includegraphics[width=1\linewidth]{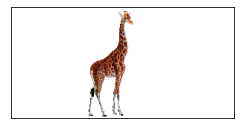}
\end{minipage}
~~~~
\begin{minipage}{.3\linewidth}
\centering
\includegraphics[width=1\linewidth]{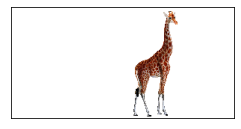}
\end{minipage}
~~~~
\begin{minipage}{.3\linewidth}
\centering
\includegraphics[width=1\linewidth]{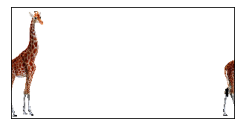}
\end{minipage}
\caption{The dataset consists in a giraffe moving forward.}
\label{fig:22}
\end{figure}

By performing a Principal Component Analysis (PCA), we project the dataset on the three principal axes. The result is a subset $X = \{x_1, ..., x_{150}\}$ of $\R^{3}$. As a last pre-processing step, we divide each point of $X$ by the value $\max\{\eucN{x}, ~x \in X\}$, so that $X$ becomes a subset of the unit ball $\closedball{0}{1} \subset \R^3$.
The point cloud $X$ is represented on Figure \ref{fig:23}. 
Next to it, we give the persistence barcodes of its Rips filtration (we choose the coefficient field $\Zd$, and represent $H^0$ in red and $H^1$ in green).
Note that $X$ actually lies close to the unit sphere $\S_2$ of $\R^3$, and moreover that it describes a circle. %, that we denote $S$.

\begin{figure}[H]
\centering
\begin{minipage}{.49\linewidth}
\centering
\includegraphics[width=.7\linewidth]{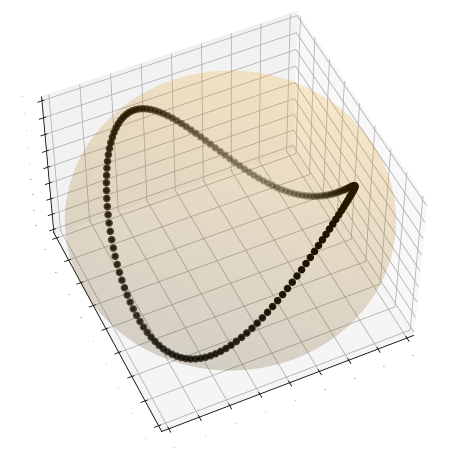}
\end{minipage}
\begin{minipage}{.49\linewidth}
\centering
\includegraphics[width=1\linewidth]{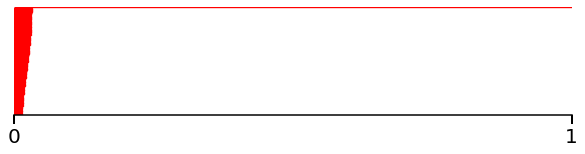}\\\includegraphics[width=1\linewidth]{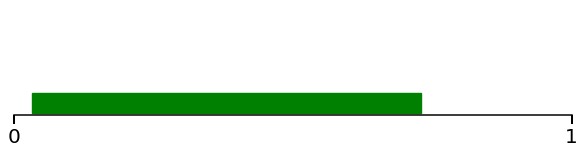}
\end{minipage}
\caption{\textbf{Left:} The point cloud $X \subset \R^3$.
\textbf{Right:} The barcode of its Rips filtration.}
\label{fig:23}
\end{figure}

Let us consider on $X$ two $1$-dimensional bundles, defined via classifying maps $X \rightarrow \Grass{1}{\R^3} \subset \matrixspace{\R^m}$: 
%\begin{itemize}
%\item $\xi\colon x_i \mapsto P(x_{i+1}-x_{i-1})$
%\item $\xi'\colon x_i \mapsto P(x_i)$
%\end{itemize}
\begin{center}
$p\colon x_i \mapsto P(x_i)$
~~~~~~~~~~
and
~~~~~~~~~~
$p'\colon x_i \mapsto P(x_{i+1}-x_{i-1})$
\end{center}
where $P(x)$ represents the orthogonal projection matrix on the 1-dimensional subspace of $\R^3$ spanned by $x$. 
The vector bundle $p$ is to be seen as the normal bundle of $\S_2$ restricted to $X$, and $p'$ is to be seen as the tangent bundle of the circle.
These two theoretical bundles --- restriction of the normal bundle of the sphere, and tangent bundle of a circle --- are trivial. 
This follows from the general fact that 1-dimensional bundles on the sphere are trivial. %Indeed, a 1-dimensional bundle is trivial if and only if its first Stiefel-Whitney class is zero, which is always the case since the first cohomology group $H^1(\S_2)$ is zero.

We represent on Figure \ref{fig:24-1} the vector bundles $p$ and $p'$, seen in $\R^3$.
One observes that, while going around the circle, the lines make a complete twist.
This is the same behavior as the trivial bundle of the circle, that we studied earlier (see Figure \ref{fig:2}).

\begin{figure}[H]
\centering
\begin{minipage}{.49\linewidth}
\centering
\includegraphics[width=.75\linewidth]{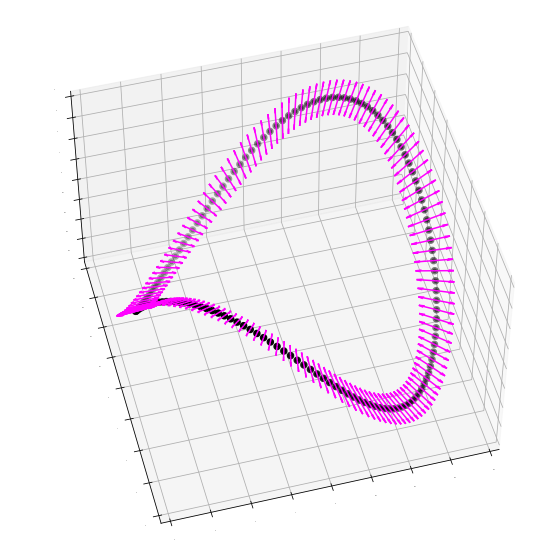}
\end{minipage}
\begin{minipage}{.49\linewidth}
\centering
\includegraphics[width=.75\linewidth]{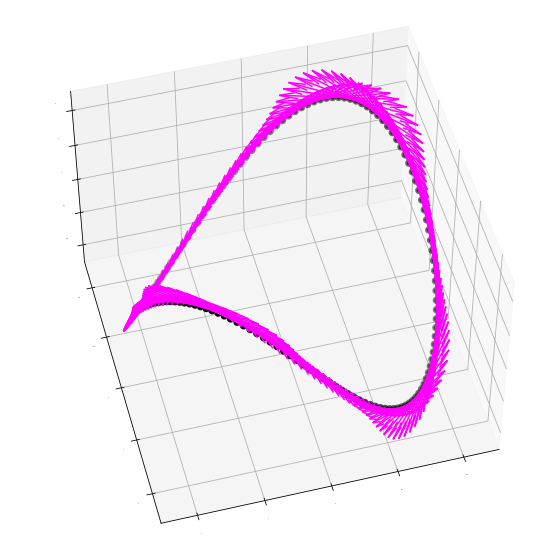}
\end{minipage}
\caption{The vector bundles $p$ and $p'$.}
\label{fig:24-1}
\end{figure}

Next, let $\gamma = 1$, and consider the lifted sets $\check X = \left\{\left(x,~\gamma p(x)\right), ~x\in X\right\}$ and $\check X' = \left\{\left(x,~\gamma p'(x)\right), ~x\in X\right\}$.
We remind the reader that this construction has been studied in Subsection \ref{subsec:consistency}.
%The parameter $\gamma$ has been chosen equal to the death time of the 
We represent the point clouds $\check X$ and $\check X'$ on Figure \ref{fig:24-2}, projected in $\R^3$ via PCA, as well as the persistence barcodes of their Rips filtration. On these two barcodes, one observes one prominent $H^0$-feature, and one prominent $H^1$-feature. Below, we plot the lifebars of their first persistent Stiefel-Whitney classes, up to the maximal filtration value $\frac{1}{2}\gamma$ (see Subsection \ref{subsec:Rips}). Both are empty, meaning that the persistent Stiefel-Whitney classes are zero all along the filtration. This is consistent with the fact that the underlying vector bundles are trivial.

\begin{figure}[H]
\centering
\begin{minipage}{.49\linewidth}
\centering
\includegraphics[width=.7\linewidth]{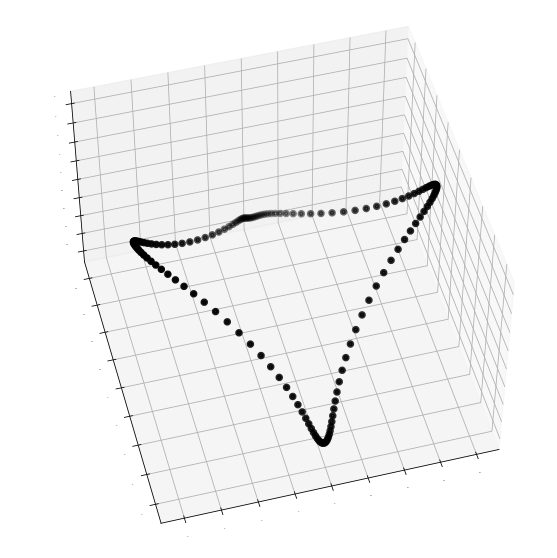}\\
\includegraphics[width=1\linewidth]{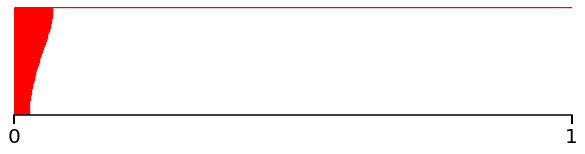}\\
\includegraphics[width=1\linewidth]{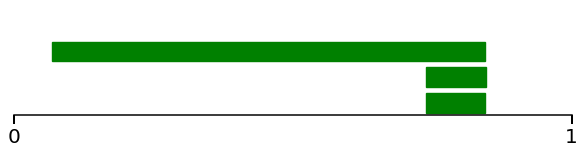}\\
\includegraphics[width=1\linewidth]{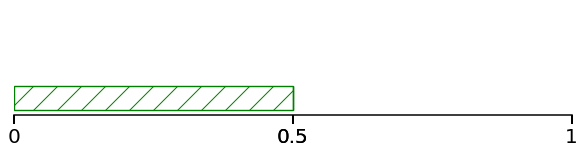}
\end{minipage}
\begin{minipage}{.49\linewidth}
\centering
\includegraphics[width=.7\linewidth]{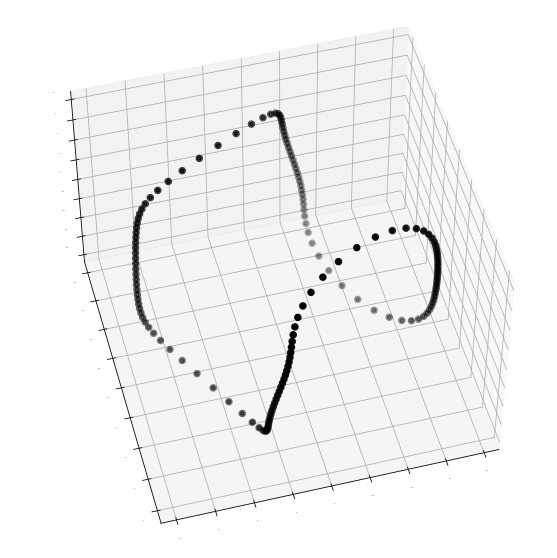}\\
\includegraphics[width=1\linewidth]{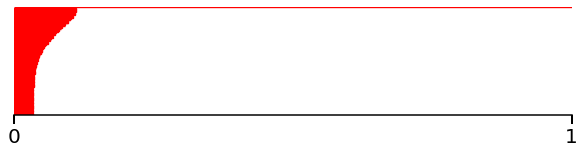}\\
\includegraphics[width=1\linewidth]{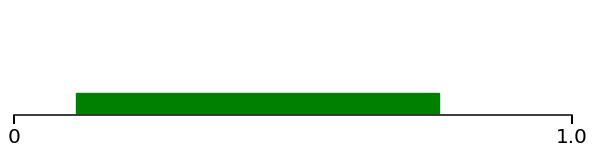}\\
\includegraphics[width=1\linewidth]{Fig24.png}
\end{minipage}
\caption{\textbf{Left:} The set $\check X$ (projected in $\R^3$ via PCA), the barcodes of its Rips filtration, and the lifebar of its first persistent Stiefel-Witney class. 
\textbf{Right:} Same for $\check X'$.
Only green bars of length larger than $0.1$ are represented.}
\label{fig:24-2}
\end{figure}

%Finally, we consider the Rips filtrations of the subsets $\check X$ and $\check X'$, and we compute the lifebars of the first persistent Stiefel-Whitney class of these vector bundle, up to the value $\frac{1}{2}\gamma$ (see Subsection \ref{subsec:Rips}). 
%It turns out that both lifebars are empty. This is consistent with the fact that the underlying vector bundles are themselves trivial.

%\begin{figure}[H]
%\centering
%\begin{minipage}{.32\linewidth}
%\centering
%\includegraphics[width=1\linewidth]{Fig22-3.png}
%\end{minipage}
%\begin{minipage}{.32\linewidth}
%\centering
%\includegraphics[width=1\linewidth]{Fig22-3.png}
%\end{minipage}
%\begin{minipage}{.32\linewidth}
%\centering
%\includegraphics[width=1\linewidth]{Fig23.png}
%\end{minipage}
%\caption{...}
%\label{fig:23}
%\end{figure}

\paragraph{Giraffe moving behind two trees.}
We will now study a variation of this dataset.
As before, we consider a giraffe walking straight, but now with two trees on the foreground. 
The dataset consists of 150 images, each of size $130\times300$ pixels, in RGB format. 
Since $130\times300\times3 = 117\,000$, the dataset can be seen as a subset of $\R^{117\,000}$.
Some of the images can be seen in Figure \ref{fig:25}.

\begin{figure}[H]
\centering
\begin{minipage}{.3\linewidth}
\centering
\includegraphics[width=1\linewidth]{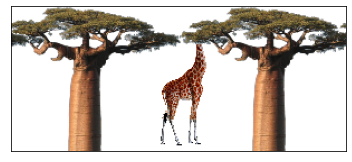}
\end{minipage}
~~~~
\begin{minipage}{.3\linewidth}
\centering
\includegraphics[width=1\linewidth]{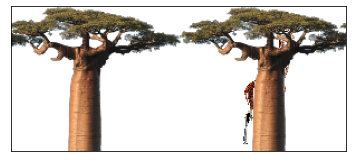}
\end{minipage}
~~~~
\begin{minipage}{.3\linewidth}
\centering
\includegraphics[width=1\linewidth]{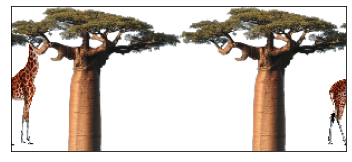}
\end{minipage}
\caption{The dataset consists in a giraffe moving forward, with two trees on the foreground.}
\label{fig:25}
\end{figure}

Just as before, we project the point cloud in $\R^3$ via PCA, and renormalize it.
The corresponding point cloud, that we denote $Y = \{y_1, ..., y_{150}\}$, and the barcode of its Rips filtration can be seen on Figure \ref{fig:26}.

\begin{figure}[H]
\centering
\begin{minipage}{.49\linewidth}
\centering
\includegraphics[width=.7\linewidth]{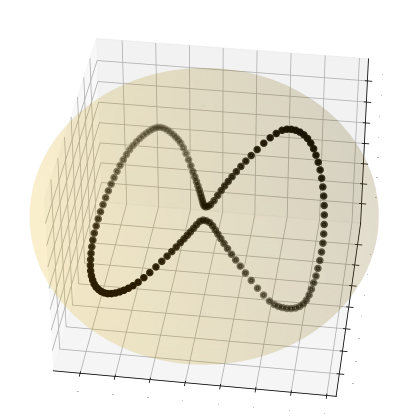}
\end{minipage}
\begin{minipage}{.49\linewidth}
\centering
\includegraphics[width=1\linewidth]{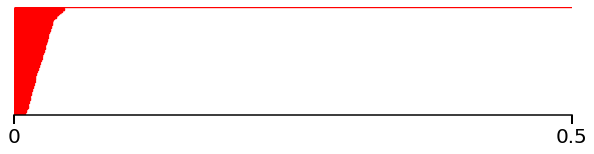}\\\includegraphics[width=1\linewidth]{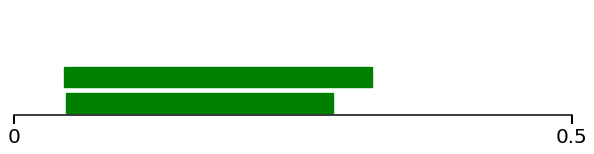}
\end{minipage}
\caption{\textbf{Left:} The point cloud $Y \subset \R^3$.
\textbf{Right:} The barcode of its Rips filtration.
Only green bars of length larger than $0.05$ are represented.}
\label{fig:26}
\end{figure}

Note that, in this collection, the images where the giraffe goes behind the trees are close to each other. Indeed, only a few pixels differ between them. As a consequence, the point cloud $Y$, though lying close to a circle, seems to come closer to itself at some point.
This behavior translates in its persistent cohomology as two prominent $H^1$-features.

We now consider the vector bundles $q$ and $q'$ defined as before:
\begin{center}
$q\colon y_i \mapsto P(y_i)$
~~~~~~~~~~
and
~~~~~~~~~~
$q'\colon y_i \mapsto P(y_{i+1}-y_{i-1})$
\end{center}
They are represented on Figure \ref{fig:27}.
Let $\gamma = 0.5$, and consider the lifted sets $\check Y = \left\{\left(y,~\gamma q(y)\right), ~y\in Y\right\}$ and $\check Y' = \left\{\left(y,~\gamma q'(y)\right), ~y\in Y\right\} \subset \R^3 \times \Grass{1}{\R^3}$. 
%The parameter $\gamma$ has been chosen equal to the death time of the 
We represent them on Figure \ref{fig:28}, as well as the persistence barcode of their Rips filtration, and the lifebars of their first Stiefel-Whitney classes, up to their maximal filtration value. We read that the second one is empty, while the first one is not.

\begin{figure}[H]
\centering
\begin{minipage}{.49\linewidth}
\centering
\includegraphics[width=.75\linewidth]{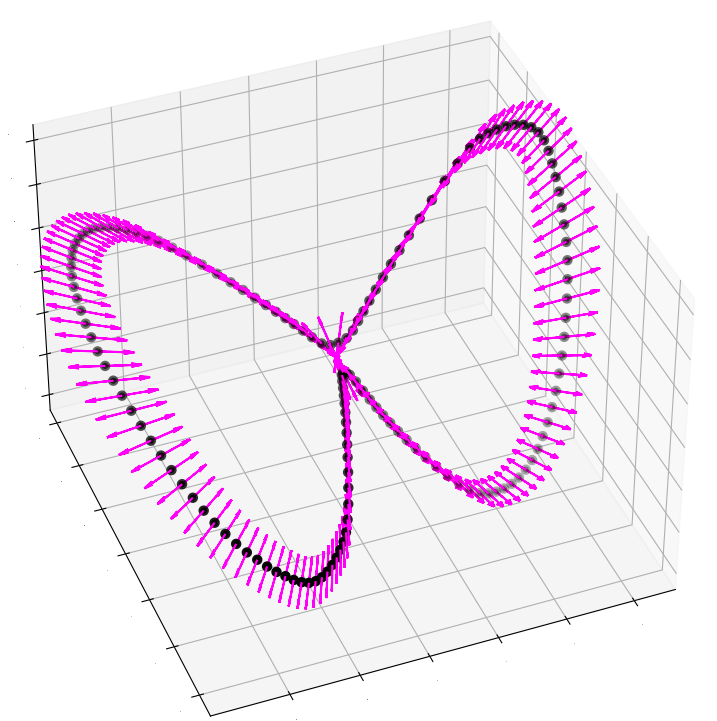}
\end{minipage}
\begin{minipage}{.49\linewidth}
\centering
\includegraphics[width=.75\linewidth]{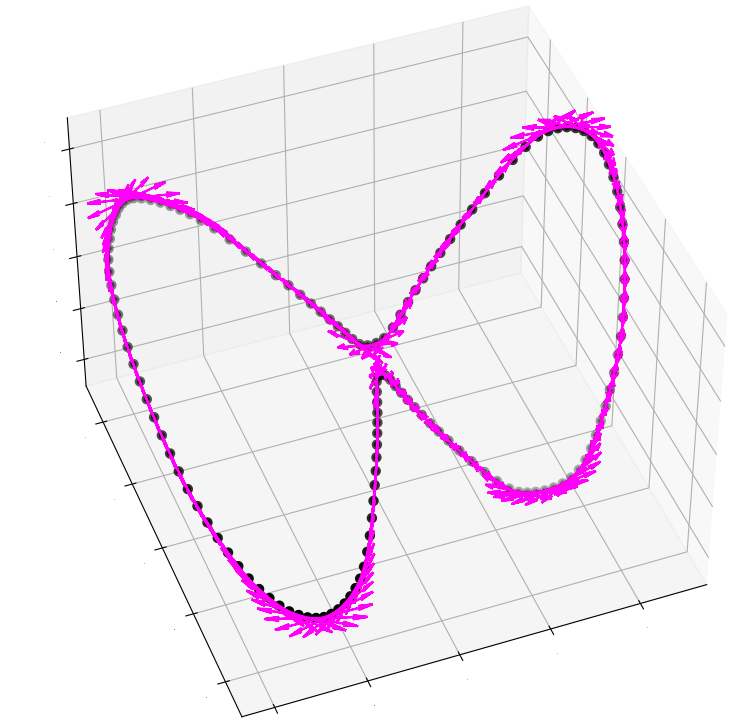}
\end{minipage}
\caption{The vector bundles $q$ and $q'$.}
\label{fig:27}
\end{figure}

\begin{figure}[H]
\centering
\begin{minipage}{.49\linewidth}
\centering
\includegraphics[width=.7\linewidth]{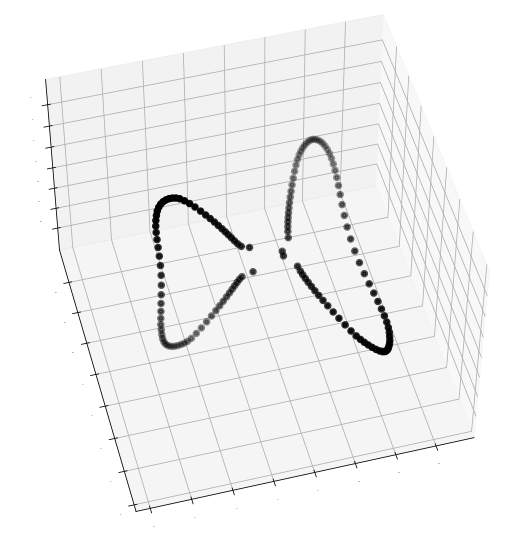}\\
\includegraphics[width=1\linewidth]{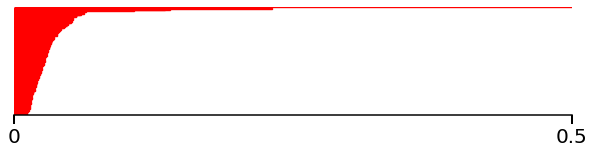}\\
\includegraphics[width=1\linewidth]{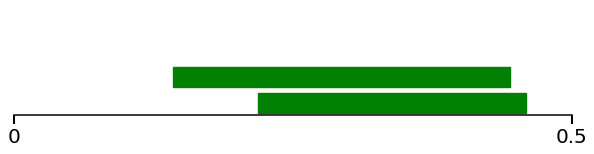}\\
\includegraphics[width=1\linewidth]{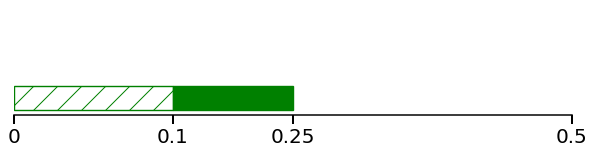}
\end{minipage}
\begin{minipage}{.49\linewidth}
\centering
\includegraphics[width=.7\linewidth]{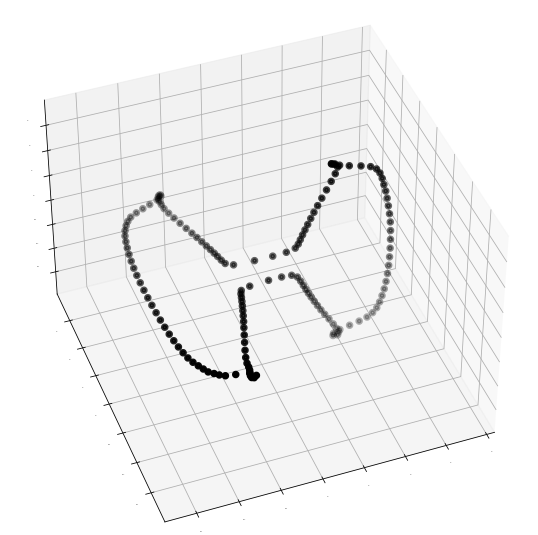}\\
\includegraphics[width=1\linewidth]{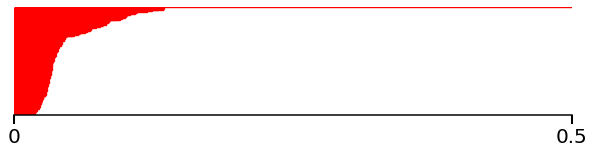}\\
\includegraphics[width=1\linewidth]{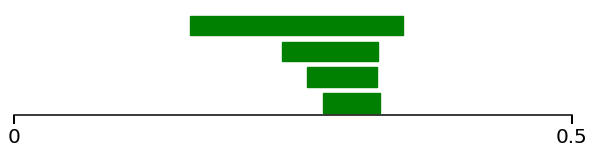}\\
\includegraphics[width=1\linewidth]{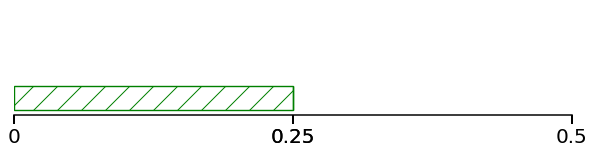}
\end{minipage}
\caption{\textbf{Left:} The set $\check Y$ (projected in $\R^3$ via PCA), the barcodes of its Rips filtration, and the lifebar of its first persistent Stiefel-Witney class.
\textbf{Right:} Same for $\check Y'$.
Only green bars of length larger than $0.05$ are represented.}
\label{fig:28}
\end{figure}

Let us explain this phenomenon.
In both cases, as we can see on Figure \ref{fig:27}, the lines make a full twist while turning around the circle. This corresponds to a trivial bundle. However, in the case of $q$, the points $x$ close to the almost self-intersection of $Y$ corresponds to lines $q(x)$ close to each other.
As a consequence, in the Rips filtration of the lifted set $\check Y$, these points will connect early. Hence the filtration behaves as if $Y$ were composed of two loops. But, on these loops, the lines make only a half-twist. This correspond to the Mobius bundle on the circle (see Figure \ref{fig:2}). Therefore we obtain a non-trivial Stifel-Whitney class.

The bundle $q'$ does not reflect this property. This is because the points $x$ close to the self-intersection of $Y$ does not correspond to lines $q'(x)$ that are close to each other. In the Rips filtration of the lifted set $\check Y'$, the two loops connect late, hence the non-orientability does not appear on its persistent Stiefel-Whitney class. 

%Nonetheless, the bundle $q'$ does not reflect this property. This is because the points $x$ close to the almost self-intersection of $Y$ does not correspond to lines $q'(x)$ that are close to each other. In the Rips filtration of the lifted set $\check Y'$, the two loops connect late. Consequently, the non-orientability does not appear on its persistent Stiefel-Whitney class. 

\paragraph{Rotating cylinders.}
We now propose a dataset whose tangent bundle reflects non-orientability.
Consider the union of two cylinders in $\R^3$, as represented on Figure \ref{fig:29}. By applying rotations, we obtain a dataset of 100 pictures, in RGB format, of $500 \times 500$ pixels. 

\begin{figure}[H]
\centering
\begin{minipage}{.19\linewidth}
\centering
\includegraphics[width=1\linewidth]{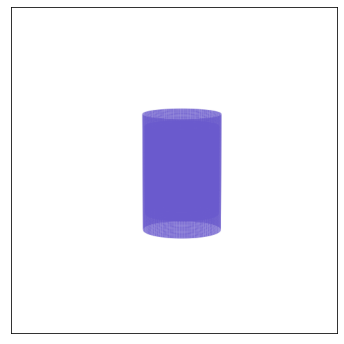}
\end{minipage}
\begin{minipage}{.19\linewidth}
\centering
\includegraphics[width=1\linewidth]{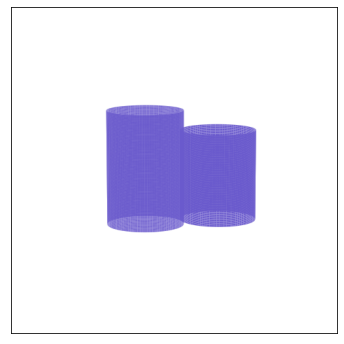}
\end{minipage}
\begin{minipage}{.19\linewidth}
\centering
\includegraphics[width=1\linewidth]{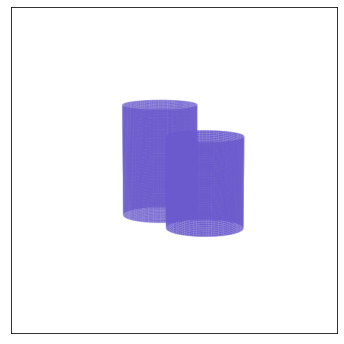}
\end{minipage}
\begin{minipage}{.19\linewidth}
\centering
\includegraphics[width=1\linewidth]{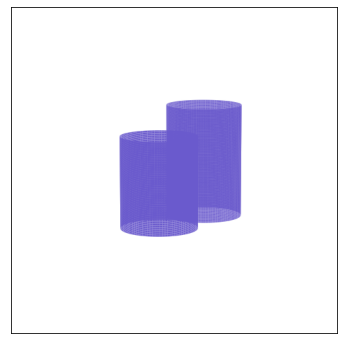}
\end{minipage}
\begin{minipage}{.19\linewidth}
\centering
\includegraphics[width=1\linewidth]{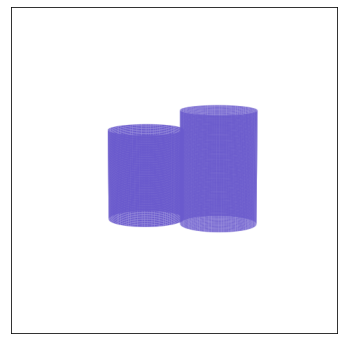}
\end{minipage}
\caption{The dataset consists in different views of two cylinders in $\R^3$.}
\label{fig:29}
\end{figure}

As before, we embed them into $\R^{750\,000}$, project them into $\R^3$ via PCA, and renormalize them. The corresponding point cloud, denoted $Z = \{z_1, ..., z_{100}\}$, and the barcodes of its Rips filtration are represented on Figure \ref{fig:30}.

\begin{figure}[H]
\centering
\begin{minipage}{.49\linewidth}
\centering
\includegraphics[width=.7\linewidth]{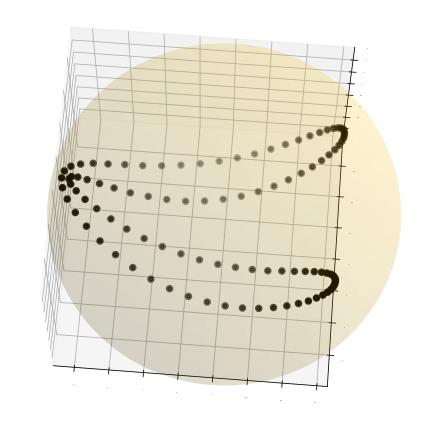}
\end{minipage}
\begin{minipage}{.49\linewidth}
\centering
\includegraphics[width=1\linewidth]{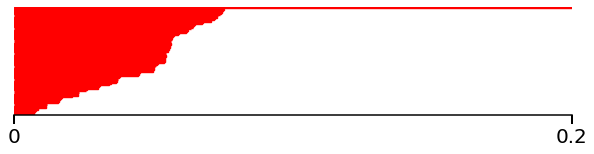}\\
\includegraphics[width=1\linewidth]{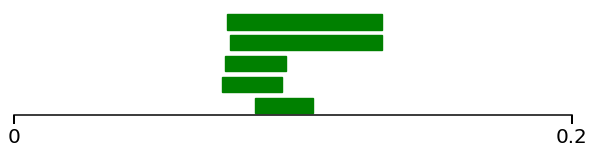}
\end{minipage}
\caption{\textbf{Left:} The point cloud $Z \subset \R^3$.
\textbf{Right:} The barcode of its Rips filtration.
Only green bars of length larger than $0.02$ are represented.}
\label{fig:30}
\end{figure}

Define on $Z$ the tangent bundle $r\colon z_i \mapsto P(z_{i+1}-z_{i-1}) \in \Grass{1}{\R^3}$. 
Let $\gamma = 0.2$ and consider the lifted set $\check Z = \left\{\left(z, \gamma r(z)\right), ~z \in Z\right\} \subset \R^3 \times \Grass{1}{\R^3}$.
The persistence barcodes of its Rips filtration, and the lifebar of its first persistent Stiefel-Whitney class are represented on Figure \ref{fig:31}. We see that the lifebar is not trivial.

\begin{figure}[H]
\centering
\begin{minipage}{.49\linewidth}
\centering
\includegraphics[width=1\linewidth]{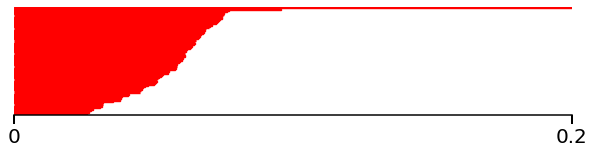}\\
\includegraphics[width=1\linewidth]{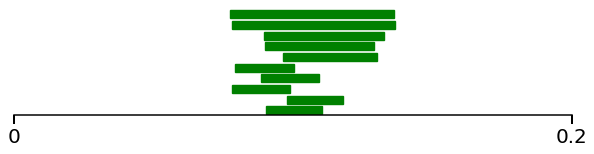}\\
\includegraphics[width=1\linewidth]{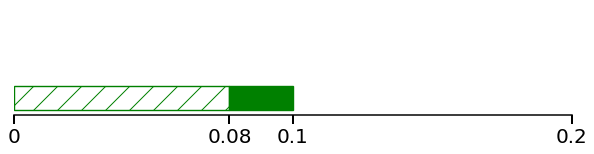}
\end{minipage}
\caption{
The barcodes of the Rips filtration of $\check Z$, and the lifebar of its first persistent Stiefel-Witney class.
Only green bars of length larger than $0.02$ are represented.}
\label{fig:31}
\end{figure}

Here, the same phenomenon as before occurs: there are two images of the collection which are almost equal (when the cylinders are one in front of the other), resulting in a point cloud that almost intersect itself.
Moreover, points $z$ close to the area where $Z$ almost intersect itself correspond to lines $r(z)$ that are close to each other. Consequently, the persistent homology of $\check Z$ shows two prominent loops, whose corresponding vector bundles are non-orientable.

In these two last experiments, we observed the following fact: trivial vector bundles, but whose underlying point clouds present self-similarity, or self-intersection, may result in a \emph{shortcut} of the vector bundle, implying a non-trivial persistent Stiefel-Whitney class.

\subsection{Datasets with intrisic (non-)orientability}
We will now present three datasets which reflect some underlying theoretical orientability or non-orientability.%, such as stated in Corollary \ref{cor:consistency_stability}.

\paragraph{Gorilla on a torus.}
Let us consider a picture of a gorilla, that we translate to the right and downwards via cyclic permutations (see Figure \ref{fig:32}). 
Each image has size $130\times120$ pixels, in RGB format.
The dataset consists of 3900 images (65 vertical permutations and 60 horizontal permutations).

Note that the images behave as if the gorilla was on a torus. Indeed, the torus can be obtained from the square $[0,1]^2$ by gluing its opposite edges.
Hence the various images can be seen as translations of the original image, whose gluing pattern follows the one of the torus.

Since $130\times120\times3 = 46\,800$, the dataset can be seen as a point cloud of $\R^{46\,800}$, that we project into $\R^4$ via PCA. The resulting subset is denoted $X = \left\{x_{i,j}, ~i \in \llbracket1,65\rrbracket, j \in \llbracket1,60\rrbracket\right\}$. The first index $i$ correponds to translations downwards, and the second index $j$ corresponds to translations to the right.

\begin{figure}[H]
\centering
\begin{minipage}{.19\linewidth}
\centering
\includegraphics[width=.8\linewidth]{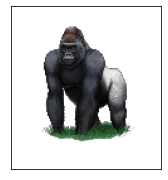}\\
\vspace{.1cm}
\includegraphics[width=.8\linewidth]{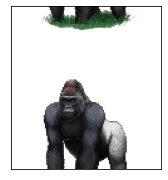}\\
\vspace{.1cm}
\includegraphics[width=.8\linewidth]{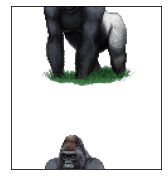}
\end{minipage}
\begin{minipage}{.19\linewidth}
\centering
\includegraphics[width=.8\linewidth]{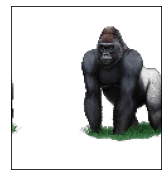}\\
\vspace{.1cm}
\includegraphics[width=.8\linewidth]{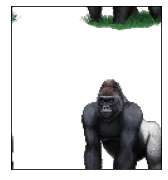}\\
\vspace{.1cm}
\includegraphics[width=.8\linewidth]{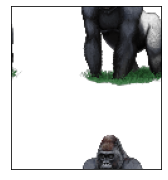}
\end{minipage}
\begin{minipage}{.19\linewidth}
\centering
\includegraphics[width=.8\linewidth]{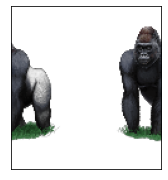}\\
\vspace{.1cm}
\includegraphics[width=.8\linewidth]{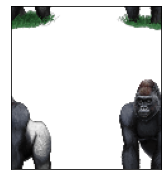}\\
\vspace{.1cm}
\includegraphics[width=.8\linewidth]{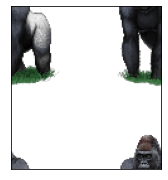}
\end{minipage}
\begin{minipage}{.19\linewidth}
\centering
\includegraphics[width=.8\linewidth]{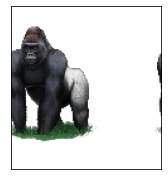}\\
\vspace{.1cm}
\includegraphics[width=.8\linewidth]{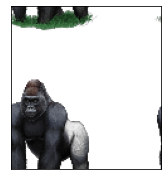}\\
\vspace{.1cm}
\includegraphics[width=.8\linewidth]{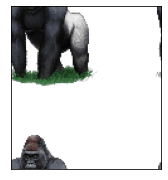}
\end{minipage}
\begin{minipage}{.19\linewidth}
\centering
\includegraphics[width=.8\linewidth]{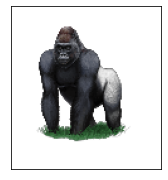}\\
\vspace{.1cm}
\includegraphics[width=.8\linewidth]{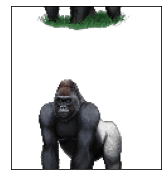}\\
\vspace{.1cm}
\includegraphics[width=.8\linewidth]{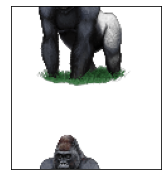}
\end{minipage}
\caption{The dataset consists in a gorilla moving forward and downwards on a torus.}
\label{fig:32}
\end{figure}

Consider on $X$ the \emph{vertical tangent bundle} and the \emph{horizontal tangent bundle}, defined respectively as 
\begin{center}
$p\colon x_{i,j} \mapsto P(x_{i-1, j}-x_{i+1, j})$
~~~~~~~~~~
and
~~~~~~~~~~
$p'\colon x_{i,j} \mapsto P(x_{i, j-1}-x_{i, j+1})$
\end{center}
%\begin{align*}
%p\colon x_{i,j} \mapsto P(x_{i-1, j}-x_{i+1, j}).
%\end{align*}
%$\xi\colon x_{i,j} \mapsto P(x_{i-1, j}-x_{i+1, j})$.
where we recall that $P(x)$ is the orthogonal projection matrix on the line spanned by $x$.
The vector bundle $p$ is to be seen as the vertical component of the tangent bundle of a torus, and $p'$ as the horizontal one. Both are trivial bundles.

Now, let $\gamma = 0.3$, and consider the corresponding lifted sets $\check X = \left\{\left(x, \gamma p(x)\right), ~x \in X\right\}$ and $\check X' = \left\{\left(x, \gamma p'(x)\right), ~x \in X\right\} \subset \R^3 \times \Grass{1}{\R^3}$.
We represent on Figure \ref{fig:33} the barcodes of the Rips filtrations of the sets $\check X$ and $\check X'$, and the lifebars of their first persistent Stiefel-Whitney classes. We read that they are zero all along the filtration. This is consistent with the underlying line bundles on the torus being trivial.

\begin{figure}[H]
\centering
\begin{minipage}{.49\linewidth}
\centering
\includegraphics[width=1\linewidth]{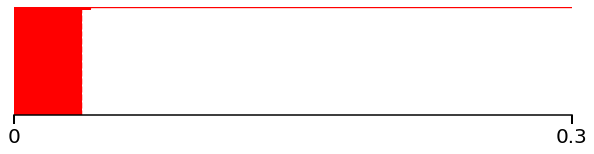}\\
\includegraphics[width=1\linewidth]{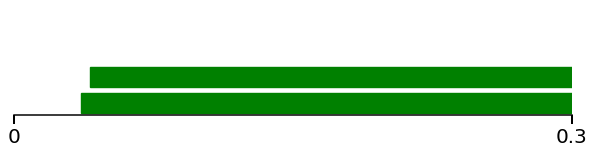}\\
\includegraphics[width=1\linewidth]{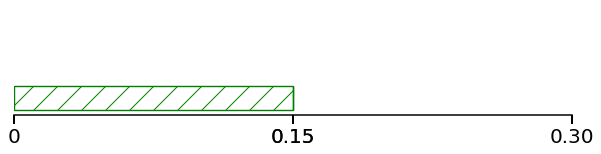}
\end{minipage}
\begin{minipage}{.49\linewidth}
\centering
\includegraphics[width=1\linewidth]{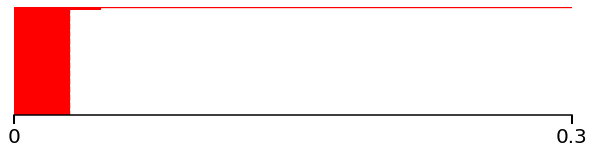}\\
\includegraphics[width=1\linewidth]{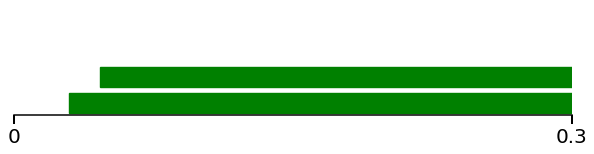}\\
\includegraphics[width=1\linewidth]{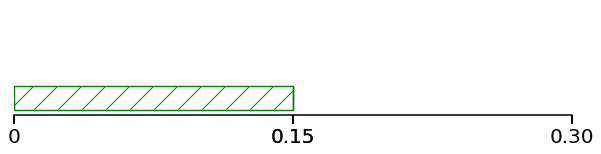}
\end{minipage}
\caption{
\textbf{Left:} The barcodes of the Rips filtration of $\check X$, and the lifebar of its first persistent Stiefel-Witney class.
\textbf{Right:} Same for $\check X'$.
Only green bars of length larger than $0.03$ are represented.}
\label{fig:33}
\end{figure}

\paragraph{Gorilla on a Klein bottle.}
We will modify this dataset: we still translate the gorilla, but while translating it to the right, we inverse the part that arrives at the left (see Figure \ref{fig:34}). It behaves as if the picture was glued to itself according to the gluing of a Klein bottle. 
Just as before, the dataset consists of 3900 images of size $130\times120$ pixels.
We embed the images into $\R^{46\,800}$ and project them into $\R^4$ via PCA. The resulting subset is denoted $Y = \{y_{i,j}, ~i \in \llbracket1,65\rrbracket, j \in \llbracket1,60\rrbracket\}$.

Again, consider the vector bundles $q\colon y_{i,j} \mapsto P(y_{i-1, j}-y_{i+1, j})$ and $q'\colon y_{i,j} \mapsto P(y_{i, j-1}-y_{i, j+1})$.
They correspond to the horizontal and vertical components of the tangent bundle of a Klein bottle. Only one of them is non-trivial.
Let $\gamma = 0.3$, and consider the corresponding lifted sets $\check Y$ and $\check Y'$.
We represent on Figure \ref{fig:35} the barcodes of their Rips filtrations, and the lifebars of their first persistent Stiefel-Whitney classes. The first lifebar is non-empty, reflecting the non-triviality of the underlying bundle.

\begin{figure}[H]
\centering
\begin{minipage}{.19\linewidth}
\centering
\includegraphics[width=.8\linewidth]{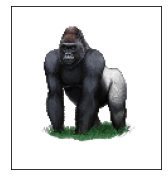}\\
\vspace{.1cm}
\includegraphics[width=.8\linewidth]{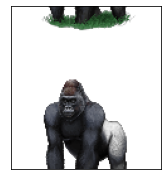}\\
\vspace{.1cm}
\includegraphics[width=.8\linewidth]{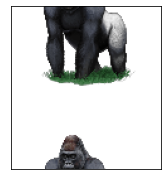}
\end{minipage}
\begin{minipage}{.19\linewidth}
\centering
\includegraphics[width=.8\linewidth]{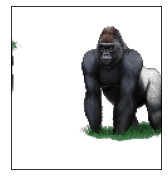}\\
\vspace{.1cm}
\includegraphics[width=.8\linewidth]{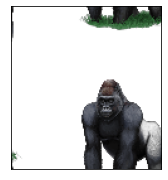}\\
\vspace{.1cm}
\includegraphics[width=.8\linewidth]{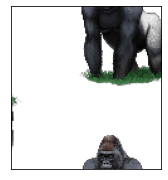}
\end{minipage}
\begin{minipage}{.19\linewidth}
\centering
\includegraphics[width=.8\linewidth]{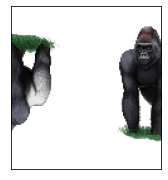}\\
\vspace{.1cm}
\includegraphics[width=.8\linewidth]{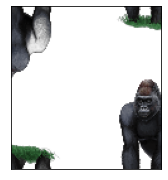}\\
\vspace{.1cm}
\includegraphics[width=.8\linewidth]{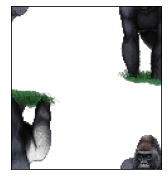}
\end{minipage}
\begin{minipage}{.19\linewidth}
\centering
\includegraphics[width=.8\linewidth]{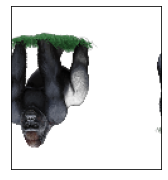}\\
\vspace{.1cm}
\includegraphics[width=.8\linewidth]{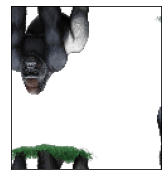}\\
\vspace{.1cm}
\includegraphics[width=.8\linewidth]{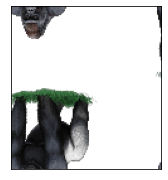}
\end{minipage}
\begin{minipage}{.19\linewidth}
\centering
\includegraphics[width=.8\linewidth]{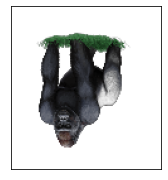}\\
\vspace{.1cm}
\includegraphics[width=.8\linewidth]{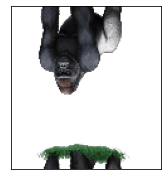}\\
\vspace{.1cm}
\includegraphics[width=.8\linewidth]{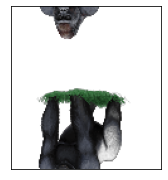}
\end{minipage}
\caption{The dataset consists in a gorilla moving forward and downwards on a Klein bottle.}
\label{fig:34}
\end{figure}
\vspace{-.2cm}

\begin{figure}[H]
\centering
\begin{minipage}{.49\linewidth}
\centering
\includegraphics[width=1\linewidth]{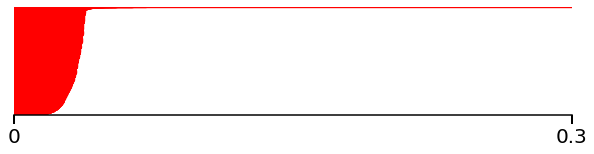}\\
\includegraphics[width=1\linewidth]{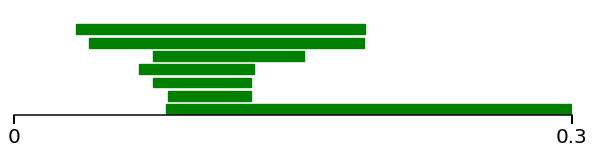}\\
\includegraphics[width=1\linewidth]{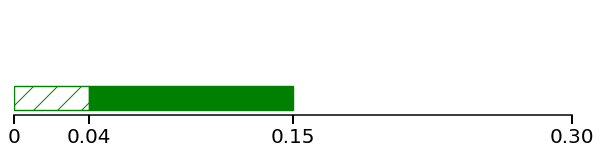}
\end{minipage}
\begin{minipage}{.49\linewidth}
\centering
\includegraphics[width=1\linewidth]{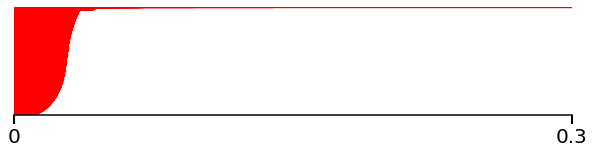}\\
\includegraphics[width=1\linewidth]{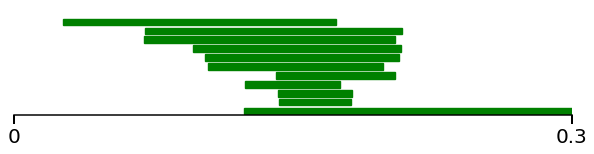}\\
\includegraphics[width=1\linewidth]{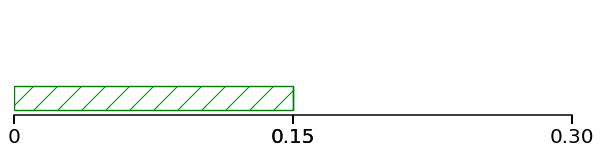}
\end{minipage}
\caption{
\textbf{Left:} The barcodes of the Rips filtration of $\check Y$, and the lifebar of its first persistent Stiefel-Witney class.
\textbf{Right:} Same for $\check Y'$.
Only green bars of length larger than $0.03$ are represented.}
\label{fig:35}
\end{figure}

\paragraph{Gorilla on the projective plane.}
We close this subsection with a last variation of the dataset: the gorilla is translated to the right, with an inversion of the left part, and translated downwards, with also an inversion of the upper part  (see Figure \ref{fig:36}). 
It behaves as if the image was glued to itself according to the gluing of the projective plane $\Grass{1}{\R^3}$. 
The dataset still consists of 3900 images of size $130\times120$ pixels, that we embed into $\R^{46\,800}$ and project into $\R^4$ via PCA. The resulting subset is denoted $Z = \{z_{i,j}, ~i \in \llbracket1,65\rrbracket, j \in \llbracket1,60\rrbracket\}$.

We still consider the vector bundles $r\colon z_{i,j} \mapsto P(z_{i-1, j}-z_{i+1, j})$ and $r'\colon z_{i,j} \mapsto P(z_{i, j-1}-z_{i, j+1})$.
%They correspond to the horizontal and vertical components of the tangent bundle of a Klein bottle. Only one of them is non-trivial.
Let $\gamma = 0.3$, and consider the corresponding lifted sets $\check Z$ and $\check Z'$.
We represent on Figure \ref{fig:37} the barcodes of their Rips filtrations, and the lifebars of their first persistent Stiefel-Whitney classes. Both lifebars are non-empty. We deduce that the underlying line bundles are non-trivial.

\begin{figure}[H]
\centering
\begin{minipage}{.19\linewidth}
\centering
\includegraphics[width=.8\linewidth]{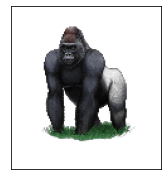}\\
\vspace{.1cm}
\includegraphics[width=.8\linewidth]{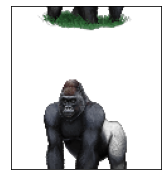}\\
\vspace{.1cm}
\includegraphics[width=.8\linewidth]{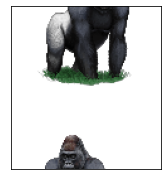}
\end{minipage}
\begin{minipage}{.19\linewidth}
\centering
\includegraphics[width=.8\linewidth]{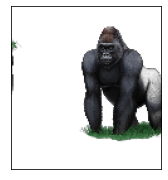}\\
\vspace{.1cm}
\includegraphics[width=.8\linewidth]{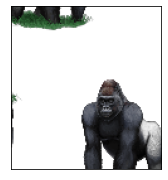}\\
\vspace{.1cm}
\includegraphics[width=.8\linewidth]{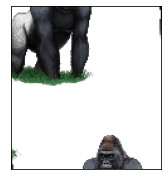}
\end{minipage}
\begin{minipage}{.19\linewidth}
\centering
\includegraphics[width=.8\linewidth]{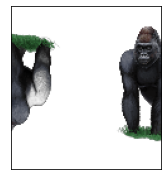}\\
\vspace{.1cm}
\includegraphics[width=.8\linewidth]{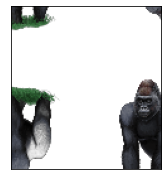}\\
\vspace{.1cm}
\includegraphics[width=.8\linewidth]{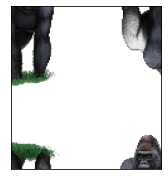}
\end{minipage}
\begin{minipage}{.19\linewidth}
\centering
\includegraphics[width=.8\linewidth]{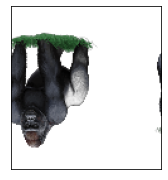}\\
\vspace{.1cm}
\includegraphics[width=.8\linewidth]{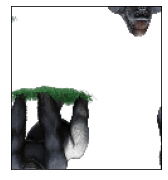}\\
\vspace{.1cm}
\includegraphics[width=.8\linewidth]{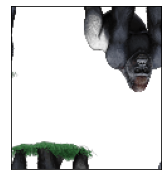}
\end{minipage}
\begin{minipage}{.19\linewidth}
\centering
\includegraphics[width=.8\linewidth]{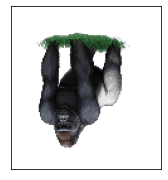}\\
\vspace{.1cm}
\includegraphics[width=.8\linewidth]{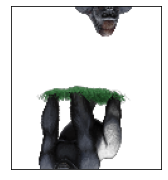}\\
\vspace{.1cm}
\includegraphics[width=.8\linewidth]{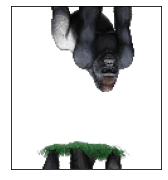}
\end{minipage}
\caption{The dataset consists in a gorilla moving forward and downwards on the projective plane.}
\label{fig:36}
\end{figure}

\begin{figure}[H]
\centering
\begin{minipage}{.49\linewidth}
\centering
\includegraphics[width=1\linewidth]{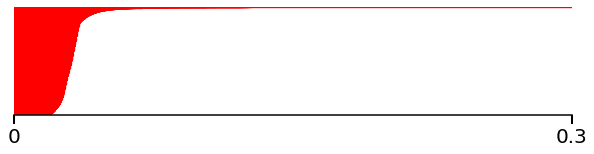}\\
\includegraphics[width=1\linewidth]{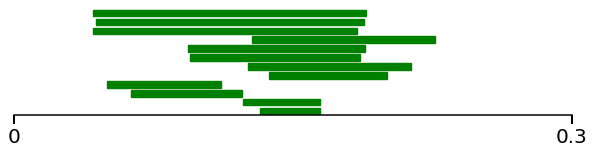}\\
\includegraphics[width=1\linewidth]{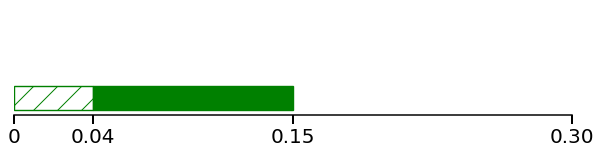}
\end{minipage}
\begin{minipage}{.49\linewidth}
\centering
\includegraphics[width=1\linewidth]{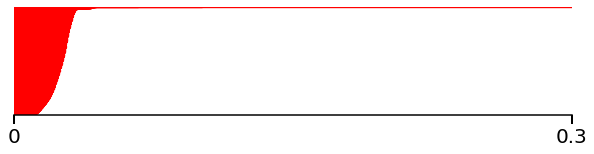}\\
\includegraphics[width=1\linewidth]{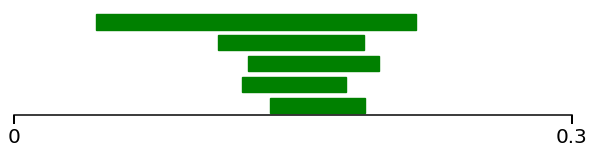}\\
\includegraphics[width=1\linewidth]{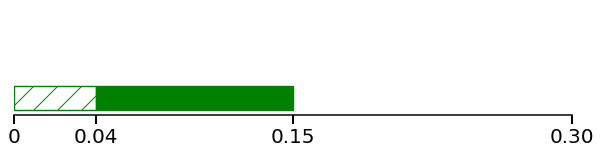}
\end{minipage}
\caption{
\textbf{Left:} The barcodes of the Rips filtration of $\check Z$, and the lifebar of its first persistent Stiefel-Witney class.
\textbf{Right:} Same for $\check Z'$.
Only green bars of length larger than $0.03$ are represented.}
\label{fig:37}
\end{figure}

In these various experiments, we applied the same method: in ordered datasets, indexed by only one value $i\mapsto x_i$, two values $(i,j) \mapsto x_{i,j}$ or more, one can consider directional line bundles by approximating partial derivatives, such as $\frac{1}{2}(x_{i+1,j}-x_{i-1,j})$.
The first persistent Stiefel-Whitney classes of such bundles deliver information about the dataset in this particular direction.

\section{Conclusion}
In this paper we defined the persistent Stiefel-Whitney classes of vector bundle filtrations. We proved that they are stable with respect to the interleaving distance between vector bundle filtrations. 
We studied the particular case of \v{C}ech bundle filtrations of subsets of $\R^n \times \matrixspace{\R^m}$, and showed that they yield consistent estimators of the usual Stiefel-Whitney classes of some underlying vector bundle.

Moreover, when the dimension of the bundle is 1 and $X$ is finite, we proposed an algorithm to compute the first persistent Stiefel-Whitney class.
We also described a way to compute their lifebars via mapping cones.

Our algorithm is limited to the bundles of dimension 1 since we only implemented triangulations of the Grassmannian $\Grass{d}{\R^m}$ when $d=1$.
However, any other triangulation of $\Grass{d}{\R^m}$, with a computable face map, could be included in the algorithm without any modification.
As far as we know, no triangulation of a Grassmannian $\Grass{d}{\R^m}$ with $d > 1$ has never been given explicitely (i.e., as a list of simplices stored in a computer). 
This is a problem of interest, which will be adressed in a further work. A strategy could consist in using the usual CW-structures of the Grassmannians, and converting them into simplicial complexes, as done theoretically in \citet[Theorem 2C.5]{Hatcher_Algebraic}.
A recent result of \cite{govc2020many} gives an idea about the complexity of this problem: the number of simplices of minimal triangulations of $\Grass{d}{\R^m}$ must grow exponentially in both $d$ and $m$.

\paragraph{Acknowledgements.}
I wish to thank Frédéric Chazal and Marc Glisse for fruitful discussions and corrections, as well as the anonymous reviewers for corrections and clarifications.
I also thank Luis Scoccola for pointing out a strengthening of Lemma \ref{lem:shadowmap}.

\appendix

\section{Supplementary material for Sect. \ref{sec:algorithm}}
\label{sec:appendix_algorithm}

\subsection{Study of Example \ref{ex:gamma_mobius}}
\label{subsec:gamma_mobius}

We consider the set
\begin{align*}
&X = \bigg\{
\bigg( 
\begin{pmatrix}
\cos(\theta)  \\
\sin(\theta) 
\end{pmatrix}
,
\begin{pmatrix}
\cos(\frac{\theta}{2})^2 & \cos(\frac{\theta}{2}) \sin(\frac{\theta}{2}) \\
\cos(\frac{\theta}{2}) \sin(\frac{\theta}{2}) & \sin(\frac{\theta}{2})^2 
\end{pmatrix}
\bigg),
\theta \in [0, 2\pi)  \bigg\}.
\end{align*}
In order to study the \v{C}ech filtration of $X$, we shall apply the following affine transformation: let $Y$ be the subset of $\R^2 \times \matrixspace{\R^2}$ defined as
\begin{align*}
&Y = \bigg\{
\bigg( 
\begin{pmatrix}
\cos(\theta)  \\
\sin(\theta) 
\end{pmatrix}
,
\gamma \begin{pmatrix}
\cos(\frac{\theta}{2})^2 & \cos(\frac{\theta}{2}) \sin(\frac{\theta}{2}) \\
\cos(\frac{\theta}{2}) \sin(\frac{\theta}{2}) & \sin(\frac{\theta}{2})^2 
\end{pmatrix}
\bigg),
\theta \in [0, 2\pi)  \bigg\}.
\end{align*}
and let $\Y = (Y^t)_{t\geq 0}$ be the \v{C}ech filtration of $Y$ in $\R^2 \times \matrixspace{\R^2}$ endowed with the norm $\gammaNun{(x,A)} = \sqrt{\eucN{x}^2 + \frobN{A}^2}$. 
We recall that the \v{C}ech filtration of $X$, denoted $\X = (X^t)_{t\geq 0}$, is defined with respect to the norm $\gammaN{\cdot}$.
It is clear that, for every $t \geq 0$, the thickenings $X^t$ and $Y^t$ are homeomorphic via the application
\begin{align*}
h \colon \R^2 \times \matrixspace{\R^2} &\longrightarrow \R^2 \times \matrixspace{\R^2} \\
(x, A) &\longmapsto (x, \gamma A).
\end{align*}
As a consequence, the persistence cohomology modules associated to $\X$ and $\Y$ are isomorphic.

Next, notice that $Y$ is a subset of the affine subspace of dimension 2 of $\R^2 \times \matrixspace{\R^2}$ with origin $O$ and spanned by the vectors $e_1$ and $e_2$, where 
\begin{align*}
O &= \left( \begin{pmatrix} 0 \\ 0   \end{pmatrix}, \frac{\gamma}{2}\begin{pmatrix} 1 & 0 \\ 0 & 1  \end{pmatrix} \right),
~~~~~
e_1 = \left( \begin{pmatrix} 1 \\ 0   \end{pmatrix}, 
\frac{\gamma}{2} \begin{pmatrix} 1 & 0 \\ 0 & -1  \end{pmatrix} \right),
~~~~~
e_2 = \left( \begin{pmatrix} 0 \\ 1   \end{pmatrix}, 
\frac{\gamma}{2} \begin{pmatrix} 0 & 1 \\ 1 & 0  \end{pmatrix} \right).
\end{align*}
Indeed, using the equality
\begin{align*}
\begin{pmatrix}
\cos(\frac{\theta}{2})^2 & \cos(\frac{\theta}{2}) \sin(\frac{\theta}{2}) \\
\cos(\frac{\theta}{2}) \sin(\frac{\theta}{2}) & \sin(\frac{\theta}{2})^2 
\end{pmatrix}
\bigg)
=
\frac{1}{2}\begin{pmatrix} 1 & 0 \\ 0 & 1  \end{pmatrix}
+ 
\frac{1}{2}
\begin{pmatrix}
\cos(\theta) & \sin(\theta) \\
\sin(\theta) & -\cos(\theta) 
\end{pmatrix},
\end{align*}
we obtain
\begin{align*}
Y = 
O
+
\left\{
\cos(\theta) e_1 + \sin(\theta) e_2,
~~\theta \in [0, 2\pi)
\right \}.
\end{align*}
We see that $Y$ is a circle, of radius $\eucN{e_1} = \eucN{e_2} = \sqrt{1 + \frac{\gamma^2}{2}}$.

Let $E$ denotes the affine space with origin $O$ and spanned by the vectors $e_1$ and $e_2$.
Lemma \ref{lem:persistence_lift}, stated below, shows that the persistent cohomology of $Y$, seen in the ambient space $\R^2 \times \matrixspace{\R^2}$, is the same as the persistent cohomology of $Y$ restricted to the subspace $E$.
As a consequence, $Y$ has the same persistence as a circle of radius $\sqrt{1 + \frac{\gamma^2}{2}}$ in the plane. Hence its barcode can be described as follows:
\begin{itemize}
\itemsep0.2em 
\item one $H^0$-feature: the bar $[0, +\infty)$,
\item one $H^1$-feature: the bar $\left[0, \sqrt{1 + \frac{\gamma^2}{2}}\right)$.
\end{itemize}

\begin{lemma}
\label{lem:persistence_lift}
Let $Y \subset \R^n$ be any subset, and define $\checkY = Y \times\{(0, ..., 0)\} \subset \R^n \times \R^m$.
Let these spaces be endowed with the usual Euclidean norms.
Then the \v{C}ech filtrations of $Y$ and $\checkY$ yields isomorphic persistence modules.
\end{lemma}
\begin{proof}
Let $\mathrm{proj}_n\colon \R^n \times \R^m \rightarrow \R^n$ be the projection on the first $n$ coordinates.
One verifies that, for every $t \geq 0$, the map $\mathrm{proj}_n \colon\checkY^t \rightarrow Y^t$ is a homotopy equivalence.
At cohomology level, these maps induce an isomorphism of persistence modules.
\end{proof}

Let us now study the \v{C}ech bundle filtration of $Y$, denoted $(\Y, \p)$.
According to Equation \eqref{eq:tmax_subset_grass}, its filtration maximal value is $\tmax{Y} = \tmaxgamma{X} = \frac{\gamma}{\sqrt{2}}$.
Note that $\frac{\gamma}{\sqrt{2}}$ is lower than $\sqrt{1+ \frac{\gamma^2}{2}}$, which is the radius of the circle $Y$.
Hence, for $t < \tmax{Y}$, the inclusion $Y \hookrightarrow Y^t$ is a homotopy equivalence.
Consider the following commutative diagram:
\begin{center}
\begin{tikzcd}%[row sep=tiny]
Y \arrow[dr, "p^0", swap] \arrow[rr, hook]&  & Y^{t} \arrow[dl, "p^{t}"] \\
& \Grass{1}{\R^2} &
\end{tikzcd}
\end{center}
It induces the following diagram in cohomology:
\begin{center}
\begin{tikzcd}%[column sep=tiny]
H^*(Y) &  & H^*(Y^{t}) \arrow[ll, "\sim", swap]  \\
& H^*(\Grass{1}{\R^2}) \arrow[ul, "(p^0)^*"] \arrow[ur, "(p^{t})^*", swap] &
\end{tikzcd}
\end{center}
The horizontal arrow is an isomorphism.
Hence the map $(p^t)^*\colon H^*(Y^t) \leftarrow H^*(\Grass{1}{\R^2})$ is equal to $(p^0)^*$. We only have to understand $(p^0)^*$.

Remark that the map $p^0\colon Y \rightarrow \Grass{1}{\R^2}$ can be seen as the tautological bundle of the circle. It is then a standard result that $(p^0)^*\colon H^*(Y) \leftarrow H^*(\Grass{1}{\R^2})$ is nontrivial. 
Alternatively, $p^0$ can be seen as a map between two circles. It is injective, hence its degree (modulo 2) is 1.
We still deduce that $(p^0)^*$ is nontrivial.
As a consequence, the persistent Stiefel-Whitney class $w_1^t(X)$ is nonzero for every $t< \tmax{Y}$.

\subsection{Study of Example \ref{ex:gamma_normal}}
\label{subsec:gamma_normal}
We consider the set 
\begin{align*}
&X = \bigg\{
\bigg( 
\begin{pmatrix}
\cos(\theta)  \\
\sin(\theta) 
\end{pmatrix}
,
\begin{pmatrix}
\cos(\theta)^2 & \cos(\theta) \sin(\theta) \\
\cos(\theta) \sin(\theta) & \sin(\theta)^2 
\end{pmatrix}
\bigg),
\theta \in [0, 2\pi)  \bigg\}.
\end{align*}
As we explained in the previous subsection, the \v{C}ech filtration of $X$ with respect to the norm $\gammaN{\cdot}$ yields the same persistence as the \v{C}ech filtration of $Y$ with respect to the norm $\gammaNun{\cdot}$, where 
\begin{align*}
&Y = \bigg\{
\bigg( 
\begin{pmatrix}
\cos(\theta)  \\
\sin(\theta) 
\end{pmatrix}
,
\gamma \begin{pmatrix}
\cos(\theta)^2 & \cos(\theta) \sin(\theta) \\
\cos(\theta) \sin(\theta) & \sin(\theta)^2 
\end{pmatrix}
\bigg),
\theta \in [0, 2\pi)  \bigg\}.
\end{align*}

Notice that $Y$ is a subset of the affine subspace of dimension 4 of $\R^2 \times \matrixspace{\R^2}$ with origin 
$O = \left( \begin{psmallmatrix} 0 \\ 0   \end{psmallmatrix}, \frac{1}{2}\begin{psmallmatrix} 1 & 0 \\ 0 & 1  \end{psmallmatrix} \right)$ and spanned by the vectors $e_1, e_2, e_3$ and $e_4$, where 
\begin{align*}
e_1 = &\left( \begin{pmatrix} 1 \\ 0   \end{pmatrix}, 
\begin{pmatrix} 0 & 0 \\ 0 & 0  \end{pmatrix} \right),
~~~~~~~~~~~~~~~~~e_2 = \left( \begin{pmatrix} 0 \\ 1   \end{pmatrix},
\begin{pmatrix} 0 & 0 \\ 0 & 0  \end{pmatrix} \right), \\
e_3 = \frac{1}{\sqrt{2}} &\left( \begin{pmatrix} 0 \\ 0   \end{pmatrix}, 
\begin{pmatrix} 1 & 0 \\ 0 & -1  \end{pmatrix} \right), 
~~~~~~~e_4 = \frac{1}{\sqrt{2}} \left( \begin{pmatrix} 0 \\ 0   \end{pmatrix}, 
\begin{pmatrix} 0 & 1 \\ 1 & 0  \end{pmatrix} \right).
\end{align*}
\noindent
Indeed, $Y$ can be written as
\begin{align*}
Y = O + 
%\left( \begin{pmatrix} 0 \\ 0   \end{pmatrix}, 
%\frac{1}{2} \begin{pmatrix} 1 & 0 \\ 0 & 1  \end{pmatrix} \right)
%+
\left\{
\cos(\theta) e_1 + \sin(\theta) e_2 + \frac{\gamma}{\sqrt{2}}\cos(2\theta) e_3 +  \frac{\gamma}{\sqrt{2}}\sin(2\theta) e_4,
~~\theta \in [0, 2\pi)
\right \}.
\end{align*}
This comes from the equality
\begin{align*}
\begin{pmatrix}
\cos(\theta)^2 & \cos(\theta) \sin(\theta) \\
\cos(\theta) \sin(\theta) & \sin(\theta)^2 
\end{pmatrix}
=
\frac{1}{2}\begin{pmatrix} 1 & 0 \\ 0 & 1  \end{pmatrix}
+ 
\frac{1}{2}
\begin{pmatrix}
\cos(2\theta) & \sin(2\theta) \\
\sin(2\theta) & -\cos(2\theta) 
\end{pmatrix}.
\end{align*}
\noindent
Observe that $Y$ is a torus knot, i.e. a simple closed curve included in the torus $\T$, defined as
\begin{align*}
\T = 
O
+
\left\{
\cos(\theta) e_1 + \sin(\theta) e_2 + \frac{\gamma}{\sqrt{2}}\cos(\nu) e_3 +  \frac{\gamma}{\sqrt{2}}\sin(\nu) e_4,
~~\theta, \nu \in [0, 2\pi)
\right \}.
\end{align*}
The curve $Y$ winds one time around the first circle of the torus, and two times around the second one, as represented in Figure \ref{fig:17}. It is known as the torus knot $(1,2)$.

Let $E$ denotes the affine subspace with origin $O$ and spanned by $e_1, e_2, e_3, e_4$. 
Since $Y$ is a subset of $E$, it is equivalent to study the \v{C}ech filtration of $Y$ restricted to $E$ (as in Lemma \ref{lem:persistence_lift}).
We shall denote the coordinates of points $x \in E$ with respect to the orthonormal basis $(e_1, e_2, e_3, e_4)$. That is, a tuple $(x_1, x_2, x_3, x_4)$ shall refer to the point $O+x_1 e_1 + x_2e_2 + x_3e_3 + x_4 e_4$ of $E$.
Seen in $E$, the set $Y$ can be written as 
\begin{align*}
Y = 
\left\{\left(\cos(\theta), \sin(\theta), 
\frac{\gamma}{\sqrt{2}}\cos(2\theta), \frac{\gamma}{\sqrt{2}}\sin(2\theta)\right), \theta \in [0, 2\pi)\right\}.
\end{align*}
Moreover, for every $\theta \in [0, 2\pi)$, we shall denote $y_\theta = \left(\cos(\theta), \sin(\theta),  \frac{\gamma}{\sqrt{2}} \cos(2\theta),  \frac{\gamma}{\sqrt{2}} \sin(2\theta)\right)$.

\begin{figure}[H]
\centering
\begin{minipage}{.49\linewidth}
\centering
\includegraphics[width=.5\linewidth]{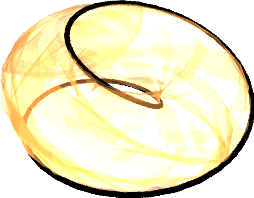}
\end{minipage}
\begin{minipage}{.49\linewidth}
\centering
\includegraphics[width=.6\linewidth]{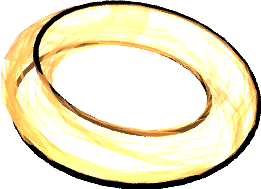}
\end{minipage}
\caption{Representations of the set $Y$, lying on a torus, for a small value of $\gamma$ (left) and a large value of $\gamma$ (right).}
\label{fig:17}
\end{figure}

We now state two lemmas that will be useful in what follows.
\begin{lemma}
\label{lem:torusknot_diameter}
For every $\theta \in [0, 2\pi)$, the map $\theta' \mapsto \eucN{y_{\theta} -  y_{\theta'}}$ admits the following critical points:
\begin{itemize}
\itemsep0.2em 
\item $\theta'-\theta = 0$ and $\theta'-\theta = \pi$ if $\gamma \leq \frac{1}{\sqrt{2}}$,
\item $\theta'-\theta = 0$, $\pi$, $\arccos(-\frac{1}{2\gamma^2})$ and $-\arccos(-\frac{1}{2\gamma^2})$ if $\gamma \geq \frac{1}{\sqrt{2}}$.
\end{itemize}
They correspond to the values 
\begin{itemize}
\itemsep0.2em 
\item $\eucN{y_{\theta} -  y_{\theta'}} = 0$ ~if~ $\theta'-\theta = 0$, 
\item $\eucN{y_{\theta} -  y_{\theta'}} = 2$ ~if~ $\theta'-\theta = \pi$,
\item $\eucN{y_{\theta} -  y_{\theta'}} = \sqrt{2}\sqrt{1 + \gamma^2 + \frac{1}{4 \gamma^2}}$ ~if~ $\theta'-\theta = \pm \arccos(-\frac{1}{2\gamma^2})$.
\end{itemize}
Moreover, we have $\sqrt{2}\sqrt{1 + \gamma^2 + \frac{1}{4 \gamma^2}} \geq 2$ when $\gamma \geq \frac{1}{\sqrt{2}}$.
%Moreover, the bottlenecks of $Y$ are the pairs $(y_\theta, y_{\theta'})$ with $\theta-\theta' = \pi$, and additionally $\theta-\theta' = \pm \arccos(-\frac{1}{2\gamma^2})$ if $\gamma \geq \frac{1}{\sqrt{2}}$.
\end{lemma}

\begin{proof}
Let $\theta, \theta' \in [0, 2\pi)$.
One computes that
\begin{equation*}
\eucN{y_{\theta} -  y_{\theta'}}^2
%= 2 \left( 2 \sin^2\left(\frac{\theta-\theta'}{2}\right) + \gamma^2 \sin^2(\theta-\theta') \right).
= 4 \sin^2\left(\frac{\theta-\theta'}{2}\right) + 2 \gamma^2 \sin^2(\theta-\theta').
\end{equation*}
Consider the map $f\colon x \in [0, 2\pi) \mapsto  4 \sin^2\left(\frac{x}{2}\right) + 2 \gamma^2 \sin^2(x)$.
Its derivative is
\begin{align*}
f'(x) 
&= 4 \cos\left(\frac{x}{2}\right) \sin\left(\frac{x}{2}\right) + 4 \gamma^2 \cos(x) \sin(x) \\
&= 2 \sin(x)\left(1+2\gamma^2 \cos(x)\right).
\end{align*}
It vanishes when $x = 0$, $x = \pi$, or $x = \pm \arccos(-\frac{1}{2\gamma^2})$ if $\gamma \geq \frac{1}{\sqrt{2}}$.
To conclude, a computation shows that $f(0) = 0$, $f(\pi) = 4$ and $f\left(\pm \arccos\left(-\frac{1}{2\gamma^2}\right) \right) = 2\left( 1 + \gamma^2 + \frac{1}{4\gamma^2}\right)$.  
\end{proof}

\begin{lemma}
\label{lem:torusknot_distance}
For every $x \in E$ such that $x \neq 0$, the map $\theta \mapsto \eucN{x-y_\theta}$ admits at most two local maxima and two local minima.
\end{lemma}

\begin{proof}
Consider the map $g\colon \theta \in [0, 2\pi) \mapsto \eucN{x-y_\theta}^2$. It can be written as
\begin{align*}
g(\theta) 
&= \eucN{x}^2 +  \eucN{y_\theta}^2 - 2\eucP{x}{y_\theta} \\
 &= \eucN{x}^2 +  1 + \frac{\gamma^2}{2} - 2\eucP{x}{y_\theta}.
\end{align*}
Let us show that its derivative $g'$ vanishes at most four times on $[0,2\pi)$, which will prove the result. 
Using the expression of $y_\theta$, we see that $g'$ can be written as
\begin{align*}
g'(\theta)
= a \cos(\theta) + b \sin(\theta)+ c \cos(2\theta)+ d \sin (2\theta),
\end{align*}
where $a,b,c,d \in \R$ are not all zero.
Denoting $\omega = \cos(\theta)$  and $\xi = \sin(\theta)$, we have $\xi^2 = 1-\omega^2$, $\cos(2 \theta) = \cos^2(\theta) - \sin^2(\theta) = 2\omega^2 -1$ and $\sin(2\theta) = 2 \cos(\theta) \sin(\theta) = 2\omega\xi$.
Hence
\begin{align*}
g'(\theta)
= a \omega + b \xi + 2 c \omega^2  + 2d \omega \xi.
\end{align*}
Now, if $g'(\theta) = 0$, we get 
\begin{equation}
\label{eq:torusknot_distance}
a \omega + 2 c \omega^2 = -( b  + 2d \omega) \xi
\end{equation}
Squaring this equality yields
$\left(a \omega + 2 c \omega^2\right)^2
= \left( b   + 2d \omega \right)^2 (1-\omega^2)$.
This degree four equation, with variable $\omega$, admits at most four roots.
To each of these $w$, there exists a unique $\xi = \pm \sqrt{1-w^2}$ that satisfies Equation \eqref{eq:torusknot_distance}.
In other words, the corresponding $\theta \in [0,2\pi)$ such that $\omega = \cos(\theta)$ is unique.
We deduce that $g'$ vanishes at most four times on $[0,2\pi)$.
\end{proof}

\medbreak
Before studying the \v{C}ech filtration of $Y$, let us describe some geometric quantities associated to it.
Using a symbolic computation software, we see that the curvature of $Y$ is constant and equal to
\begin{align*}
\rho = \frac{\sqrt{1+8\gamma^2}}{1+2\gamma^2}.
\end{align*}
In particular, we have $\rho \geq 1$ if $\gamma \leq 1$, and $\rho < 1$ if $\gamma > 1$.
We also have an expression for the diameter of $Y$:
\begin{align*}
\frac{1}{2}\diam{Y} 
= \left\{
    \begin{array}{ll}
        1 & \mbox{ if } \gamma \leq \frac{1}{\sqrt{2}}, \\
        \frac{1}{\sqrt{2}} \sqrt{1 + \gamma^2 + \frac{1}{4 \gamma^2}} & \mbox{ if } \gamma \geq \frac{1}{\sqrt{2}}.
    \end{array}
\right.
\end{align*}
It is a consequence of Lemma \ref{lem:torusknot_diameter}.
We now describe the reach of $Y$:
\begin{equation}
\label{eq:reach_Y}
\reach{Y}
= \left\{
    \begin{array}{ll}
        \frac{1+2\gamma^2}{\sqrt{1+8\gamma^2}} & \mbox{ if } \gamma \leq 1, \\
        1 & \mbox{ if } \gamma \geq 1.
    \end{array}
\right.
\end{equation}
To prove this, we first define a \emph{bottleneck} of $Y$ as pair of distinct points $(y,y') \in Y^2$ such that the open ball $\openball{\frac{1}{2}(y+y')}{\frac{1}{2}\eucN{y-y'}}$ does not intersect $Y$. Its \emph{length} is defined as $\frac{1}{2}\eucN{y-y'}$.
According to the results of \citet[Theorem 3.4]{aamari:hal-01521955}, the reach of $Y$ is equal to
\begin{align*}
\reach{Y} = \min\left\{\frac{1}{\rho}, \delta\right\},
\end{align*}
where $\frac{1}{\rho}$ is the inverse curvature of $Y$, and $\delta$ is the minimal length of bottlenecks of $Y$.
As we computed, $\frac{1}{\rho}$ is equal to $\frac{1+2\gamma^2}{\sqrt{1+8\gamma^2}} $.
Besides, according to Lemma \ref{lem:torusknot_diameter}, a bottleneck $(y_\theta, y_{\theta'})$ has to satisfy $\theta'-\theta = \pi$ or $\pm\arccos(-\frac{1}{2\gamma^2})$. The smallest length is attained when $\theta'-\theta = \pi$, for which $\frac{1}{2}\eucN{y_\theta - y_{\theta'}}=1$. It is straightforward to verify that the pair $(y_\theta, y_{\theta'})$  is indeed a bottleneck.
Therefore we have $\delta=1$, and we deduce the expression of $\reach{Y}$.

Last, the weak feature size of $Y$ does not depend on $\gamma$ and is equal to 1:
\begin{equation}
\label{eq:torusknot_wfs}
\wfs{Y} = 1.
\end{equation}
We shall prove it by using the characterization of \citet{boissonnat2018geometric}: $\wfs{Y}$ is the infimum of distances $\dist{x}{Y}$, where $x \in E$ is a critical point of the distance function $d_Y$. 
In this context, $x$ is a critical point if it lies in the convex hull of its projections on $Y$. 
Remark that, if $x \neq 0$, then $x$ admits at most two projections on $Y$. This follows from Lemma \ref{lem:torusknot_distance}.
As a consequence, if $x$ is a critical point, then there exists $y, y' \in Y$ such that $x$ lies in the middle of the segment $[y,y']$, and the open ball $\openball{x}{\dist{x}{Y}}$ does not intersect $Y$. Therefore $y'$ is a critical point of $y' \mapsto \eucN{y-y'}$, hence Lemma \ref{lem:torusknot_diameter} gives that $\eucN{y-y'} \geq 2$. We deduce the result.

\medbreak
We now describe the thickenings $Y^t$. They present four different behaviours:
\begin{itemize}
\item $0\leq t<1$: $Y^t$ is homotopy equivalent to a circle,
\item $1 \leq t < \frac{1}{2}\diam{Y}$: $Y^t$ is homotopy equivalent to a circle,
\item $ \frac{1}{2}\diam{Y} \leq t < \sqrt{1+\frac{\gamma^2}{2}}$: $Y^t$ is homotopy equivalent to a 3-sphere,
\item $ t \geq \sqrt{1+\frac{\gamma^2}{2}}$: $Y^t$ is homotopy equivalent to a point.
\end{itemize}
Recall that, in the case where $\gamma \leq \frac{1}{\sqrt{2}}$, we have $\frac{1}{2}\diam{Y} = 1$.
Consequently, the interval $\left[1,\frac{1}{2}\diam{Y}\right)$ is empty, and the second point does not appear in this case.

\medbreak\noindent
\emph{Study of the case $0\leq t<1$.}
For $t \in [0,1)$, let us show that $Y^t$ deform retracts on $Y$. 
According to Equation \eqref{eq:torusknot_wfs}, we have $\wfs{Y} = 1$. 
Moreover, Equation \eqref{eq:reach_Y} gives that $\reach{Y} > 0$.
Using the results of \citet{boissonnat2018geometric}, we deduce that $Y^t$ is isotopic to $Y$.

\medbreak\noindent
\emph{Study of the case  $1 \leq t < \frac{1}{2}\diam{Y}$.}
Denote $z_\theta = \left(0,0,\frac{\gamma}{\sqrt{2}}\cos(2\theta), \frac{\gamma}{\sqrt{2}}\sin(2\theta)\right)$, and define the circle $Z = \left\{ z_\theta, \theta \in [0,\pi) \right\}$.
It is repredented in Figure \ref{fig:18}.

\begin{figure}[H]
\centering
\includegraphics[width=.25\linewidth]{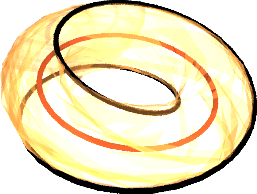}
\caption{Representation of the set $Y$ (black) and the circle $Z$ (red).}
\label{fig:18}
\end{figure}

We claim that $Y^t$ deform retracts on $Z$.
To prove so, we shall define a continuous application $f\colon Y^t \rightarrow Z$ such that, for every $y\in Y^t$, the segment $[y, f(y)]$ is included in $Y^t$. 
This would lead to a deformation retraction of $Y^t$ onto $Z$, via
\begin{align*}
(s,y) \in [0,1]\times Y^t \mapsto (1-s)y + s f(y).
\end{align*}
Equivalently, we shall define an application $\Theta\colon Y^t \rightarrow [0,\pi)$ such that the segment $[y, z_{\Theta(y)}]$ is included in $Y^t$.

Let $y \in Y^t$. 
According to Lemma \ref{lem:torusknot_distance}, $y$ admits at most two projection on $Y$.
We start with the case where $y$ admits only one projection, namely $y_{\theta}$ with $\theta \in [0, 2\pi)$. Let $\overline{\theta} \in [0, \pi)$ be the reduction of $\theta$ modulo $\pi$, and consider the point $z_{\overline{\theta}}$ of $Z$. 
A computation shows that the distance $\eucN{y_\theta - z_{\overline{\theta}}}$ is equal to 1. Besides, since $y \in Y^t$, the distance $\eucN{y_\theta - y}$ is at most $t$. By convexity, the segment $\left[y, z_{\overline{\theta}}\right]$ is included in the ball $\closedball{y_\theta}{t}$, which is a subset of $Y^t$.
We then define $\Theta(y)= \overline{\theta}$.

Now suppose that $y$ admits exactly two projection $y_\theta$ and $y_{\theta'}$. According to Lemma \ref{lem:torusknot_diameter}, these angles must satisfy $\theta'-\theta = \pi$. 
Indeed, the case $\eucN{y_{\theta} -  y_{\theta'}} = \sqrt{2}\sqrt{1 + \gamma^2 + \frac{1}{4 \gamma^2}}$ does not occur since we chose $t < \frac{1}{2}\diam{Y} = \frac{\sqrt{2}}{2}\sqrt{1 + \gamma^2 + \frac{1}{4 \gamma^2}}$.
The angles $\theta$ and $\theta'$ correspond to the same reduction modulo $\pi$, denoted $\overline{\theta}$, and we also define $\Theta(y)= \overline{\theta}$.

\medbreak\noindent
\emph{Study of the case  $t\in \left[ \frac{1}{2}\diam{X}, \sqrt{1+\frac{\gamma^2}{2}}\right)$.}
Let $\S_3$ denotes the unit sphere of $E$.
For every $v = (v_1,v_2,v_3,v_4) \in \S_3$, we shall denote by $\langle v \rangle$ the linear subspace spanned by $v$, and by $\langle v \rangle_+$ the cone $\{\lambda v, \lambda \geq 0\}$.
Moreover, we define the quantity
\begin{align*}
\delta(v) = \min_{y \in Y} \dist{y}{\langle v \rangle_+}.
\end{align*}
and the set
\begin{align*}
S = \left\{ \delta(v) v, v \in \S_3  \right\}.
\end{align*}
The situation is depicted in Figure \ref{fig:19}.
We claim that $S$ is a subset of $Y^t$, and that $Y^t$ deform retracts on it.
This follows from the two following facts: for every $v \in \S_3$, 
\begin{enumerate}
\item $\delta(v)$ is not greater than $\frac{1}{2}\diam{Y}$,
\item $\langle v \rangle_+ \cap Y^t$ consists of one connected component: an interval centered on $\delta(v)v$, that does not contain the point 0.
\end{enumerate}
Suppose that these assertions are true.
Then one defines a deformation retraction of $Y^t$ on $S$ by retracting each fiber $\langle v \rangle_+ \cap Y^t$ linearly on the singleton $\{ \delta(v)v \}$.
We shall now prove the two items.
\begin{figure}[H]
\centering
\includegraphics[width=.3\linewidth]{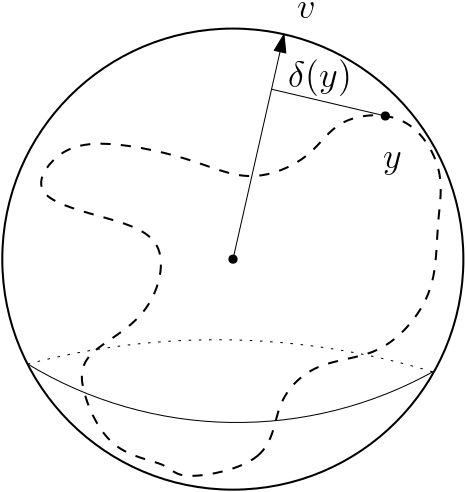}
\caption{Representation of the set $Y$ (dashed), lying on a 3-sphere of radius $\sqrt{1+\frac{\gamma^2}{2}}$.}
\label{fig:19}
\end{figure}

\noindent
\textit{Item 1.} %\newline \noindent
Note that Item 1 can be reformulated as follows:
\begin{equation}
\label{eq:torusknot_maxmin_plus}
\max_{v \in \S_3} \min_{y \in Y} \dist{y}{\langle v \rangle_+}
\leq \frac{1}{2} \diam{Y}.
\end{equation}
Let us justify that the pairs $(v,y)$ that attain this maximum-minimum are the same as in
\begin{equation}
\label{eq:torusknot_maxmin_norm}
\max_{v \in \S_3} \min_{y \in Y} \eucN{y-v}.
\end{equation}
From the definition of $Y = \{y_\theta, \theta \in [0,2\pi)\}$, we see that $\min_{y \in Y} \dist{y}{\langle v \rangle_+} = \min_{y \in Y} \dist{y}{\langle v \rangle}$.
A vector $v \in \S_3$ being fixed, let us show that
$y\mapsto \dist{y}{\langle v \rangle}$ is minimized when $y \mapsto \eucN{v-y}$ is.
Let $y\in Y$.
Since $v$ is a unit vector, the projection of $y$ on $\langle v \rangle$ can be written as $\eucP{y}{v}v$. 
Hence $\dist{y}{\langle v \rangle}^2 = \eucN{\eucP{y}{v}v - y}^2$, and expanding this norm yields 
\begin{align*}
\dist{y}{\langle v \rangle}^2 = \eucN{y}^2-\eucP{y}{v}^2.
\end{align*}
Expanding the norm $\eucN{y-v}^2$ and
using that $\eucN{y}^2 = 1 + \frac{\gamma^2}{2}$, we get $\eucP{y}{v} = \frac{1}{2}\left( 2 + \frac{\gamma^2}{2} - \eucN{y-v}^2\right)$.
We inject this relation in the preceding equation to obtain
\begin{align*}
\dist{y}{\langle v \rangle}^2 
= -\left(\frac{\gamma}{2}\right)^4+\gamma^2 + \frac{1}{4}\eucN{y-v}^2\left(4+\gamma^2 -\eucN{y-v}^2\right).
%= \frac{\gamma^2}{4}+\frac{1}{2}\eucN{y-v}^2.
\end{align*}
Now we can deduce that $y \mapsto \dist{y}{\langle v \rangle}^2$ is minimized when $y \mapsto \eucN{y-v}$ is minimized.
Indeed, the map $\eucN{y-v}\mapsto \frac{1}{4}\eucN{y-v}^2\left(4+\gamma^2 -\eucN{y-v}^2\right)$ is increasing on $\left[0, \frac{1}{2}(4+\gamma^2)\right]$. 
But $\eucN{y-v} \leq \eucN{y} + \eucN{v} = \frac{1}{2}(4+\gamma^2)$.

We deduce that studying the left hand term of Equation \eqref{eq:torusknot_maxmin_plus} is equivalent to studying Equation \eqref{eq:torusknot_maxmin_norm}.
We shall denote by $g\colon \S_3 \rightarrow \R$ the map
\begin{equation}
\label{eq:torusknot_g}
%g(v) = \min_{y \in Y} \dist{y}{\langle v \rangle}.
g(v) = \min_{y \in Y} \eucN{y-v}.
\end{equation}

Let $v \in \S_3$ that attains the maximum of $g$, and let $y$ be a corresponding point that attains the minimum of $\eucN{y-v}$.
The points $v$ and $y$ attains the quantity in Equation \eqref{eq:torusknot_maxmin_plus}.
In order to prove that $\dist{y}{\langle v \rangle} \leq \frac{1}{2}\diam{Y}$, let $p(y)$ denotes the projection of $y$ on $\langle v \rangle$.
We shall show that there exists another point $y' \in Y$ such that $p(y)$ is equal to $\frac{1}{2}(y+y')$
Consequently, we would have $\eucN{y - p(y)} = \frac{1}{2}\eucN{y' - y} \leq \frac{1}{2}\diam{Y}$, i.e.
\begin{align*}
\dist{y}{\langle v \rangle} \leq \frac{1}{2}\diam{Y}.
\end{align*}

Remark the following fact: if $w \in \S_3$ is a unit vector such that $\eucP{p(y)-y}{w} >0$, then for $\epsilon>0$ small enough, we have 
\begin{align*}
\dist{y}{\langle v+\epsilon w \rangle} > \dist{y}{\langle v \rangle}. 
\end{align*}
Equivalently, this statement reformulates as $0 \leq \eucP{y}{\frac{1}{\eucN{v+\epsilon w}}(v+\epsilon w)} < \eucP{y}{v}$.
Let us show that
\begin{equation}
\label{eq:torusknot_eucP}
\eucP{y}{\frac{1}{\eucN{v+\epsilon w}}(v+\epsilon w)} 
= \eucP{y}{v} - \epsilon \kappa + \petito{\epsilon},
\end{equation}
where $\kappa = \eucP{p(y)-y}{w} > 0$, and where $\petito{\epsilon}$ is the little-o notation.
Note that $\frac{1}{\eucN{v+\epsilon w}} = 1-\epsilon \eucP{v}{w} + \petito{\epsilon}$. 
We also have
\begin{align*}
\frac{1}{\eucN{v+\epsilon w}}(v+\epsilon w)
&= v + \epsilon \left(w-\eucP{v}{w}v \right) + \petito{\epsilon}.
%&= v + \epsilon(w-p(y)) + \petito{\epsilon}.
\end{align*}
Expanding the inner product in Equation \eqref{eq:torusknot_eucP} gives
\begin{align*}
\eucP{y}{\frac{1}{\eucN{v+\epsilon w}}(v+\epsilon w)} 
&= \eucP{y}{v} +  \epsilon \bigg( \eucP{y}{w} - \eucP{v}{w} \eucP{y}{v} \bigg)  + \petito{\epsilon} \\
&= \eucP{y}{v} + \epsilon \bigg \langle y-\eucP{y}{v}v , w \bigg \rangle +\petito{\epsilon} \\
&= \eucP{y}{v} + \epsilon \eucP{ y-p(y) }{w} +\petito{\epsilon},
\end{align*}
and we obtain the result.

Next, let us prove that $y$ is not the only point of $Y$ that attains the minimum in Equation \eqref{eq:torusknot_g}.
Suppose that it is the case by contradiction.
Let $w \in \S_3$ be a unit vector such that $\eucP{p(y)-y}{w} >0$.
For $\epsilon$ small enough, let us prove that the vector $v' = \frac{1}{\eucN{v+\epsilon w}}(v+\epsilon w)$ of $\S_3$ contradicts the maximality of $v$.
That is, let us prove that $g(v') > g(v)$.
Let $y' \in Y$ be a minimizer $\eucN{y'-v'}$.
We have to show that $\eucN{y'-v'} > \eucN{y-v}$.
This would lead to $g(v') > g(v)$, hence the contradiction.

Expanding the norm yields
\begin{align*}
\eucN{v'-y'}^2
= \eucN{v'-v+v-y'}^2
&\geq \eucN{v'-v}^2 + \eucN{v-y'}^2 -2\eucP{v'-v}{v-y'}.
\end{align*}
Using $\eucN{v'-v}^2 \geq 0$ and $\eucN{v-y'}^2 \geq \eucN{v-y}^2$ by definition of $y$, we obtain
\begin{align*}
\eucN{v'-y'}^2
\geq \eucN{v-y}^2 -2\eucP{v'-v}{v-y'}.
\end{align*}
We have to show that $\eucP{v'-v}{y-y'}$ is positive for $\epsilon$ small enough.
By writing $v-y' = v-y+(y-y')$ we get
\begin{align*}
\eucP{v'-v}{v-y'}
&= \eucP{v'-v}{v} - \eucP{v'-v}{y} + \eucP{v'-v}{y-y'}
\end{align*}
According to Equation \eqref{eq:torusknot_eucP}, $- \eucP{v'-v}{y} = \epsilon \kappa + \petito{\epsilon}$.
Besides, using $v' - v = \epsilon(w-\eucP{v}{w}v) + \petito{\epsilon}$, we get $\eucP{v'-v}{v} = \petito{\epsilon}$.
Last, Cauchy-Schwarz inequality gives $|\eucP{v'-v}{y-y'}| \leq \eucN{v'-v}\eucN{y-y'}$. Therefore, $\eucP{v'-v}{y-y'} = \grando{\epsilon}\eucN{y-y'}$, where $\grando{\epsilon}$ is the big-o notation.
Gathering these three equalities, we obtain
\begin{align*}
\eucP{v'-v}{v-y'}
&= \petito{\epsilon} + \epsilon \kappa + \grando{\epsilon}\eucN{y-y'}.
\end{align*}
As we can read from this equation, if $\eucN{y-y'}$ goes to zero as $\epsilon$ does, then $\eucP{v'-v}{v-y'}$ is positive for $\epsilon$ small enough.
Observe that $v'$ goes to $v$ when $\epsilon$ goes to 0.
By assumption $y$ is the only minimizer in Equation \eqref{eq:torusknot_g}. By continuity of $g$, we deduce that $y'$ goes to $y$.

By contradiction, we deduce that there exists another point $y'$ which attains the minimum in $g(v)$.
Note that it is the only other one, according to Lemma \ref{lem:torusknot_distance}.
Let us show that $p(y)$ lies in the middle of the segment $[y,y']$.
Suppose that it is not the case. 
Then $p(y)-y$ is not equal to $-(p(y')-y')$, where $p(y')$ denotes the projection of $y'$ on $\langle v \rangle$.
Consequently, the half-spaces $\{w \in E, \eucP{p(y)-y}{w} > 0\}$ and $\{w \in E, \eucP{p(y')-y'}{w} > 0\}$ intersects. Let $w$ be any vector in the intersection.
For $\epsilon>0$, denote $v' = \frac{1}{\eucN{v+\epsilon w}}(1+\epsilon w)$.
If $\epsilon$ is small enough, the same reasoning as before shows that $v'$ contradicts the maximality of $v$.
The situation is represented in Figure \ref{fig:20}.

\begin{figure}[H]
\centering
\begin{minipage}{.49\linewidth}
\centering
\includegraphics[width=.6\linewidth]{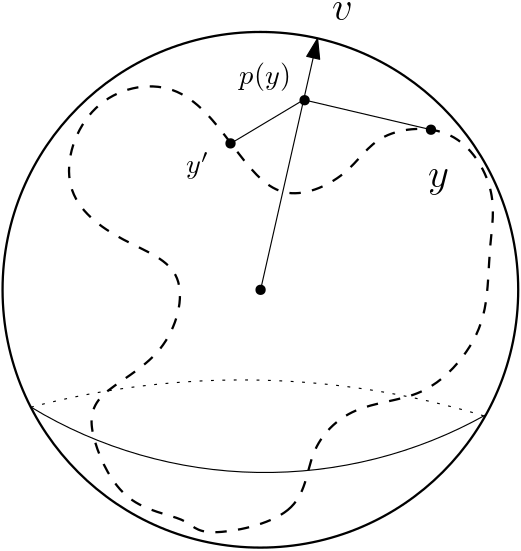}
\end{minipage}
\begin{minipage}{.49\linewidth}
\centering
\includegraphics[width=.7\linewidth]{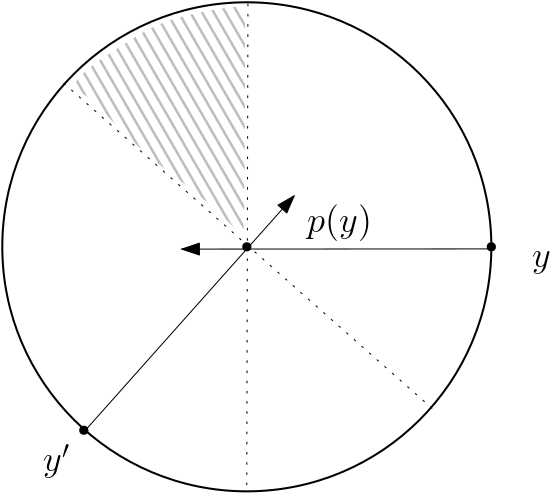}
\end{minipage}
\caption{Left: Representation of the situation where $y$ and $y'$ are minimizers of Equation \eqref{eq:torusknot_g}.
Right: Representation in the plane passing through the points $y$, $y'$ and $p(y)$. The dashed area corresponds to the intersection of the half-spaces $\{w \in E, \eucP{p(y)-y}{w} > 0\}$ and $\{w \in E, \eucP{p(y')-y'}{w} > 0\}$.}
\label{fig:20}
\end{figure}

\medbreak \noindent
\textit{Item 2.} 
Let $v \in \S_3$. The set $\langle v \rangle_+ \cap Y^t$ can be described as
\begin{align*}
\langle v \rangle_+ \cap \bigcup_{y \in Y} \closedball{y}{t}.
\end{align*}
Let $y \in Y$ such that $\langle v \rangle_+ \cap \closedball{y}{t} \neq \emptyset$. Denote by $p(y)$ the projection of $y$ on $\langle v \rangle_+$. 
It is equal to $\eucP{y}{v}v$.
Using Pythagoras' theorem, we obtain that the set $\langle v \rangle_+ \cap \closedball{y}{t}$ is equal to the interval 
\begin{align*}
\left[ p(y) \pm \sqrt{t^2 - \dist{y}{\langle v \rangle}^2} v \right].
\end{align*}
Using the identity $\dist{y}{\langle v \rangle}^2 = \eucN{y} - \eucP{y}{v}^2 = 1 + \frac{\gamma^2}{2} - \eucP{y}{v}^2$, we can write this interval as
\begin{align*}
\big[I_1(y) \cdot v, ~ I_2(y) \cdot v\big],
\end{align*}
where 
$I_1(y) = \eucP{y}{v} - \sqrt{ \eucP{y}{v}^2 -  (1 +  \frac{\gamma^2}{2} -t^2)}$ and
$I_2(y) = \eucP{y}{v} + \sqrt{ \eucP{y}{v}^2 -  (1 +  \frac{\gamma^2}{2} -t^2)}$.
Seen as functions of $\eucP{y}{v}$, the map $I_1$ is decreasing, and the map $I_2$ is increasing (see Figure \ref{fig:21}).
Let $y^* \in Y$ that minimizes $\dist{y}{\langle v \rangle}$. Equivalently, $y^*$ maximizes $\langle y,v\rangle$. It follows that the corresponding interval $\big[I_1(y^*) \cdot v, ~ I_2(y^*) \cdot v\big]$ contains all the others.
We deduce that the set $\langle v \rangle_+ \cap Y^t$ is equal to this interval.

\begin{figure}[H]
\centering
\begin{minipage}{.44\linewidth}
\centering
\includegraphics[width=.65\linewidth]{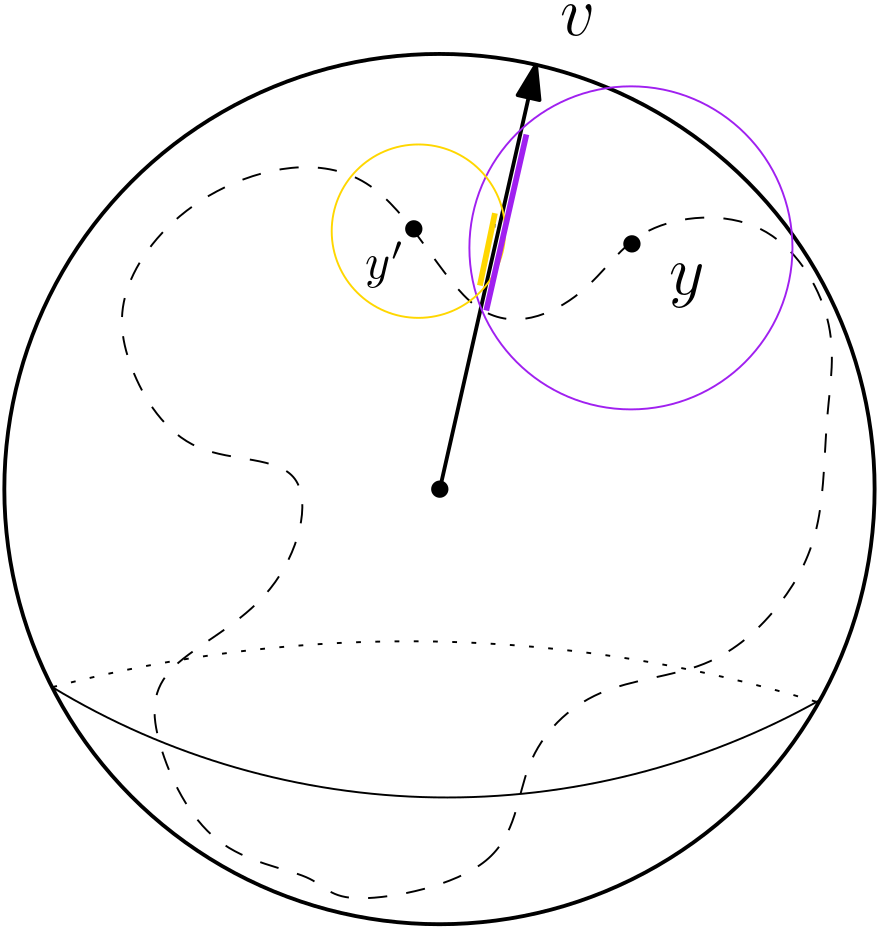}
\end{minipage}
\begin{minipage}{.54\linewidth}
\centering
\includegraphics[width=.85\linewidth]{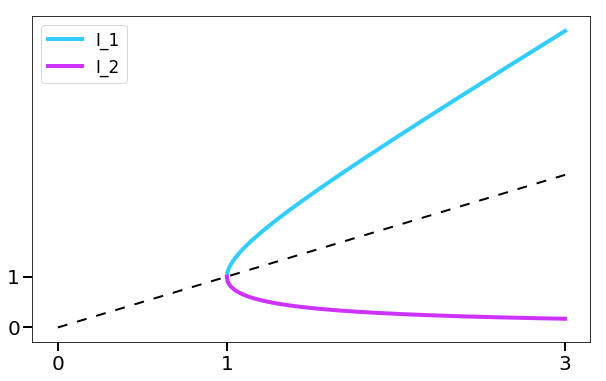}
\end{minipage}
\caption{Left: Representation of two intervals $\langle v \rangle_+ \cup \closedball{y}{t}$  and $\langle v \rangle_+ \cup \closedball{y'}{t}$. 
Right: Representation of the maps $x\mapsto x \pm \sqrt{x^2-1}$.}
\label{fig:21}
\end{figure}

\medbreak \noindent
\emph{Study of the case $t \geq \sqrt{1+\frac{1}{2}\gamma^2}$.}
For every $y \in Y$, we have $\eucN{y} = \sqrt{1+\frac{1}{2}\gamma^2}$.
Therefore, if $t \geq \sqrt{1+\frac{1}{2}\gamma^2}$, then $Y^t$ is star shaped around the point 0, hence it deform retracts on it.

\bigbreak
To close this subsection, let us study the \v{C}ech bundle filtration $(\Y, \p)$ of $Y$.
According to Equation \eqref{eq:tmax_subset_grass}, its filtration maximal value is $\tmax{Y} = \tmaxgamma{X} = \frac{\gamma}{\sqrt{2}}$.
Note that $\frac{\gamma}{\sqrt{2}}$ is lower than $\frac{1}{2}\diam{Y}$.
Consequently, only two cases are to be studied: $t \in [0,1)$, and $t \in \left[1, \frac{1}{2}\diam{Y} \right)$.

The same argument as in Subsect. \ref{subsec:gamma_normal} yields that for every $t \in [0,1)$, the persistent Stiefel-Whitney class $w_1^t(Y)$ is equal to $w_1^0(Y)$.
Accordingly, for every $t \in \left[1, \frac{1}{2}\diam{Y} \right)$, the class $w_1^t(Y)$ is equal to $w_1^1(Y)$.
Let us show that $w_1^0(Y)$ is zero, and that $w_1^1(Y)$ is not.

First, remark that the map $p^0\colon Y \rightarrow \Grass{1}{\R^2}$ can be seen as the normal bundle of the circle. 
Hence $(p^0)^*\colon H^*(Y) \leftarrow H^*(\Grass{1}{\R^2})$ is nontrivial, and we deduce that $w_1^0(Y)=0$.
As a consequence, the persistent Stiefel-Whitney class $w_1^t(X)$ is nonzero for every $t < 1$.

Next, consider $p^1\colon Y^1 \rightarrow \Grass{1}{\R^2}$.
Recall that $Y^1$ deform retracts on the circle
\begin{align*}
Z = \left\{ \left(0,0,\frac{\gamma}{\sqrt{2}}\cos(2\theta), \frac{\gamma}{\sqrt{2}}\sin(2\theta)\right), \theta \in [0,\pi) \right\}.
\end{align*}
Seen in $\R^2 \times \matrixspace{\R^2}$, we have
\begin{align*}
Z = \bigg\{
\bigg( 
\begin{pmatrix}
0  \\
0 
\end{pmatrix}
,
\gamma \begin{pmatrix}
\cos(\theta)^2 & \cos(\theta) \sin(\theta) \\
\cos(\theta) \sin(\theta) & \sin(\theta)^2 
\end{pmatrix}
\bigg),
\theta \in [0, \pi)  \bigg\}.
\end{align*}
Notice that the map $q\colon Z \rightarrow \Grass{1}{\R^2}$, the projection on $\Grass{1}{\R^2}$, is injective. Seen as a map between two circles, it has degree (modulo 2) equal to 1.
We deduce that $q^*\colon H^*(Z) \leftarrow H^*(\Grass{1}{\R^2})$ is nontrivial.
Now, remark that the map $q$ factorizes through $p^1$:
\begin{center}
\begin{tikzcd}%[row sep=tiny]
Z \arrow[dr, "q", swap] \arrow[rr, hook]&  & Y^{1} \arrow[dl, "p^{1}"] \\
& \Grass{1}{\R^2} &
\end{tikzcd}
\end{center}
It induces the following diagram in cohomology:
\begin{center}
\begin{tikzcd}%[column sep=tiny]
H^*(Z) &  & H^*(Y^{1}) \arrow[ll, "\sim", swap]  \\
& H^*(\Grass{1}{\R^2}) \arrow[ul, "q^*"] \arrow[ur, "(p^{1})^*", swap] &
\end{tikzcd}
\end{center}
Since $q^*$ is nontrivial, this commutative diagram yields that the persistent Stiefel-Whitney class $w_1^1(Y)$ is nonzero.
As a consequence, the persistent Stiefel-Whitney class $w_1^t(Y)$ is nonzero for every $t \geq 1$.

\bibliographystyle{spbasic}      % basic style, author-year citations
\bibliography{biblio_JACT}   % name your BibTeX data base

\begin{thebibliography}{16}
\providecommand{\natexlab}[1]{#1}
\providecommand{\url}[1]{{#1}}
\providecommand{\urlprefix}{URL }
\expandafter\ifx\csname urlstyle\endcsname\relax
  \providecommand{\doi}[1]{DOI~\discretionary{}{}{}#1}\else
  \providecommand{\doi}{DOI~\discretionary{}{}{}\begingroup
  \urlstyle{rm}\Url}\fi
\providecommand{\eprint}[2][]{\url{#2}}

\bibitem[{Aamari et~al.(2019)Aamari, Kim, Chazal, Michel, Rinaldo, and
  Wasserman}]{aamari:hal-01521955}
Aamari E, Kim J, Chazal F, Michel B, Rinaldo A, Wasserman L (2019) {Estimating
  the Reach of a Manifold}. {Electronic journal of statistics}

\bibitem[{Aubrey(2011)}]{aubrey2011persistent}
Aubrey H (2011) Persistent cohomology operations. PhD thesis, Duke University

\bibitem[{Bauer and Edelsbrunner(2017)}]{bauer2017morse}
Bauer U, Edelsbrunner H (2017) The {M}orse theory of \v{C}ech and {D}elaunay
  complexes. Transactions of the American Mathematical Society
  369(5):3741--3762

\bibitem[{Bell et~al.(2019)Bell, Lawson, Martin, Rudzinski, and
  Smyth}]{bell2017weighted}
Bell G, Lawson A, Martin J, Rudzinski J, Smyth C (2019) Weighted persistent
  homology. Involve, a Journal of Mathematics 12(5):823--837

\bibitem[{Boissonnat et~al.(2018)Boissonnat, Chazal, and
  Yvinec}]{boissonnat2018geometric}
Boissonnat JD, Chazal F, Yvinec M (2018) Geometric and topological inference,
  vol~57. Cambridge University Press

\bibitem[{Botnan and Crawley-Boevey(2020)}]{botnan2020decomposition}
Botnan M, Crawley-Boevey W (2020) Decomposition of persistence modules.
  Proceedings of the American Mathematical Society 148(11):4581--4596

\bibitem[{Chazal et~al.(2009)Chazal, Cohen-Steiner, and
  Lieutier}]{chazal2009sampling}
Chazal F, Cohen-Steiner D, Lieutier A (2009) A sampling theory for compact sets
  in {E}uclidean space. Discrete \& Computational Geometry 41(3):461--479

\bibitem[{Chazal et~al.(2016)Chazal, de~Silva, Glisse, and
  Oudot}]{Chazal_Persistencemodules}
Chazal F, de~Silva V, Glisse M, Oudot S (2016) The Structure and Stability of
  Persistence Modules. SpringerBriefs in Mathematics

\bibitem[{Edelsbrunner(1993)}]{edelsbrunner1993union}
Edelsbrunner H (1993) The union of balls and its dual shape. In: Proceedings of
  the ninth annual symposium on Computational geometry, pp 218--231

\bibitem[{Govc et~al.(2020)Govc, Marzantowicz, Pave{\v{s}}i{\'c}
  et~al.}]{govc2020many}
Govc D, Marzantowicz W, Pave{\v{s}}i{\'c} P, et~al. (2020) How many simplices
  are needed to triangulate a {G}rassmannian? Topological Methods in Nonlinear
  Analysis

\bibitem[{Hatcher(2002)}]{Hatcher_Algebraic}
Hatcher A (2002) Algebraic Topology. Cambridge University Press

\bibitem[{von K{\"u}hnel(1987)}]{von1987minimal}
von K{\"u}hnel W (1987) Minimal triangulations of kummer varieties. In:
  Abhandlungen aus dem Mathematischen Seminar der Universit{\"a}t Hamburg,
  Springer, vol~57, pp 7--20

\bibitem[{Milnor and Stasheff(2016)}]{Milnor_Characteristic}
Milnor J, Stasheff JD (2016) Characteristic Classes.(AM-76), vol~76. Princeton
  university press

\bibitem[{Munkres(1984)}]{Munkres84}
Munkres JR (1984) Elements of Algebraic Topology. Addison-Wesley

\bibitem[{Perea(2018)}]{perea2018multiscale}
Perea JA (2018) Multiscale projective coordinates via persistent cohomology of
  sparse filtrations. Discrete \& Computational Geometry 59(1):175--225

\bibitem[{Scoccola and Perea(2021)}]{scoccola2021approximate}
Scoccola L, Perea JA (2021) Approximate and discrete euclidean vector bundles.
  arXiv preprint arXiv:210407563

\end{thebibliography}

%% Non-BibTeX users please use
%\begin{thebibliography}{}
%%
%% and use \bibitem to create references. Consult the Instructions
%% for authors for reference list style.
%\end{thebibliography}

\end{document}